\documentclass[11pt]{article}

\usepackage[style=numeric,maxbibnames=99]{biblatex}
\usepackage{amsmath,amssymb,amsthm,color}
\usepackage{hyperref,url}
\usepackage{fancybox,graphicx} 
\usepackage[title]{appendix}

\newtheorem{theorem}{Theorem}[section]
\newtheorem{lemma}[theorem]{Lemma}
\newtheorem{corollary}[theorem]{Corollary}

\newtheorem{assumption}[theorem]{Assumption}

\newtheorem{remark}[theorem]{Remark}

\newcommand{\XX}{{\cal X}}
\newcommand{\YY}{{\cal Y}}
\newcommand{\ZZ}{{\cal Z}}

\newcommand{\RR}{\mathbb R}
\newcommand{\EE}{\mathbb E}

\newcommand{\Rmnum}[1]{\uppercase\expandafter{\romannumeral #1}} 
\usepackage[raggedright]{titlesec}
\titleformat{\chapter}{\centering\Huge\bfseries}{Chapter \Rmnum{\thechapter} }{1em}{} 

\usepackage{mathrsfs} 
\usepackage{enumerate} 
\graphicspath{{figures/}}

\usepackage{algorithm}
\usepackage{algorithmicx}
\usepackage{multirow}
\usepackage{amsmath}
\usepackage{stmaryrd}
\usepackage{float}
\usepackage{graphicx}
\usepackage{subfigure}
\usepackage{xcolor}
\usepackage{algpseudocode}
\usepackage{array}

\usepackage{color}

\begin{document}
\title{An Accelerated Variance Reduced Extra-Point Approach to Finite-Sum Hemivariational Inequality Problem\footnote{This is the second version of the aiXiv report entitled ``An accelerated variance reduced extra-point approach to finite-sum VI and optimization'' (2211.03269). In this new version, the organization, analysis, as well as the numerical experiments are updated, and we adopt the current new title.}} 

\author{
Kevin Huang\thanks{Institute of Industrial Engineering, National Taiwan University, kdhuang@ntu.edu.tw}
\hspace{1cm}
Nuozhou Wang\thanks{Department of Industrial and System Engineering, University of Minnesota, wang9886@umn.edu}
\hspace{1cm}
Shuzhong Zhang\thanks{Department of Industrial and System Engineering, University of Minnesota, zhangs@umn.edu}
}

\date{\today}
\maketitle

\begin{abstract}
    In this paper, we develop stochastic variance reduced algorithms for solving a class of {\it finite-sum} hemivariational inequality (HVI) problem. In this HVI problem, the associated function is assumed to be differentiable, and both the vector mapping and the function are of finite-sum structure. We propose two algorithms to solve the cases when the vector mapping is either merely monotone or strongly monotone, while the function is assumed to be convex. We show how to apply variance reduction in the proposed algorithms when such an HVI problem 
    has a finite-sum structure, and the resulting accelerated gradient complexities can match the best bound established for finite-sum VI problem, as well as the bound given by the direct Katyusha for finite-sum optimization respectively, in terms of the corresponding parameters such as (gradient) Lipschitz constants and the sizes of the finite-sums. We demonstrate the application of our algorithms through solving a finite-sum constrained finite-sum optimization problem and provide preliminary numerical results. 
    
    \vspace{3mm}

    \noindent\textbf{Keywords:} finite-sum optimization, variance reduction method, variational inequalities, hemivariational inequalities.
\end{abstract}

\section{Introduction}
Let $\ZZ\subseteq \RR^n$ be a closed convex set, $H:\RR^n\rightarrow\RR^n$ and $g:\RR^n\rightarrow\RR$, the so-called hemivariational inequality (HVI) problem is to find $x^*\in\ZZ$ such that
\begin{eqnarray}
    \langle H(x^*),x-x^*\rangle+g(x)-g(x^*)\ge0,\quad\forall x\in\ZZ.\label{HVI-problem}
\end{eqnarray}
In this paper, we are interested in solving the HVI problem~\eqref{HVI-problem} when the mapping $H$ and function $g$ are {\it both} in the {\it finite-sum} form:
\begin{eqnarray}
    H(x):=\sum\limits_{i=1}^{m_1}H_i(x),\quad g(x):=\sum\limits_{i=1}^{m_2}g_i(x).\label{finite-sum H and g}
\end{eqnarray}

\subsection{The HVI problem}
The HVI problem, or sometimes referred to as the variational inequalities (VI) of the second kind~\cite{facchinei2014vi}, can be viewed as a more general class of VI problems when $g(x)=0$ for all $x$. The VI models are particularly powerful for computing equilibria of various types, with applications stemming from economics, applied engineering, and non-cooperative games~\cite{facchinei2007finite}. On the other hand, the HVI formulation, while originally motivated from the infinite-dimensional domain~\cite{panagiotopoulos1993hemivariational, naniewicz1994mathematical, motreanu2013minimax}, has recently received attention under the framework of mathematical programming due to its capability to model and to solve certain classes of saddle point problems and constrained optimization problems more efficiently; see e.g.~\cite{jofre2007variational, monteiro2011complexity}. In a standard setting, $H$ is assumed to be a monotone Lipschitz continuous mapping:
\begin{eqnarray}
    \langle H(x)-H(y),x-y\rangle\ge\mu\|x-y\|^2,\quad \|H(x)-H(y)\|\le L\|x-y\|,\quad \forall x,y\in\ZZ\label{monotone and lipschitz}
\end{eqnarray}
for some $L\ge \mu\ge 0$, while $g$ is assumed to be a proper convex lower semicontinuous function. This allows $g$ to account for the non-smoothness part of the problem, which is more often expressed in the general form of {\it monotone inclusion} problem:
\begin{eqnarray}
    \mbox{find $x^*\in\RR^n$ such that}\,\,0\in (H+\partial g)(x^*),\label{monotone-inclusion-prob}
\end{eqnarray}
where $\partial g(x)$ is the subdifferential set of $g$ at $x$. In formulation~\eqref{monotone-inclusion-prob}, the presence of the constraint $\ZZ$ is unnecessary since it is already implicitly incorporated by having either $g$ as an indicator function with respect to $\ZZ$ or $\partial g(x)$ include the normal cone of $\ZZ$ at $x$. A useful example would be to solve the smooth convex optimization with functional constraints. See the problem discussed in Section~\ref{sec:finite-sum-constrained} as an example. Numerical algorithms for solving~\eqref{monotone-inclusion-prob} often include the prox-mappings with respect to $g$ throughout the iterations, which is assumed to be easily computable.

In this paper, however, we adopt a setting different from the commonly assumed for hemivariational inequality as described above. In particular, while we still assume $H$ to be a monotone Lipschitz continuous mapping, we consider $g$ to be convex {\it differentiable} with Lipschitz continuous gradient. The consequence of this setting is the need to re-introduce the constraint set $\ZZ$, which brings us back to the first problem formulation~\eqref{HVI-problem}. We use $L_h$ and $L_g$ to distinguish between the Lipschitz constants for mapping $H$ and $\nabla g$, that is,
\begin{eqnarray}
    \|H(x)-H(y)\|\le L_h\|x-y\|,\quad \|\nabla g(x)-\nabla g(y)\|\le L_g\|x-y\|,\quad \forall x,y\in\ZZ.\nonumber
\end{eqnarray}
In fact, under our setting, the HVI problem in~\eqref{HVI-problem} can be equivalently reformulated as the following VI problem:
\begin{eqnarray}
    \mbox{find $x^*\in\ZZ$ such that}\,\,\langle H(x^*)+\nabla g(x^*),x-x^*\rangle\ge 0,\quad\forall x\in\ZZ.\label{HVI-prob-2}
\end{eqnarray}
Indeed, a solution to~\eqref{HVI-prob-2} implies a solution to~\eqref{HVI-problem} due to the convexity of $g$. To show the opposite, we have to involve the monotonicity of $H$ as well. Using the fact that
\begin{eqnarray}
    \langle H(x),x-x^*\rangle\ge\langle H(x^*),x-x^*\rangle,\quad g(x)-g(x^*)\le \langle\nabla g(x),x-x^*\rangle\nonumber,
\end{eqnarray}
we have
\begin{eqnarray}
    \langle H(x)+\nabla g(x),x-x^*\rangle\ge \langle H(x^*),x-x^*\rangle+g(x)-g(x^*)\ge0,\quad\forall x\in\ZZ.
\end{eqnarray}
That is, a solution $x^*$ to~\eqref{HVI-problem} is also a {\it Minty solution}~\cite{minty1973monotone} 
(aka a 
{\it weak} solution) to the VI problem associated with mapping $F(x):=H(x)+\nabla g(x)$ and constraint $\ZZ$, satisfying
\[
\langle F(x),x-x^*\rangle\ge0,\quad\forall x\in\ZZ.
\]
Since $F$ is continuous and $\ZZ$ is non-empty, closed and convex, by 
a well-known 
result due to~\cite{minty1973monotone}, every Minty solution $x^*$ is a regular (strong) VI solution satisfying
\begin{eqnarray}
    \langle F(x^*),x-x^*\rangle\ge0,\quad\forall x\in\ZZ,\label{VI-prob-1}
\end{eqnarray}
which is exactly~\eqref{HVI-prob-2}.

Compared to the more common setting in HVI problems where $g$ is only lower semicontinuous, the differentiability of $g$ (and the Lipschitz continuity of the gradient) can bring us positive effects when applying numerical algorithms for solving approximated solutions. While incorporating the prox-mapping of $g$ is more likely than not a ``must'' to an algorithm in the former setting, explicitly exploiting the gradient mapping $\nabla g(x)$ in the algorithm is in fact a ``plus'' in the latter setting. While one can naively solve the HVI problem~\eqref{HVI-problem} through solving the VI problem~\eqref{VI-prob-1} with $F(x):=H(x)+\nabla g(x)$, as we have shown the equivalence between the two, the iteration complexity turns out to be suboptimal even when applying the ``optimal'' first-order methods such as the extra-gradient method~\cite{korpelevich1976extragradient}, optimistic gradient descent ascent method~\cite{popov1980modification, mokhtari2020convergence}, dual extrapolation method~\cite{nesterov2007dual}, among others. The aforementioned optimal first-order methods for solving the general VI problem~\eqref{VI-prob-1} generate $\epsilon$-solutions in at most $\mathcal{O}(1/\epsilon)$ iterations for monotone $H$ and $\mathcal{O}(\kappa\ln(1/\epsilon))$ iterations for strongly monotone $H$ (when there exists $\mu>0$ in~\eqref{monotone and lipschitz}), where $\kappa$ is the condition number defined as $\kappa:=\frac{L}{\mu}$ for $\mu>0$. These iteration complexities have been proven optimal~\cite{zhang2021lower} in terms of $L,\mu,\epsilon$ for solving the general VI problem~\eqref{VI-prob-1}, but if we are faced with the HVI problem~\eqref{HVI-problem} (equivalently~\eqref{HVI-prob-2}) and define $L:=L_h+L_g$, then these methods are no longer optimal in terms of dependency on $L_g$. 

The pioneering work on designing accelerated algorithm for the HVI problem under this setting is~\cite{chen2017accelerated}, where the authors propose a stochastic accelerated mirror-prox method (SAMP) to solve the VI problem in the form
\begin{eqnarray}
    \mbox{find $x^*\in\ZZ$ such that }\langle H(x)+\nabla g(x)+p'(x),x-x^*\rangle\ge0,\quad\forall x\in\ZZ,\label{nonsmooth-HVI}
\end{eqnarray}
where $p'(x)\in\partial p(x)$ is a subgradient of a relatively simple convex function $p(x)$. In particular, they assume the mappings $H(x)$ and $\nabla g(x)$ can only be evaluated through unbiased stochastic estimators $H(x;\xi)$ and $\nabla g(x;\zeta)$ with bounded variance $\sigma^2$:
\begin{eqnarray}
    \EE\left[\|H(x;\xi)-H(x)\|^2\right]\le \sigma^2,\quad \EE\left[\|\nabla g(x;\zeta)-\nabla g(x)\|^2\right]\le \sigma^2.\label{constant-variance-bound}
\end{eqnarray}
The proposed algorithm SAMP~\cite{chen2017accelerated} can achieve the iteration complexity of
\begin{eqnarray}
    \mathcal{O}\left(\sqrt{\frac{L_g}{\epsilon}}+\frac{L_h}{\epsilon}+\frac{\sigma^2}{\epsilon^2}\right).\label{optimal-nonsmooth-HVI}
\end{eqnarray}

In a recent work~\cite{lan2021mirror} considering a similar problem to~\eqref{nonsmooth-HVI}, the authors further improve the complexities from~\eqref{optimal-nonsmooth-HVI} to
\begin{eqnarray}
    \mathcal{O}\left(\sqrt{\frac{L_g}{\epsilon}}\right)\,\mbox{for $\nabla g$, and}\,\, \mathcal{O}\left(\sqrt{\frac{L_g}{\epsilon}}+\frac{L_h}{\epsilon}+\frac{\sigma^2}{\epsilon^2}\right)\,\mbox{for $H$},\label{optimal-grad-complexity-nonsmooth-HVI}
\end{eqnarray}
with the proposed mirror-prox sliding (MPS) method. We note that the problem in~\cite{lan2021mirror} has slightly different settings than~\cite{chen2017accelerated}, where $\nabla g$ is deterministic and $H$ is stochastic with variance bound as in~\eqref{constant-variance-bound}. In particular the algorithm consists of a double-loop structure, where a new estimation of $\nabla g$ is only obtained at the start of each new outer-loop and remains the same in each iterations in the inner-loop. Therefore, the algorithm is effectively skipping estimations of $\nabla g$ from time to time and is able to retrieve the same optimal (gradient) complexities for a pure optimization problem.

In view of the structure given in the problem~\eqref{nonsmooth-HVI}, the iteration complexity~\eqref{optimal-nonsmooth-HVI} indeed matches the lower bound~\cite{chen2017accelerated} hence optimal. When it comes to {\it gradient complexity}, the result~\eqref{optimal-grad-complexity-nonsmooth-HVI} given in~\cite{lan2021mirror} is optimal. Due to the mapping $p'(x)$ in~\eqref{nonsmooth-HVI} which is not necessarily continuous, solving~\eqref{nonsmooth-HVI} (with solution $x^*$) does not guarantee solving the HVI problem:
\begin{eqnarray}
    \langle H(x^*),x-x^*\rangle+g(x)-g(x^*)+p(x)-p(x^*)\ge0,\quad\forall x\in\ZZ\nonumber
\end{eqnarray}
ending with the same $x^*$ (while the reverse is true). In this paper, we do not consider the presence of the possibly nonsmooth function $p(x)$, which does not impose any noticeable influence on the convergence results. The main difference will be changing from performing prox-mapping on $p$ to projections onto the constraint set directly, which are both assumed to be easily executable in 
this context. 

\subsection{The finite-sum HVI problem}
In this paper, we investigate a specific case when random sampling of the mapping $H$ and $\nabla g$ is necessary. That is, when $H$ and $g$ are both in the finite-sum structure~\eqref{finite-sum H and g}.
Solving finite-sum problem is originally motivated from large-scale machine learning problems, in which a commonly encountered optimization problem is the so-called finite-sum optimization:
\begin{eqnarray}
    \min\limits_{x\in\XX}\,\, g(x):=\sum\limits_{i=1}^mg_i(x),\label{finite-sum-opt}
\end{eqnarray}
where the objective consists of the sum of finitely many (convex) loss functions. When the total number of functions (namely $m$) is large, it can be costly for a deterministic gradient method to evaluate the gradients of all the functions in each iteration. A conventional way for solving the finite-sum model \eqref{finite-sum-opt} is through stochastic gradient descent (SGD), where in each iteration only one or a mini-batch of functions are randomly chosen and the corresponding gradients are estimated. While SGD may improve the overall gradient complexity over the deterministic methods, the iteration complexity to obtain an $\epsilon$-solution is only $\mathcal{O}\left(\frac{1}{\epsilon}\right)$ even if each of the function $g_i(x)$ is strongly convex and smooth. Similarly, when the HVI problem~\eqref{HVI-problem} is faced with the finite-sum structure~\eqref{finite-sum H and g}, taking only single (or mini-batch) sample each iteration for $H$ and $\nabla g$ can result in suboptimal gradient complexity in terms of the number of finite-sum components $m_1$ and $m_2$. In particular, if simply assuming $m=m_1=m_2$, the constant variance bounds in~\eqref{constant-variance-bound} will deteriorate by a factor of $m^2$, then the iteration complexity of SAMP (same as gradient complexity due to constant samples in each iteration) will become
\begin{eqnarray}
    \mathcal{O}\left(\sqrt{\frac{L_g}{\epsilon}}+\frac{L_h}{\epsilon}+\frac{m^2\sigma^2}{\epsilon^2}\right),\label{finite-sum-HVI-complexity}
\end{eqnarray}
which can be much less attractive as the two large terms $m^2$ and $\epsilon^{-2}$ combine.

Thus, it becomes imperative to apply {\it variance reduction} in order to alleviate the dependency on the potentially very large numbers $m_1$ and $m_2$. Variance reduction techniques are first developed for finite-sum optimization to remedy the suboptimality of SGD. Methods such as SAG \cite{schmidt2017minimizing}, SAGA \cite{defazio2014saga}, SVRG \cite{johnson2013accelerating} achieve the gradient complexity $\mathcal{O}\left(\left(m+\frac{L}{\mu}\right)\log\frac{1}{\epsilon}\right)$, assuming each function $g_i(x)$ in~\eqref{finite-sum-opt} is strongly convex with modulus $\mu>0$ and gradient Lipschitz continuous with constant $L\ge\mu$. Further acceleration for variance reduced algorithms is accomplished by Katyusha \cite{allen2017katyusha} and SSNM \cite{zhou2019direct} with gradient complexity $\mathcal{O}\left(\left(m+\sqrt \frac{mL}{\mu}\right)\log\frac{1}{\epsilon}\right)$ for strongly convex $g_i(x)$ and $\mathcal{O}\left(m\sqrt{\frac{1}{\epsilon}}+\sqrt{\frac{mL}{\epsilon}}\right)$ for merely convex $g_i(x)$~\cite{allen2017katyusha}. See also similar results for RPDG~\cite{lan2018optimal}, Catalyst~\cite{lin2015universal}, RGEM~\cite{lan2018random}. In particular, the accelerated variance reduced algorithms in these previous work are optimal for the strongly convex case, but are still suboptimal in view of the lower bound gradient complexity $\Omega\left(m+\sqrt{\frac{mL}{\epsilon}}\right)$ established in~\cite{lan2018optimal}. On the other hand, the work in~\cite{lan2019unified} proposed a unified method Varag (unifying the cases for convex and strongly convex cases), which is first to obtain a near-optimal gradient complexity $\mathcal{O}\left(m\log m+\sqrt{\frac{mL}{\epsilon}}\right)$ for the convex case, which only differs by a log factor from the lower bound.

Turning the focus to (hemi)VI problems, variance reduction techniques have also been incorporated into conventional first-order (stochastic) VI algorithms when the finite-sum structure is considered. In~\cite{alacaoglu2022stochastic}, they consider the HVI problem~\eqref{HVI-problem} when $g$ is lower semicontinuous and the finite-sum manifests in $H$. The various variance reduced algorithms proposed therein have gradient complexities $\mathcal{O}\left(m+\frac{\sqrt{m}L}{\epsilon}\right)$ for monotone $H(x)$ and $\mathcal{O}\left(\left(m+\frac{\sqrt{m}L}{\mu}\right)\log\frac{1}{\epsilon}\right)$ for strongly monotone $H(x)$ (with modulus $\mu>0$ and Lipschitz constant $L\ge\mu$). These gradient complexities are optimal for strongly monotone problem and near-optimal for monotone problem in view of the lower bounds established in~\cite{xie2020lower}.

In the setting of this paper, we consider the HVI problem~\eqref{HVI-problem} when $g$ is gradient Lipschitz continuous {\it and} in the finite-sum form~\eqref{finite-sum H and g}. In particular, we may assume each $H_i(x)$ (resp.\ $\nabla g_i(x)$) is Lipschitz continuous with constant $L_{h(i)}$ (resp.\ $L_{g(i)}$) and define $L_h:=\sum\limits_{i=1}^{m_1}L_h(i)$ (resp.\ $L_g:=\sum\limits_{i=1}^{m_2}L_{g(i)}$). On the one hand, one would definitely look for an algorithm that can achieve better (optimal) dependency on $L_g$ such as SAMP in~\eqref{optimal-nonsmooth-HVI}. On the other hand, due to the finite-sum structure, applying variance reduction in the algorithm is also necessary to avoid poor dependency on $m$ (specifically, $m_1$ and $m_2$) such as in~\eqref{finite-sum-HVI-complexity}. A natural question arises: Is there an algorithm that can deal with both aforementioned aspects and provide improved gradient complexity results in the framework of HVI problem with settings considered in this paper? 
As far as we know, no such algorithm was established 
in the literature yet. This motivates the work in this paper, which provides an affirmative answer to this question. We propose two algorithms for solving the following two finite-sum HVI problems: 
(1) the {\bf S}tochastic {\bf A}ccelerated {\bf V}ariance {\bf R}educed {\bf E}xtra {\bf P}oint method {for {\bf m}onotone} (SAVREP-m) for the setting when $H(x)$ is  merely monotone in Section~\ref{sec:monotone};  
(2) SAVREP for the setting when $H(x)$ is strongly monotone in Section~\ref{sec:strong-monotone}.

\subsection{Main results and the contributions of the paper}

\begin{table}[htbp!]
    \centering
    {
    \begin{tabular}{|m{7em}||m{5.5em}|c||c|}
    \hline
  Related Work  & Problem & Strongly Convex/Monotone & Convex/Monotone \\
    \hline\hline
        $\mbox{Katyusha}^{\scriptsize \mbox{ns}}$ \cite{allen2017katyusha} & {\small Optimization $(m_1=0)$} &$\left(m_2+\sqrt{\frac{m_2L_g}{\mu}}\right)\log\frac{1}{\epsilon}$&$m_2\sqrt{\frac{1}{\epsilon}}+\sqrt{\frac{m_2L_g}{\epsilon}}$ \\
         \hline
         Varag~\cite{lan2019unified}& {\small Optimization $(m_1=0)$} &$\left(m_2+\sqrt{\frac{m_2L_g}{\mu}}\right)\log\frac{1}{\epsilon}$&$m_2\log m_2+\sqrt{\frac{m_2L_g}{\epsilon}}$ \\
         \hline
         {\small Lower Bound \cite{lan2018optimal, woodworth2016tight}}  & {\small Optimization $(m_1=0)$} & $\left(m_2+\sqrt{\frac{m_2L_g}{\mu}}\right)\log\frac{1}{\epsilon}$ & $m_2+\sqrt{\frac{m_2L_g}{\epsilon}}$\\
         \hline\hline
         &&&\\
         {SAMP~\cite{chen2017accelerated}} & {\small HVI ($m_1{=}m_2{=}1$) }&{\small$\left(\frac{L_h}{\mu}+\sqrt{\frac{L_g}{\mu}}\right)\log\frac{1}{\epsilon}$}&{\small$\frac{L_h}{\epsilon}+\sqrt{\frac{L_g}{\epsilon}}$}\\
         &&&\\
         \hline
         &&&\\
         {\scriptsize Alacaoglu and Malitsky \cite{alacaoglu2022stochastic}}& {\small HVI $(m_1,m_2{\gg}1)$} & {\small$\left(m_1+m_2+\frac{\sqrt{m_1+m_2}(L_h+L_g)}{\mu}\right)\log\frac{1}{\epsilon}$}& $m_1+m_2+\frac{\sqrt{m_1+m_2}(L_h+L_g)}{\epsilon}$\\
         &&&\\
         \hline
         &&&\\
         {\bf This work} & {\small HVI ($m_1,m_2{\gg}1$)}&{\small$\left(m_1+m_2+\frac{\sqrt{m_1}L_h}{\mu}+\sqrt{\frac{m_2L_g}{\mu}}\right)\log\frac{1}{\epsilon}$}&{\small$m_1+\frac{\sqrt{m_1}L_h}{\epsilon}+m_2\sqrt{\frac{1}{\epsilon}}+\sqrt{\frac{m_2L_g}{\epsilon}}$}\\
         &&&\\
         \hline
    \end{tabular}
    }
    \caption{Comparison of gradient complexities and the lower bound. The $\mathcal{O}(\cdot)$ notation for the upper bounds and $\Omega(\cdot)$ notation for the lower bounds are omitted. 
    }
    \label{tab:literature}
\end{table}

The main contributions of this paper are as follows.
\begin{itemize}
    \item We consider a finite-sum HVI problem~\eqref{HVI-problem} where the mapping $H$ is (strongly) monotone and Lipschitz continuous, and the function $g$ is convex, differentiable and {\it gradient Lipschitz continuous}. In particular, {\it both} $H$ and $g$ consist of finite-sum of component mappings (functions). 
    \item We propose two variance reduced algorithms for the finite-sum HVI problem of this kind, which are first in the literature as far as our knowledge goes. The first algorithm is for the setting when $H$ is monotone while the second algorithm is for the setting when $H$ is strongly monotone.
    \item The gradient complexity results in this paper can be interpreted as matching 
    the bounds for accelerated methods in various ways. 
    When the problem is {\it not} of finite-sum structure, i.e.\ $m_1=m_2=1$, then the {\it iteration complexities} match the optimal results for the HVI problem for $H$ being either strongly monotone~\cite{huang2021unifying} or monotone~\cite{chen2017accelerated} (the {\it gradient complexity} for $\nabla g$ is still improvable in view of~\cite{lan2021mirror}). When the finite-sum HVI problem is reduced to a regular finite-sum VI problem (i.e.\ $m_2=0$  and $L_g=0$), then the gradient complexities coincide with the best results as in~\cite{alacaoglu2022stochastic} for either strongly monotone or monotone $H$. On the other hand, for $m_1=0$ where the monotone mapping $H$ is null, the gradient complexities coincide with the results for the {\it direct} Katyusha ($\mbox{Katyusha}^{\scriptsize \mbox{ns}}$)~\cite{allen2017katyusha} for either strongly convex or convex objective function. The results for the above correspondences are summarized in Table~\ref{tab:literature}. We also remark that methods such as Katyusha~\cite{allen2017katyusha} (and RPDG~\cite{lan2018optimal}, Catalyst~\cite{lin2015universal}) can have near-optimal gradient complexities for convex problems by adding strongly convex perturbations~\cite{allen2016optimal} and applying their variants for solving strongly convex problems. Since our proposed method SAVREP-m is itself a direct method under the monotone setting, the comparison is only made with the direct methods such as $\mbox{Katyusha}^{\scriptsize \mbox{ns}}$ and Varag~\cite{lan2019unified}. 
    \item We discuss an application of our methods to finite-sum convex optimization with finite-sum constraints through reformulating the problem into a constrained saddle-point problem. We demonstrate potential advantages of exploiting the structure of the problem by differentiating gradient mappings from the general vector mappings (such as in the proposed algorithms) in {the numerical experiments}.
\end{itemize}

Finally, we remark that while our results match the corresponding bounds for existing accelerated methods for either finite-sum optimization~\cite{allen2017katyusha} and finite-sum VI problem~\cite{alacaoglu2022stochastic} for the respective components, there is still room for potential improvements. In particular, the work in~\cite{lan2019unified} has demonstrated an accelerated method that is near-optimal for finite-sum optimization with merely convex objective function, which is a major improvement in terms of $m_2\sqrt{\frac{1}{\epsilon}}$ given by $\mbox{Katyusha}^{\scriptsize \mbox{ns}}$. Currently our results for monotone VI mapping only matches the bound given by $\mbox{Katyusha}^{\scriptsize \mbox{ns}}$ but not Varag~\cite{lan2019unified}. Furthermore, the work in~\cite{lan2021mirror} also demonstrates the possibilities of further reducing the gradient complexities for the mapping $\nabla g_i(\cdot)$ such that it is independent of $L_h$, for another type of (non-finite-sum) HVI problem~\eqref{nonsmooth-HVI} with monotone mappings. While it requires further study to show how these improvements can be achieved in the finite-sum HVI settings considered in this paper, 
the results in this paper can be seen as a step toward optimal bounds.

\subsection{Organization of the paper}

The rest of the paper is organized as follows. In Section \ref{sec:monotone}, we propose a stochastic variance reduced algorithm for the HVI problem~\eqref{HVI-problem} when $H(\cdot)$ is monotone and $g(\cdot)$ is convex. In Section \ref{sec:strong-monotone}, we provide an alternative variance reduced method to solve the case when $H(\cdot)$ is strongly monotone and $g(\cdot)$ is convex. In Section \ref{sec:finite-sum-constrained}, we demonstrate the application to solving finite-sum convex optimization with finite-sum inequality constraints. We present numerical results in Section \ref{sec:numerical} and conclude the paper in Section \ref{sec:conclusion}.

\section{Variance Reduced Scheme for Finite-Sum HVI: Monotone \texorpdfstring{$H(x)$}{H(x)} and Convex \texorpdfstring{$g(x)$}{g(x)}}
\label{sec:monotone}

In this section, we present our first variance reduced scheme for solving the HVI~\eqref{HVI-problem}, where {\it both} the mapping $H(x)$ and the function $g(x)$ take the finite-sum structure in~\eqref{finite-sum H and g}. We assume the constraint set $\ZZ$ to be closed and convex, and the problem is summarized below:
\begin{eqnarray}
\left\{
\begin{array}{ll}
     \mbox{find $x^*\in\ZZ$\,\, s.t.}\quad \langle H(x^*),x-x^*\rangle+g(x)-g(x^*)\ge0,\quad\forall x\in\ZZ, \\
     H(x):=\sum\limits_{i=1}^{m_1}H_i(x),\quad g(x):=\sum\limits_{i=1}^{m_2}g_i(x).
\end{array}
\right.
\label{finite-sum-HVI-summarize}
\end{eqnarray}

We specifically consider the finite-sum mapping $H(\cdot)$ being monotone and each function $g_i(\cdot)$ being convex in this section, and we shall propose an alternative approach for $H(\cdot)$ being {\it strongly monotone} and each $g_i(\cdot)$ being convex in the next section. In both sections, $g(\cdot)$ is assumed to be differentiable. In particular, we assume each $H_i(\cdot)$ to be Lipschitz continuous with constant $L_{h(i)}$, and each $\nabla g_i(x)$ to be Lipschitz continuous with constant $L_{g(i)}$. Let us also define the sum of the Lipschitz constants $L_h:=\sum\limits_{i=1}^{m_1}L_{h(i)}$ and $L_g:=\sum\limits_{i=1}^{m_2}L_{g(i)}$. 

Consider the following update for iteration count $k$ {and non-negative parameters $\alpha_k,\beta_k,\gamma_k$ and $p_1\in[0,1]$}: 

\begin{equation}
\left\{
\begin{array}{lcl}
     \Bar{x}^k&=& (1-p_1)x^k+p_1w^k\\
     {y^k} &=& (1-\alpha_k-\beta_k)v^k+\alpha_k x^k+\beta_k\bar w^k\\
     x^{k+0.5}&=&\arg\min\limits_{x\in\mathcal{Z}}\quad\gamma_k\langle H(w^k)+{\tilde\nabla g(y^k)},x-\Bar{x}^k\rangle+\frac{1}{2}\|x-\Bar{x}^k\|^2\\
     x^{k+1} &=& \arg\min\limits_{x\in\mathcal{Z}}\quad\gamma_k\langle \hat H(x^{k+0.5})+{\tilde\nabla g(y^k)},x-\Bar{x}^k\rangle+\frac{1}{2}\|x-\Bar{x}^k\|^2\\
     { v^{k+1}} &=& {(1-\alpha_k-\beta_k) v^{k}+\alpha_k x^{k+0.5}+\beta_k\bar w^k}\\
     w^{k+1}&=&\left\{
     \begin{array}{ll}
          x^{k+1},& \mbox{with prob. $p_1$} \\
          w^k,    & \mbox{with prob. $1-p_1$}
       
     \end{array}\right.\\
     \bar w^{k+1}&=&\left\{
     \begin{array}{ll}
        \frac{1}{m_2}\sum_{i=k+2-m_2}^{k+1} v^{i}, & \mbox{$m_2|(k+1)$}\\
          \bar w^k, & \mbox{otherwise}.
     \end{array}
     \right.
\end{array}
\right.\label{savrep-m-Update}
\end{equation}
Let us first give explicit definitions for the variance reduced gradient estimators at the corresponding iterates given in \eqref{savrep-m-Update}:
\begin{eqnarray}
&&\hat{{H}}(x^{k+0.5}):={H}(w^k)+{H}_{\xi_k}(x^{k+0.5})-{H}_{\xi_k}(w^k)\label{VR-grad-H}\\
&& \tilde\nabla {g}(y^k):=\nabla {g}(\bar w^k)+\nabla {g_{\zeta_k}}(y^k)-\nabla {g_{\zeta_k}}(\bar w^k).\label{VR-grad-g}
\end{eqnarray}
The above forms follow from the well-established variance reduction literature \cite{allen2017katyusha, alacaoglu2022stochastic}, and the random variables $\xi\,(\zeta)$ take samples from the $m_1\,(m_2)$ individual operators $H_i(\cdot)\,(\nabla g_i(\cdot))$ with probability distribution taking respective Lipschitz constants $L_{h(i)}\,(L_{g(i)})$ into account. In particular, we have
\begin{eqnarray}
\mbox{Pr}\{\xi=i\}=\frac{L_{h(i)}}{L_h} \, := q_i,\,\, i=1,2,...,m_1, \quad \mbox{Pr}\{\zeta=i\}=\frac{L_{g(i)}}{L_g}\, := \pi_i,\,\, i=1,2,...,m_2.\label{sample-prob}
\end{eqnarray}
The stochastic oracles are given by $H_{\xi}(\cdot):=\frac{1}{q_i}H_i(\cdot)$ and $\nabla g_{\zeta}(\cdot)=\frac{1}{\pi_i}\nabla g_i(\cdot)$.

Method \eqref{savrep-m-Update} is a general {\it stochastic variance reduced} scheme for solving \eqref{finite-sum-HVI-summarize}, and is referred to as ``SAVREP-m'' in this paper.
We shall make the following remarks. First, the variance reduction techniques are applied to {\it both} the general vector mapping $H(\cdot)$ and the gradient mapping $\nabla g(\cdot)$, and the resulting update procedure will require using the variance reduced gradient estimator $\hat H(\cdot)$ and $\tilde\nabla g(\cdot)$ respectively. Such variance reduced gradient estimators are also used in the literature for finite-sum optimization~\cite{allen2017katyusha} and for finite-sum VI problem~\cite{alacaoglu2022stochastic}.
In addition, while the update for sequences $x^k$ and $x^{k+0.5}$ is inspired by the famous extra-gradient method~\cite{korpelevich1976extragradient}, the overall update in~\eqref{savrep-m-Update} takes on a more complicated structure with multiple sequences maintained throughout,
which is key to our algorithm, and the derivations of gradient complexity and sample complexity involving the analysis of each of these sequences are discussed in Section \ref{sec:grad-comp-monotone}. Finally, we note the double-loop structure in~\eqref{savrep-m-Update}, which updates $\bar w^{k}$ once every $m_2$ iterations. As a result, the full gradient $\nabla g(\bar w^k)$ is estimated at the beginning of each outer-loop, and such gradient is used to obtain the variance reduced gradient $\tilde\nabla g(y^k)$ within each inner-loop.

\subsection{Gradient complexity analysis}
\label{sec:grad-comp-monotone}

In order to establish a theoretical guarantee for the gradient complexity, we make an additional assumption that the constraint set $\ZZ$ is bounded, which will become unnecessary in the analysis in the next section, where we consider $H(\cdot)$ to be strongly monotone instead. We summarize the assumptions used in this section below.

\begin{assumption}
\label{ass:mon-and-Lip}
    For problem~\eqref{finite-sum-HVI-summarize}, we assume the following: (1) $H(\cdot)$ is monotone with each $H_i(\cdot)$ being Lipschitz continuous with constant $L_{h(i)}$, and we define $L_h:=\sum\limits_{i=1}^{m_1}L_{h(i)}$; (2) Each $g_i(\cdot)$ is convex and Lipschitz smooth with constant $L_{g(i)}$, and we define $L_g:=\sum\limits_{i=1}^{m_2}L_{g(i)}$.

\end{assumption}

\begin{assumption}
\label{ass:monotone-const-bd}
The diameter of the constraint set $\ZZ$ is $\Omega_\ZZ$, i.e., $\sup_{x,y\in \ZZ}\|x-y\|=\Omega_\ZZ.$
\end{assumption}

To simplify the notations in the following analysis, denote the expressions of conditional expectations taken for different random variables:
\begin{eqnarray}
&&\EE_{k_1}[\cdot]:=\EE_{\xi_k}[\cdot|x^k,w^k],\quad \EE_{k_2}[\cdot]:=\EE_{\zeta_k}[\cdot|x^k,\bar w^k,v^k],\label{expectations-1}\\
&& \EE_{k_1+}[\cdot]:=\EE_{\xi_k}[\cdot|x^{k+1},w^k],\quad \EE_{k_2+}[\cdot]:=\EE_{\zeta_k}[\cdot|\bar w^k,v^{k+1}]. \label{expectations-2}
\end{eqnarray}

The gradient complexity analysis of SAVREP-m \eqref{savrep-m-Update} consists of two major steps. In the first step, we first establish one-iteration relation for the vector mapping $H(\cdot)$, followed by one-iteration relation for function $g(\cdot)$, and finally combine the previous two results to establish a one-iteration relation for a function $Q(x;\cdot)$ to be defined later.
In this step, we only consider the iterations from $k$ to $k+1$, which is within a single inner-loop in the update \eqref{savrep-m-Update} with $\bar w^k$ unchanged. In the second step, we derive the relation among iterates after one outer-loop, where the iterations proceed from $sm_2$ to $(s+1)m_2$. This step specifically establishes an inequality relating $\bar w^{(s+1)m_2}$ and $\bar w^{sm_2}$, which eventually guarantees the convergence of the iterate $\bar w^k$ as long as the parameters are chosen to satisfy certain conditions. In particular, the convergence will be in terms of the {\it expected dual gap function} $\EE\left[\max\limits_{x\in\ZZ}\,Q(\bar w^k;x)\right]$. The results derived from the first step are presented in the next three lemmas.

\begin{lemma}
    \label{lem:savrep-m-VI-bd}
    {Consider problem~\eqref{finite-sum-HVI-summarize} with Assumption~\ref{ass:mon-and-Lip}.} For the iterates generated by~\eqref{savrep-m-Update}, the following inequality holds for any $x\in\ZZ$ and for all $k=0,1,2,...$:
    \begin{eqnarray}
        &&\gamma_k\langle H(x)+\tilde\nabla {g}(y^k),x^{k+0.5}-x\rangle\le-d_k(x)+e_{k1}(x)+e_{k2}(x),\nonumber
    \end{eqnarray}
    where
    \begin{eqnarray}
        &&d_k(x):=\frac{1}{2}\left(\|x^{k+1}-x\|^2-(1-p_1)\|x^k-x\|^2-p_1\|w^k-x\|^2+\left(1-p_1\right)\|x^{k+0.5}-x^k\|^2\right),\nonumber\\
        &&e_{k1}(x):=\frac{1}{2}\left(2\gamma_k^2\|{H}_{\xi_k}(x^{k+0.5})-{H}_{\xi_k}(w^k)\|^2-p_1\|x^{k+0.5}-w^k\|^2-\frac{1}{2}\|x^{k+1}-x^{k+0.5}\|^2\right),\nonumber\\
        &&e_{k2}(x):=\gamma_k\langle\hat {H}(x^{k+0.5})-H(x^{k+0.5}),x-x^{k+0.5}\rangle.\nonumber
    \end{eqnarray}

    \begin{proof}
        See Appendix~\ref{proof:savrep-m-VI-bd}.
    \end{proof}
\end{lemma}

\begin{lemma}
    \label{lem:savrep-m-g-bd}
    {Consider problem~\eqref{finite-sum-HVI-summarize} with Assumption~\ref{ass:mon-and-Lip}.} For the iterates generated by~\eqref{savrep-m-Update}, {suppose the condition $1-\alpha_k-\beta_k\ge0$ holds for all $k=0,1,2,...$,} then the following inequality holds for any $x\in\ZZ$ and for all $k=0,1,2,...$:
    \begin{eqnarray}
        g(v^{k+1})-g(x)&\le& (1-\alpha_k-\beta_k)\left(g(v^k)-g(x)\right)+\beta_k\left(g(\bar w^k)-g(x)\right)+\alpha_k\langle\tilde \nabla g(y^k),x^{k+0.5}-x\rangle\nonumber\\
        &&+\left(\frac{\alpha_k^2L_g}{2}+\frac{\alpha_k^2L_g}{\beta_k}\right)\|x^{k+0.5}-x^k\|^2+\alpha_k e_{k3}(x)\nonumber
    \end{eqnarray}
    where
    \begin{eqnarray}
        e_{k3}(x)&:=& \langle \nabla g(y^k)-\tilde\nabla g(y^k),x^{k+0.5}-x\rangle - \frac{\beta_k}{\alpha_k}\left(g(\bar w^k)-g(y^k)-\langle\nabla g(y^k),\bar w^k-y^k\rangle\right)\nonumber\\
        &&-\frac{\alpha_kL_g}{\beta_k}\|x^{k+0.5}-x^k\|^2.\nonumber
    \end{eqnarray}

    \begin{proof}
        See Appendix~\ref{proof:savrep-m-g-bd}.
    \end{proof}
\end{lemma}

{
Now we shall combine the previous two results and derive a one-iteration relation involving the following function:
\begin{eqnarray}
Q(x';x):= \langle H(x),x'-x\rangle+g(x')-g(x).\label{Q-function}
\end{eqnarray}
It can be easily verified that $\max\limits_{x\in\ZZ}\,Q(x';x)$ serves as a merit function (the dual gap function) to our problem~\eqref{finite-sum-HVI-summarize}. In the context of finite sums, we will use the expected dual gap function $\EE\left[\max\limits_{x\in \ZZ}\,Q(x';x)\right]$ to establish the convergence as shown later. 
}

\begin{lemma}
\label{lem:savrep-m-inner-relation}
{Consider problem~\eqref{finite-sum-HVI-summarize} with Assumption~\ref{ass:mon-and-Lip}.} For the iterates generated by \eqref{savrep-m-Update}, assume the following condition holds for all $k=0,1,2,...$:
{
\begin{equation}
    \label{const-3}
    \left\{
    \begin{array}{ll}
        1-p_1-\alpha_k\gamma_k L_g-\frac{2\alpha_k \gamma_k L_g}{\beta_k}\ge0,\\
        1-\alpha_k-\beta_k\ge 0.
    \end{array}
    \right.
\end{equation}
}
Then the following inequality holds for all $k=0,1,2,...$:
{
\begin{eqnarray}
&&Q(v^{k+1};x)+\frac{\alpha_k}{2\gamma_k}\left((1-p_1)\|x^{k+1}-x\|^2+\|w^{k+1}-x\|^2\right)\label{mu-equal-zero-bd} \\
&\le&(1-\alpha_k-\beta_k)Q(v^k;x)+\beta_k Q(\bar w^k;x)+\frac{\alpha_k}{2\gamma_k}\left(\left(1-p_1\right)\|x^k-x\|^2+\|w^k-x\|^2\right)\nonumber\\
&&+\frac{\alpha_k}{\gamma_k}\bar e_k(x).\nonumber 
\nonumber
\end{eqnarray}
where $\bar e_k(x):=e_{k1}(x)+e_{k2}(x)+\gamma_ke_{k3}(x)+\frac{1}{2}e_{k4}(x)$ with the first three terms defined in Lemma~\ref{lem:savrep-m-VI-bd} and Lemma~\ref{lem:savrep-m-g-bd}, and $e_{k4}(x)$ defined as follows:
\begin{eqnarray}
    && e_{k4}(x):= \|w^{k+1}-x\|^2-p_1\|x^{k+1}-x\|^2-(1-p_1)\|w^k-x\|^2.\label{error-terms}
\end{eqnarray}

\begin{proof}
See Appendix~\ref{proof:savrep-m-inner-relation}.
\end{proof}
}
\end{lemma}

{Now we shall proceed to the second step of the analysis.} To simplify the notations in the derivations that follow, define:
{
\[
V_k(x):=(1-p_1)\|x^k-x\|^2+\|w^k-x\|^2.
\]
}
Note that while Lemma \ref{lem:savrep-m-inner-relation} establishes the relation of iterates between iteration $k$ and $k+1$, $\bar w^k$ remains unchanged (unless $m_2|k+1$). Since $\bar w^k$ plays the central role in the convergence under the monotone case, we have to extend the result in \eqref{mu-equal-zero-bd} to iterations between $sm_2$ and $(s+1)m_2$, where $s$ denotes the number of outer-loops (or {\it epochs}). In particular, we assume that the parameters $\alpha_k,\beta_k,\gamma_k$ are also unchanged within each interval of updating $\bar w^k$, i.e.\, $\alpha_{sm_2}=\alpha_{sm_2+1}=\cdots=\alpha_{(s+1)m_2-1}$,  $\beta_{sm_2}=\beta_{sm_2+1}=\cdots=\beta_{(s+1)m_2-1}$, and $\gamma_{sm_2}=\gamma_{sm_2+1}=\cdots=\gamma_{(s+1)m_2-1}$. Then, by summing up inequality \eqref{mu-equal-zero-bd} from $k=sm_2$ to $k=(s+1)m_2-1$, we get
{
\begin{eqnarray}
&&Q(v^{(s+1)m_2};x)+(\alpha_{sm_2}+\beta_{sm_2})\sum_{k=sm_2+1}^{(s+1)m_2-1}Q(v^{k};x)+\frac{\alpha_{sm_2}}{2\gamma_{sm_2}}V_{(s+1)m_2}(x)\nonumber\\
&\le&
(1-\alpha_{sm_2}-\beta_{sm_2} )Q(v^{sm_2};x)+\beta_{sm_2}m_2 Q(\bar{w}^{sm_2};x)+\frac{\alpha_{sm_2}}{2\gamma_{sm_2}}V_{sm_2}(x)+\sum_{k=sm_2}^{(s+1)m_2-1}\frac{\alpha_{sm_2}}{\gamma_{sm_2}}\bar e_k(x)\nonumber\label{eq-non-25}\\
&\le&
(1-\alpha_{sm_2})Q(v^{sm_2};x)+\beta_{sm_2} \sum_{k=(s-1)m_2+1}^{sm_2-1}Q(v^{k};x)+\frac{\alpha_{sm_2}}{2\gamma_{sm_2}}V_{sm_2}(x)+\sum_{k=sm_2}^{(s+1)m_2-1}\frac{\alpha_{sm_2}}{\gamma_{sm_2}}\bar e_k(x)\nonumber.\\
&&\label{eq-non-3}
\end{eqnarray}
}
The last inequality is due to $\sum_{k=(s-1)m_2+1}^{sm_2}Q(v^{k};x)\ge m_2 Q(\bar{w}^{sm_2};x)$, which is from the definition $\bar{w}^{sm_2}=\frac{1}{m_2}\sum_{i=(s-1)m_2+1}^{sm_2} v^{i}$ and the fact that $Q(\cdot;x)$ is convex ($Q(x';x):= \langle H(x),x'-x\rangle+g(x')-g(x)$ and $g$ is convex).

Let us define
\[\Gamma_{s}=\left\{\begin{array}{ll}
1, & \text { when } s=0 \\
(1-\alpha_{(s-1)m_2}) \Gamma_{s-1}, & \text { when } s>0.
\end{array}\right.\]

{
Dividing both sides with $\Gamma_{s+1}$ in~\eqref{eq-non-3}:
\begin{eqnarray}
&&\frac{1}{\Gamma_{s+1}}Q(v^{(s+1)m_2};x)+\frac{\alpha_{sm_2}+\beta_{sm_2}}{\Gamma_{s+1}}\sum_{k=sm_2+1}^{(s+1)m_2-1}Q(v^{k};x)\nonumber\\
&\le&
\frac{1}{\Gamma_{s}}Q(v^{sm_2};x)+\frac{\beta_{sm_2}}{\Gamma_{s+1}} \sum_{k=(s-1)m_2+1}^{sm_2-1}Q(v^{k};x)\nonumber+\frac{\alpha_{sm_2}}{2\gamma_{sm_2}\Gamma_{s+1}}\left[V_{sm_2}(x)-V_{(s+1)m_2}(x)\right]\nonumber\\
&&+\sum_{k=sm_2}^{(s+1)m_2-1}\frac{\alpha_{sm2}}{\Gamma_{s+1}\gamma_{sm_2}}\bar e_k(x)\nonumber\\
&=&
\frac{1}{\Gamma_{s}}Q(v^{sm_2};x)+\frac{\alpha_{(s-1)m_2}+\beta_{(s-1)m_2}}{\Gamma_{s}}\sum_{k=(s-1)m_2+1}^{sm_2-1}Q(v^{k};x)\nonumber\\
&&+\frac{\alpha_{sm_2}}{2\gamma_{sm_2}\Gamma_{s+1}}\left[V_{sm_2}(x)-V_{(s+1)m_2}(x)\right]+\sum_{k=sm_2}^{(s+1)m_2-1}\frac{\alpha_{sm2}}{\Gamma_{s+1}\gamma_{sm_2}}\bar e_k(x),\label{eq-non-4.5}
\end{eqnarray}
where the equality follows by enforcing the next condition:
\begin{eqnarray}
\frac{\beta_{sm_2}}{\Gamma_{s+1}}= \frac{\alpha_{(s-1)m_2}+\beta_{(s-1)m_2}}{\Gamma_{s}}.\label{const-4}
\end{eqnarray}
}

{
Now, assume the next two conditions to hold for {$s=1,...,S$}:
\begin{eqnarray}
B_s:=\frac{\alpha_{sm_2}}{2\gamma_{sm_2}\Gamma_{s+1}},\quad B_{s-1}\le B_{s},\label{const-5}
\end{eqnarray}
\begin{eqnarray}
\alpha_{(s-1)m_2}+\beta_{(s-1)m_2}=1.\label{const-6}
\end{eqnarray}
Then we can obtain the next inequalities by summing up \eqref{eq-non-4.5} for $s=1,...,S-1$, {while simplifying the notation as $\sum\limits_{s,k}:=\sum\limits_{s=0}^{S-1}\sum\limits_{k=sm_2}^{(s+1)m_2-1}$ and $\sum\limits_{s=1,k}:=\sum\limits_{s=1}^{S-1}\sum\limits_{k=sm_2}^{(s+1)m_2-1}$}:
{
\begin{eqnarray}
    &&\frac{1}{\Gamma_{S}}Q(v^{Sm_2};x)+\frac{\alpha_{(S-1)m_2}+\beta_{(S-1)m_2}}{\Gamma_{S}}\sum_{k=(S-1)m_2+1}^{Sm_2-1}Q(v^{k};x)\label{new-ineq-1}\\
&\le& 
\frac{1}{\Gamma_{1}}Q(v^{m_2};x)+\frac{\alpha_{0}+\beta_{0}}{\Gamma_{1}}\sum_{k=1}^{m_2-1}Q(v^{k};x)+\sum_{s=1}^{S-1}B_s\left[V_{sm_2}(x)-V_{(s+1)m_2}(x)\right]+\sum\limits_{s=1,k}2B_s\bar e_k(x).\nonumber
\end{eqnarray}
We can lower bound the LHS of~\eqref{new-ineq-1} from the condition~\eqref{const-6}:
\begin{eqnarray*}
    &&\frac{1}{\Gamma_{S}}Q(v^{Sm_2};x)+\frac{\alpha_{(S-1)m_2}+\beta_{(S-1)m_2}}{\Gamma_{S}}\sum_{k=(S-1)m_2+1}^{Sm_2-1}Q(v^{k};x)\\
    &=& \frac{\alpha_{(S-1)m_2}+\beta_{(S-1)m_2}}{\Gamma_{S}}\sum_{k=(S-1)m_2+1}^{Sm_2}Q(v^{k};x)\\
    &\ge& \frac{m_2\left(\alpha_{(S-1)m_2}+\beta_{(S-1)m_2}\right)}{\Gamma_{S}}Q(\bar{w}^{Sm_2};x),\nonumber
\end{eqnarray*}
where in the inequality we again apply the relation $\sum_{k=(s-1)m_2+1}^{sm_2}Q(v^{k};x)\ge m_2 Q(\bar{w}^{sm_2};x)$. On the other hand, the RHS of~\eqref{new-ineq-1} can be upper bounded by applying~\eqref{eq-non-25} with $s=0$:
\begin{eqnarray*}
    &&\frac{1}{\Gamma_{1}}Q(v^{m_2};x)+\frac{\alpha_{0}+\beta_{0}}{\Gamma_{1}}\sum_{k=1}^{m_2-1}Q(v^{k};x)+\sum_{s=1}^{S-1}B_s\left[V_{sm_2}(x)-V_{(s+1)m_2}(x)\right]+\sum\limits_{s=1,k}2B_s\bar e_k(x)\nonumber\\
    &\le&\frac{(1-\alpha_{0}-\beta_{0} )}{\Gamma_{1}}Q(v^{0};x)+\frac{\beta_0 m_2}{\Gamma_{1}}Q(\bar{w}^{0};x)+\sum_{s=0}^{S-1}B_s\left[V_{sm_2}(x)-V_{(s+1)m_2}(x)\right]+\sum\limits_{s,k}2B_s\bar e_k(x)\nonumber \\
&\le&\frac{(1-\alpha_{0}+(m_2-1)\beta_{0} )}{\Gamma_{1}}Q(w^{0};x)+B_0V_0(x)+\sum_{s=1}^{S-1}(B_s-B_{s-1})V_{sm_2}(x)+\sum\limits_{s,k}2B_s\bar e_k(x),\nonumber
\end{eqnarray*}
where the second inequality is due to combining the first two terms and re-grouping the summation of $V_{sm_2}(x)$. In particular, the nonpositive term $-B_{S-1}V_{Sm_2}(x)$ is removed to create an upper bound.

Combining the above three inequalities, we obtain:
\begin{eqnarray}
    &&\frac{m_2\left(\alpha_{(S-1)m_2}+\beta_{(S-1)m_2}\right)}{\Gamma_{S}}Q(\bar{w}^{Sm_2};x)\nonumber\\
    &\le& \frac{(1-\alpha_{0}+(m_2-1)\beta_{0} )}{\Gamma_{1}}Q(w^{0};x)+B_0V_0(x)+\sum_{s=1}^{S-1}(B_s-B_{s-1})V_{sm_2}(x)+\sum\limits_{s,k}2B_s\bar e_k(x).\nonumber
\end{eqnarray}
Taking maximum over $x\in\ZZ$ on both sides, the next inequality follows:
\begin{eqnarray}
    &&\frac{m_2\beta_{(S-1)m_2}}{\Gamma_{S}}\max\limits_{x\in\ZZ}Q(\bar{w}^{Sm_2};x)\le\frac{m_2\left(\alpha_{(S-1)m_2}+\beta_{(S-1)m_2}\right)}{\Gamma_{S}}\max\limits_{x\in\ZZ}Q(\bar{w}^{Sm_2};x)\nonumber\\
    &\le& \frac{(1-\alpha_{0}+(m_2-1)\beta_{0} )}{\Gamma_{1}}\max\limits_{x\in\ZZ}Q(w^{0};x)+2B_0\Omega_\ZZ^2+2\sum_{s=1}^{S-1}(B_s-B_{s-1})\Omega_\ZZ^2\nonumber+\max\limits_{x\in\ZZ}\left\{\sum\limits_{s,k}2B_s\bar e_k(x)\right\},\nonumber
\end{eqnarray}
where in the first inequality we use the fact that both the parameter $\alpha_{(S-1)m_2}$ and the the term $\max\limits_{x\in\ZZ}Q(\bar{w}^{Sm_2};x)$ are nonnegative, and in the second inequality we apply Assumption~\ref{ass:monotone-const-bd} ($V(x)\le 2\Omega_\ZZ^2$) and condition~\eqref{const-5}. Note that the middle two terms can be further combined: 
\begin{eqnarray*}
    2B_0\Omega_\ZZ^2+2\sum_{s=1}^{S-1}(B_s-B_{s-1})\Omega_\ZZ^2=2B_{S-1}\Omega_\ZZ^2=\frac{\alpha_{(S-1)m_2}}{\gamma_{(S-1)m_2}\Gamma_{S}} \Omega_{\ZZ}^2.
\end{eqnarray*}

Finally, we rearrange the coefficients and summarize the above results together with the required conditions on the parameters $\eqref{const-3},\eqref{const-4},\eqref{const-5}$ in the next lemma:
\begin{lemma}\label{theo-non}
{Consider problem~\eqref{finite-sum-HVI-summarize} with Assumption~\ref{ass:mon-and-Lip} and~\ref{ass:monotone-const-bd}. For the iterates generated by \eqref{savrep-m-Update},} suppose the following conditions hold for $k\ge0$ and $s=1,...,S$:
\begin{eqnarray}
\left\{
\begin{array}{lcl}
    1-p_1-\alpha_k\gamma_k L_g-\frac{2\alpha_k \gamma_k L_g}{\beta_k}&\ge&0 \\
    1-\alpha_k-\beta_k&=&0,
\end{array}
\right.\quad
\left\{
\begin{array}{lcl}
    \frac{\alpha_{(s-1)m_2}}{\gamma_{(s-1)m_2}\Gamma_s}&\le& \frac{\alpha_{sm_2}}{\gamma_{sm_2}\Gamma_{s+1}}\\
    \frac{\beta_{sm_2}}{1-\alpha_{sm_2}}&=& \alpha_{(s-1)m_2}+\beta_{(s-1)m_2} 
\end{array}
\right.
\label{paramter-conditions}
\end{eqnarray}
where $\alpha_k$, $\beta_k$, $\gamma_k$ are constants within each interval of updating $\bar{w}$, i.e.\, $\alpha_{sm_2}=\alpha_{sm_2+1}=\cdots=\alpha_{(s+1)m_2-1}$,  $\beta_{sm_2}=\beta_{sm_2+1}=\cdots=\beta_{(s+1)m_2-1}$, and $\gamma_{sm_2}=\gamma_{sm_2+1}=\cdots=\gamma_{(s+1)m_2-1}$. Then,
\begin{eqnarray}
    \max\limits_{x\in\ZZ}Q(\bar{w}^{Sm_2};x)&\le&\frac{1}{m_2\beta_{(S-1)m_2}} \frac{(1-\alpha_{0}+(m_2-1)\beta_{0} )\Gamma_S}{\Gamma_{1}}\max\limits_{x\in\ZZ}Q(w^{0};x)\nonumber\\
    &&+\frac{\alpha_{(S-1)m_2}}{m_2\gamma_{(S-1)m_2}\beta_{(S-1)m_2}} \Omega_\ZZ^2+\frac{\Gamma_S}{m_2\beta_{(S-1)m_2}}\max\limits_{x\in\ZZ}\left\{\sum\limits_{s,k}\frac{\alpha_{sm_2}}{\Gamma_{s+1}\gamma_{sm_2}}\bar e_k(x)\right\}.\label{Q-converge-err}
\end{eqnarray}
\end{lemma}
}

{
In the next step, we shall take total expectation on both sides in~\eqref{Q-converge-err}, which eventually leads to the convergence of the expected dual gap function $\EE\left[\max\limits_{x\in\ZZ}Q(\bar w^{Sm_2};x)\right]$. To establish a meaningful bound, it is critical that we derive an upper bound for the stochastic error term $\max\limits_{x\in\ZZ}\left\{\sum\limits_{s,k}\frac{\alpha_{sm_2}}{\Gamma_{s+1}\gamma_{sm_2}}\bar e_k(x)\right\}$ in expectation, where the simplified notation for the summation is defined as $\sum\limits_{s,k}:=\sum\limits_{s=0}^{S-1}\sum\limits_{k=sm_2}^{(s+1)m_2-1}$. Such bound is given in the next lemma.
}

{
\begin{lemma}\label{lem:error-upp-bd}
{Consider problem~\eqref{finite-sum-HVI-summarize} with Assumption~\ref{ass:mon-and-Lip} and~\ref{ass:monotone-const-bd}. For the iterates generated by \eqref{savrep-m-Update}, suppose the conditions~\eqref{paramter-conditions} hold for $k\ge0$ and $s=1,...,S$, and the following conditions hold for all $s=1,...,S$:}
\begin{eqnarray}
    3\gamma_{sm_2}^2L_h^2-\frac{p_1}{2}\le 0,\quad \frac{2\alpha_{(S-1)m_2}L_g}{\beta_{(S-1)m_2}}\cdot S\ge \frac{2\alpha^2_{sm_2}L_g}{\Gamma_{s+1}\beta_{sm_2}}
    .\label{para-condition-2}
\end{eqnarray}
Then, we have
    \begin{eqnarray}
        \EE\left[\max\limits_{x\in\ZZ}\left\{\sum\limits_{s=0}^{S-1}\sum_{k=sm_2}^{(s+1)m_2-1}\frac{\alpha_{sm_2}}{\Gamma_{s+1}\gamma_{sm_2}}\bar e_k(x)\right\}\right]\le \frac{1}{2}\left(S_2+S_3+S_4\right)\Omega_{\ZZ}^2
    \end{eqnarray}
    where
    \[
    S_2=S_4=\frac{4\alpha_{(S-1)m_2}}{\Gamma_{S}\gamma_{(S-1)m_2}},\quad S_3 = \frac{2\alpha_{(S-1)m_2}L_g}{\beta_{(S-1)m_2}}\cdot S.
    \]
\end{lemma}
\begin{proof}
    See Appendix~\ref{proof:error-upp-bd}.
\end{proof}

{
In view of Lemma~\ref{lem:error-upp-bd},
the convergence of the expected gap function, derived from taking total expectation in~\eqref{Q-converge-err}, is established in the next theorem.
}

\begin{theorem}
\label{thm:conv-exp-gap}
    Consider the problem~\eqref{finite-sum-HVI-summarize} with Assumption~\ref{ass:mon-and-Lip} and~\ref{ass:monotone-const-bd}.
    Suppose the conditions in~\eqref{paramter-conditions} and~\eqref{para-condition-2} hold for SAVREP-m~\eqref{savrep-m-Update} for $k\ge0$ and $s=1,...,S-1$. Then,
    \begin{eqnarray}
        \EE\left[\max\limits_{x\in\ZZ}\,Q(\bar{w}^{Sm_2};x)\right]&\le&\frac{1}{m_2\beta_{(S-1)m_2}} \frac{(1-\alpha_{0}+(m_2-1)\beta_{0} )\Gamma_S}{\Gamma_{1}}\max\limits_{x\in\ZZ}\,Q(w^{0};x)\nonumber \\
        &&+\left(\frac{\alpha_{(S-1)m_2}}{m_2\gamma_{(S-1)m_2}\beta_{(S-1)m_2}}+\frac{\Gamma_S}{2m_2\beta_{(S-1)m_2}}\cdot (S_2+S_3+S_4)\right)\Omega_\ZZ^2.\nonumber\label{expected-gap-converge}
    \end{eqnarray}
\end{theorem}

}
{We shall specify a set of parameters that satisfy the conditions in \eqref{paramter-conditions} and~\eqref{para-condition-2} and give the corresponding gradient complexities in the next corollary.
\begin{corollary} \label{cor-parameter-complexity}
{In view of Theorem~\ref{thm:conv-exp-gap},} if we choose 
\begin{eqnarray}
&&
p_1=\frac{1}{m_1}\le \frac{1}{2},\quad \alpha_k=\frac{2}{s+4},\quad\beta_k=\frac{s+2}{s+4},\quad \gamma_k=\frac{s+3}{24(L_g+(s+1)L_h\sqrt{m_1}) 
},\nonumber
\end{eqnarray}
where $s=\left\lfloor\frac{k}{m_2}\right\rfloor$, then when $m_2|k$, 
\begin{eqnarray}
\mathbb{E}\left[\max\limits_{x\in\ZZ}\,Q(\bar{w}^{k},x)\right]&\le& \frac{6}
{S^2}\max\limits_{x\in\XX}\,Q(w^0,x)+\frac{228}{m_2S^2}L_g\Omega_\ZZ^2+\frac{216}{m_2S}L_h\sqrt{m_1}\Omega_\ZZ^2
\nonumber\\
    &=&\frac{6m_2^2}{k^2}Q(w^0,x)+\frac{228m_2}{k^2}L_g\Omega_\ZZ^2+\frac{216}{k}L_h\sqrt{m_1}\Omega_\ZZ^2
    \label{monotone-convergence-rate}
\end{eqnarray}
where $S=k/m_2$. The expected gradient complexity for reducing $\mathbb{E}\left[\max\limits_{x\in\ZZ}Q(\bar{w}^{k},x)\right]$ to some $\epsilon>0$ is given by
\begin{eqnarray}
\mathcal{O}
\left(\sqrt{\frac{Q(w^0,x)}{\epsilon}}m_2+\sqrt{\frac{L_gm_2}{\epsilon}}\Omega_\ZZ+\frac{L_h\sqrt{m_1}\Omega_{\ZZ}^2}{\epsilon}+m_1
\right).\label{monotone-grad-complexity}
\end{eqnarray}
\end{corollary}

\begin{proof}
We first verify the conditions \eqref{paramter-conditions} and~\eqref{para-condition-2} are satisfied by the specific choices of the parameters. Note that $\Gamma_s=\frac{6}{(s+2)(s+3)}$, and the following inequalities hold:
\begin{eqnarray}
3\gamma_k^{2}L_h^2\leq 3\left(\frac{s+3}{24(s+1)\sqrt{m_1}}\right)^2{\le \frac{1}{2m_1}=\frac{p_1}{2}} ,\nonumber
\end{eqnarray}
\begin{eqnarray}
p_1+\alpha_k\gamma_k L_g+\frac{2\alpha_k \gamma_k L_g}{\beta_k}\le p_1+5\alpha_k\gamma_k L_g\le \frac{1}{2}+\frac{10}{s+4}\cdot\frac{s+3}{24}\le 1,\nonumber
\end{eqnarray}
\begin{eqnarray}
\frac{\alpha_{sm_2}}{\gamma_{sm_2}\Gamma_{s+1}}={8}(L_g+(s+1)L_h\sqrt{m_1}),\nonumber
\end{eqnarray}
which is non-decreasing in $s=0,1,...,S-1$, and

\begin{eqnarray}
\frac{\beta_{sm_2}}{1-\alpha_{sm_2}}=1= \alpha_{(s-1)m_2}+\beta_{(s-1)m_2}.\nonumber
\end{eqnarray}
Finally,
\begin{eqnarray}
    &&S_3 = \frac{2\alpha_{(S-1)m_2}L_g}{\beta_{(S-1)m_2}}\cdot S=4L_g\cdot\frac{S}{S+1}\ge4L_g\cdot\frac{s+1}{s+2},\,\,s=0,...,S-1. \nonumber
\end{eqnarray}
Furthermore,
\begin{eqnarray}
    \frac{2\alpha^2_{sm_2}L_g}{\Gamma_{s+1}\beta_{sm_2}}= \frac{2L_g\cdot \frac{4}{(s+4)^2}}{\frac{6}{(s+3)(s+4)}\cdot\frac{s+2}{s+4}}=\frac{4L_g}{3}\cdot\frac{(s+3)(s+4)^2}{(s+2)(s+4)^2}=\frac{4L_g}{3}\cdot\frac{s+3}{s+2}\le 4L_g\cdot\frac{s+1}{s+2}. \nonumber
\end{eqnarray}
Therefore, the conditions in \eqref{paramter-conditions} and~\eqref{para-condition-2} are indeed satisfied. 

The convergence rate \eqref{monotone-convergence-rate} can be derived by noticing the next inequalities:

\begin{eqnarray}
\frac{1}{m_2\beta_{(S-1)m_2}} \frac{(1-\alpha_{0}+(m_2-1)\beta_{0} )\Gamma_S}{\Gamma_{1}}= \frac{s+3}{s+1}\Gamma_{S}\le \frac{6}{S^2},\nonumber
\end{eqnarray}
\begin{eqnarray}
\frac{\alpha_{(S-1)m_2}}{m_2\gamma_{(S-1)m_2}\beta_{(S-1)m_2}}\le\frac{48}{m_2S^2}L_g+\frac{48}{m_2S}L_h\sqrt{m_1}
,\nonumber
\end{eqnarray}
\begin{eqnarray}
\frac{\Gamma_S(S_2+S_4)}{2m_2\beta_{(S-1)m_2}}=\frac{4\alpha_{(S-1)m_2}}{m_2\gamma_{(S-1)m_2}\beta_{(S-1)m_2}}= \frac{192}{(S+1)m_2}\cdot\left(\frac{L_g}{S+2}+\frac{SL_h\sqrt{m_1}}{S+2}\right)\le \frac{192L_g}{m_2S^2}+\frac{192L_h\sqrt{m_1}}{m_2S},\nonumber
\end{eqnarray}
\begin{eqnarray}
    \frac{\Gamma_S}{2m_2\beta_{(S-1)m_2}}\cdot S_3=\frac{\Gamma_S\alpha_{(S-1)m_2}L_gS}{m_2\beta^2_{(S-1)m_2}}=\frac{12SL_g}{m_2(S+1)^2(S+2)}\le \frac{12L_g}{m_2S^2},\nonumber
\end{eqnarray}
where the definition of $S_2,S_3,S_3$ can be referred to Lemma~\ref{lem:error-upp-bd}. Therefore, we have 
\begin{eqnarray}
    \mathbb{E}\left[\max\limits_{x\in\ZZ}\,Q(\bar{w}^{k},x)\right]&\le& \frac{6}{S^2}Q(w^0,x)+\frac{252}{m_2S^2}L_g\Omega_\ZZ^2+\frac{240}{m_2S}L_h\sqrt{m_1}\Omega_\ZZ^2 . \nonumber
\end{eqnarray}
Substituting $S=k/m_2$, we get
\[
\mathbb{E}\left[\max\limits_{x\in\ZZ}\,Q(\bar{w}^{k},x)\right]\le \frac{6m_2^2}{k^2}Q(w^0,x)+\frac{252m_2}{k^2}L_g\Omega_\ZZ^2+\frac{240}{k}L_h\sqrt{m_1}\Omega_\ZZ^2.\nonumber
\]

\end{proof}

}

\begin{remark}
{\rm     
{In the case when $m_1=0$, the gradient complexity~\eqref{monotone-grad-complexity} matches the gradient complexity of $\mbox{Katyusha}^{\scriptsize \mbox{ns}}$ for finite-sum optimization~\cite{allen2017katyusha} for convex objective function.} In the case when $m_2=0$, our problem will become similar to the HVI problem considered in~\cite{alacaoglu2022stochastic} where only the vector mapping $H(\cdot)$ consists of finite-sum, and the gradient complexity~\eqref{monotone-grad-complexity} matches the results in~\cite{alacaoglu2022stochastic}. The key structure $m_1\gg 1$ and $m_2\gg 1$ in our HVI problem is, however, what differentiates this work from the previously established results, and it requires applying algorithms such as the proposed ones to guarantee the improved gradient complexity, {which is mainly reflected in the combined term involving both $m_2$ and $\frac{L_g}{\epsilon}$.}
}
\end{remark}

{
\begin{remark}
\rm
    The boundedness assumption on the constraint set in Assumption~\ref{ass:monotone-const-bd} can be relaxed by adopting the {\it restricted dual gap function} in the analysis. The interested readers can refer to~\cite{nesterov2007dual} for more detailed discussion. In this work we adopt Assumption~\ref{ass:monotone-const-bd} instead to keep the analysis simple.
\end{remark}
}

{
\begin{remark}
\label{remark:assumptions}
\rm
    We shall point out that the separate treatment of the finite-sum function $g(\cdot)$ to gain acceleration in terms of constants such as $m_2$ and $L_g$ requires slightly stronger assumptions compared to~\cite{alacaoglu2022stochastic} that equivalently treats $\nabla g(\cdot)$ simply as part of the VI mapping. Indeed, in view of Assumption~\ref{ass:mon-and-Lip}, we only assume the whole finite-sum VI mapping $H(\cdot)$ to be monotone, but need to assume {\it each} function $g_i(\cdot)$ to be convex. While we assume each mapping $H_i(\cdot)$ and $\nabla g_i(\cdot)$ are Lipschitz continuous, such assumption for $H_i(\cdot)$ can in fact be relaxed to become a ``mean-squared smoothness'' (cf.\,Assumption 1-(iv) in~\cite{alacaoglu2022stochastic}) in our analysis (see e.g.\,~\eqref{tower-lipschitz} and~\eqref{tower-lipschitz-2}). In short, the assumptions on convexity and smoothness are slightly more restrictive (in the sense that the assumptions are imposed for each component function, see e.g.\,~\cite{allen2017katyusha,lan2019unified}) to obtain accelerated complexities for finite-sum optimization than for finite-sum VI problem (where similar assumptions do not necessarily hold for each mapping). To gain improved complexities for both parts in finite-sum HVI problem as considered in this paper, we have to make the aforementioned more restrictive assumptions on the corresponding optimization part.
\end{remark}
}

\section{Variance Reduced Scheme for Finite-Sum HVI: Strongly Monotone \texorpdfstring{$H(x)$}{H(x)} and Convex \texorpdfstring{$g(x)$}{g(x)}}
\label{sec:strong-monotone}

We consider the finite-sum mapping $H(\cdot)$ being strongly monotone with modulus $\mu_h>0$ and each function $g_i(\cdot)$ being convex in this section. Same as in 
the previous section, we define $H(\cdot)=\sum\limits_{i=1}^{m_1}H_i(\cdot)$ where each $H_i(\cdot)$ is Lipschitz continuous with constant $L_{h(i)}$, and $ g(x)=\sum\limits_{i=1}^{m_2} g_i(x)$ is sum of convex functions, each with gradient Lipschitz constant $L_{g(i)}$. {As an alternative (but equivalent) setting, we may consider the problem where $H(\cdot)$ is merely monotone but at least one $g_i(\cdot)$ is strongly convex (in addition to each function being convex). Then solving the HVI problem with $H(x)$ replaced by $H(x)+\mu x$ and $g(x)$ replaced by $g(x)-\frac{\mu}{2}\|x\|^2$ (which is an equivalent formulation in view of~\eqref{HVI-prob-2}) allows us to apply the algorithm in this section with the original assumptions where $H(\cdot)$ being strongly monotone instead.}
        
Consider the following update for iteration number $k$ {with non-negative parameters $\alpha,\beta,\gamma$ and $p_1,p_2\in[0,1]$}:
\begin{equation}
\left\{
\begin{array}{lcl}
     \Bar{x}^k&=& (1-p_1)x^k+p_1w^k\\
     {y^k} &=& (1-\alpha-\beta)v^k+\alpha x^k+\beta\bar w^k\\
     x^{k+0.5}&=&\arg\min\limits_{x\in\mathcal{Z}}\,\,\gamma\langle {H}(w^k)+{\tilde\nabla {g}(y^k)},x-\Bar{x}^k\rangle+\frac{1}{2}\|x-\Bar{x}^k\|^2\\
     x^{k+1} &=& \arg\min\limits_{x\in\mathcal{Z}}\,\,\gamma\langle \hat {H}(x^{k+0.5})+{\tilde\nabla {g}(y^k)},x-\Bar{x}^k\rangle+\frac{1}{2}\|x-\Bar{x}^k\|^2\\
     { v^{k+1}} &=& {(1-\alpha-\beta) v^{k}+\alpha x^{k+0.5}+\beta\bar w^k}\\
     w^{k+1}&=&\left\{
     \begin{array}{ll}
          x^{k+1},& \mbox{with prob. $p_1$} \\
          w^k,    & \mbox{with prob. $1-p_1$}
    \end{array}\right.\\
     \bar w^{k+1}&=&\left\{
     \begin{array}{ll}
          v^{k+1},  & \mbox{with prob. $p_2$} \\
          \bar w^k, & \mbox{with prob. $1-p_2$}.
       
     \end{array}
     \right.
\end{array}
\right.\label{VR-VI-Opt-Strong-Update}
\end{equation}

There are two main differences between the update \eqref{VR-VI-Opt-Strong-Update} presented above and the update \eqref{savrep-m-Update} in the previous section. First, while \eqref{savrep-m-Update} has a double-loop structure, which updates $\bar w^{k}$ once every $m_2$ iterations, \eqref{VR-VI-Opt-Strong-Update} simply updates $\bar w^k$ with probability $p_2$ in each iteration. {This single loop structure for strongly monotone is largely inspired by the work in~\cite{kovalev2020don}, where similar loopless variant of Katyusha is proposed.}
Second, instead of using parameters $\alpha_k,\beta_k,\gamma_k$ that depend on iteration number $k$ as in~\eqref{savrep-m-Update}, the update in \eqref{VR-VI-Opt-Strong-Update} uses constant parameters: $\alpha_k=\alpha$, $\beta_k=\beta$, and $\gamma_k=\gamma$ for all $k$. 
We shall refer to the update~\eqref{VR-VI-Opt-Strong-Update} as ``SAVREP''.

\subsection{Gradient complexity analysis}
\label{sec:grad-comp-strong}

{
Similar to the first step in the analysis in Section~\ref{sec:grad-comp-monotone}, we first establish one-iteration relation for the vector mapping $H(\cdot)$, followed by one-iteration relation for function $g(\cdot)$, and finally combine the two results in the decrease in a potential function. The key difference is that the (expected) potential function used in the analysis is an upper bound of the expected distance to the optimal solution $x^*$. Therefore, we will be able to establish the gradient complexity for obtaining an $\epsilon$-solution in expectation: $\EE\left[\|x^k-x^*\|^2\right]\le\epsilon$, instead of considering the expected dual gap function as done for the monotone case. We first summarize the assumptions used in this Section below.
}

{
\begin{assumption}
\label{ass:strong-mon-and-Lip}
    For problem~\eqref{finite-sum-HVI-summarize}, we assume the following: (1) $H(\cdot)$ is strongly monotone with modulus $\mu$ and each $H_i(\cdot)$ is Lipschitz continuous with constant $L_{h(i)}$, and we define $L_h:=\sum\limits_{i=1}^{m_1}L_{h(i)}$; (2) Each $g_i(\cdot)$ is convex and Lipschitz smooth with constant $L_{g(i)}$, and we define $L_g:=\sum\limits_{i=1}^{m_2}L_{g(i)}$.
\end{assumption}
}

The lemma below summarizes the results from the first part of the analysis.

\begin{lemma}
\label{lem:VI-relation-1}
{Consider problem~\eqref{finite-sum-HVI-summarize} with Assumption~\ref{ass:strong-mon-and-Lip}.} For the iterates generated by \eqref{VR-VI-Opt-Strong-Update}, 
the following inequality holds for any $x\in\ZZ$ and $k=0,1,2,...$
\begin{eqnarray}
&&\EE_{k_1}\left[\gamma\langle H(x)+\tilde\nabla {g}(y^k),x^{k+0.5}-x\rangle\right]\nonumber\\
&\le& \frac{1}{2}\EE_{k_1}\left[(1-p_1-{}\gamma\mu_h)\|x^k-x\|^2+p_1\|w^k-x\|^2-\|x^{k+1}-x\|^2\right]\nonumber\\
&&-\frac{1}{2}{\left(p_1-2\gamma^2L_h^2\right)}\EE_{k_1}\left[\|x^{k+0.5}-w^k\|^2\right]-\frac{1}{2}\left(1-p_1-2\gamma\mu_h\right)\EE_{k_1}\left[\|x^{k+0.5}-x^k\|^2\right],\nonumber
\end{eqnarray}
{where $\EE_{k_1}[\cdot]:=\EE_{\xi_k}[\cdot|x^k,w^k]$ as defined in~\eqref{expectations-1}.}

\begin{proof}
See Appendix \ref{proof:VI-relation-1}.
\end{proof}
\end{lemma}

Now we shall proceed to presenting the results in the second part of the analysis, summarized in the next lemma.
\begin{lemma}
\label{lem:function-relation-1}
{Consider problem~\eqref{finite-sum-HVI-summarize} with Assumption~\ref{ass:strong-mon-and-Lip}.} For the iterates generated by \eqref{VR-VI-Opt-Strong-Update}, if the condition $1-\alpha-\beta\ge0$ holds, the following inequality holds for any $x\in\ZZ$ and $k=0,1,2,...$
\begin{eqnarray}
\EE_{k_2}\left[g(v^{k+1})-g(x)\right]&\le& \EE_{k_2}\left[(1-\alpha-\beta)\left(g(v^k)-g(x)\right)+\beta \left(g(\bar w^k)-g(x)\right)\right]\nonumber\\
&&+\EE_{k_2}\left[\alpha\langle\tilde\nabla g(y^k),x^{k+0.5}-x\rangle\right]+{\left(\frac{\alpha^2L_g}{2}+\frac{\alpha^2L_g}{2\beta}
\right)}\EE_{k_2}\left[\|x^{k+0.5}-x^k\|^2\right], \nonumber
\end{eqnarray}
{where $\EE_{k_2}[\cdot]:=\EE_{\zeta_k}[\cdot|x^k,\bar w^k,v^k]$ as defined in~\eqref{expectations-1}.}

\begin{proof}
See Appendix \ref{proof:function-relation-1}.
\end{proof}
\end{lemma}

The last part of the analysis will combine the results from Lemma \ref{lem:VI-relation-1} and Lemma \ref{lem:function-relation-1} and establish the overall per-iteration convergence in terms of a potential function, which involves the function defined in~\eqref{Q-function}.
In particular, we will use the function $Q(x';x^*)$ with $x'$ being the iterates generated by SAVREP \eqref{VR-VI-Opt-Strong-Update}. Then $Q(x';x^*)$ is nonnegative for any $x'\in\ZZ$ by definition in our HVI problem setting~\eqref{finite-sum-HVI-summarize}. In addition, it is upper-bounded in terms of $x'$:
\begin{eqnarray}
Q(x';x^*)&=& \langle H(x^*),x'-x^*\rangle+g(x')-g(x^*)\le \langle H(x'),x'-x^*\rangle-\mu_h\|x'-x^*\|^2+g(x')-g(x^*)\nonumber\\
&\le& \langle H(x')+\nabla g(x'),x'-x^*\rangle-\mu_h\|x'-x^*\|^2\le \frac{1}{4\mu_h}\left\|H(x')+\nabla g(x')\right\|^2.\nonumber
\end{eqnarray}
Now we are ready to show the per-iteration convergence for \eqref{VR-VI-Opt-Strong-Update}:
\begin{theorem}
\label{thm:per-iter-conv-1}
Consider problem~\eqref{finite-sum-HVI-summarize} with Assumption~\ref{ass:strong-mon-and-Lip}.
For the iterates generated by SAVREP~\eqref{VR-VI-Opt-Strong-Update}, suppose the following conditions on the parameters hold:
\begin{equation}
    \label{const-2}
    \left\{
    \begin{array}{ll}
        p_1-2\gamma^2L_h^2-\frac{4\gamma\mu_h}{3}\ge0,\\
        1-p_1-{\frac{10\gamma\mu_h}{3}}-\alpha\gamma L_g-\frac{\alpha \gamma L_g}{\beta}\ge0,
    \end{array}
    \right.,\quad 1-\alpha-\beta\ge0,
\end{equation}
then the following inequality holds for $k=0,1,2,...$
\begin{eqnarray}
&&\EE\left[(1-\phi p_2)Q(v^{k+1};x^*)+\phi Q(\bar w^{k+1};x^*)\right]+\frac{\alpha}{2\gamma}\EE\left[(1-p_1)\|x^{k+1}-x^*\|^2+\|w^{k+1}-x^*\|^2\right]\nonumber\\
&\le& \EE\left[(1-\alpha-\beta)Q(v^k;x^*)+(\beta+\phi(1-p_2)) Q(\bar w^k;x^*)\right]\nonumber\\
&&+\left(1-\frac{\gamma\mu_h}{3}\right)\frac{\alpha}{2\gamma}\EE\left[\left(1-p_1\right)\|x^k-x^*\|^2+\|w^k-x^*\|^2\right].\label{potential-reduce-sto-error-2}
\end{eqnarray}

\begin{proof}
See Appendix \ref{proof:per-iter-conv-1}.
\end{proof}

\end{theorem}

Theorem~\ref{thm:per-iter-conv-1} establishes the relation for the subsequent iterates generated by~\eqref{VR-VI-Opt-Strong-Update}, and we are left with specifying the parameters $\alpha,\beta,\gamma,\phi,p_1,p_2$ under the condition~\eqref{const-2} in order to give explicit gradient complexity results.
We summarize the gradient complexity results in the next corollary, whose proof is relegated to Appendix~\ref{proof:grad-complexity-1}.
\begin{corollary}
\label{prop:grad-complexity-1}
In view of Theorem \ref{thm:per-iter-conv-1}, by specifying the following parameters:
\begin{eqnarray}
\gamma=\frac{1}{4}\min\left(\frac{\sqrt{p_1}}{L_h},\sqrt\frac{p_2}{L_g\mu_h},\frac{p_1}{\mu_h}\right),\quad \alpha=\frac{1}{12}\min\left(\sqrt{\frac{\mu_h}{L_gp_2}},1\right),\quad \beta=\frac{1}{2},\label{SAVREP-parameter-choice}
\end{eqnarray}
and
\[
\phi = \frac{(1+\alpha)m_2}{2},\quad p_1=\frac{1}{m_1},\quad p_2 = \frac{1}{m_2},
\]
{the expected gradient complexity for obtaining $\EE\left[\|w^k-x^*\|^2\right]\le \epsilon$ is}
\begin{eqnarray}
\mathcal{O}
\left(\left(m_1+m_2+\sqrt{\frac{L_gm_2}{\mu_h}}+\frac{L_h\sqrt{m_1}}{\mu_h}\right)\log\frac{d_0}{\epsilon}\right)\label{grad-complexity-strong},
\end{eqnarray}
where $d_0:=\frac{\gamma}{\alpha\mu_h}\left\|H(x^0)+\nabla g(x^0)\right\|^2+2\|x^0-x^*\|^2.$

\begin{proof}
See Appendix \ref{proof:grad-complexity-1}.
\end{proof}

\end{corollary}

The result in Corollary~\ref{prop:grad-complexity-1} compares to the literature as follows.
In~\cite{alacaoglu2022stochastic}, the authors consider a similar finite-sum HVI problem~\eqref{HVI-problem}, but the function $g(x)$ is assumed to be convex lower-semicontinuous and the finite-sum structure is only present in the vector mapping $H(x)$. As a result, the gradient complexity derived in~\cite{alacaoglu2022stochastic} for strongly monotone $H$ is $\mathcal{O}\left(\left(m_1+\frac{L_h\sqrt{m_1}}{\mu_h}\right)\log\frac{1}{\epsilon}\right)$. However, when $g(x)$ becomes differentiable and gradient Lipschitz as in our setting, the obtained gradient complexity by the proposed variance reduced algorithm SAVREP can significantly improve the dependency on $L_g$ {through the term $\sqrt{\frac{L_g}{\mu}}$ and its combined effect with $\sqrt{m_2}$} {(with slightly more restrictive assumptions; see Remark~\ref{remark:assumptions})}. {See Table~\ref{tab:literature} for a more detailed comparisons.} In fact, the dependency on these parameters matches the optimal gradient complexity established for finite-sum strongly convex optimization $\mathcal{O}\left(\left(m_2+\sqrt{\frac{L_gm_2}{\mu_g}}\right)\log\frac{1}{\epsilon}\right)$~\cite{zhou2019direct, allen2017katyusha, lan2019unified}. On the other hand, while the work~\cite{chen2017accelerated} is the first to propose an accelerated (optimal) algorithm for a similar HVI problem considered in this paper (where $H(\cdot)$ is monotone), they focus on the non-finite-sum stochastic setting and the acceleration manifests mainly in $L_g$. In contrast, the finite-sum structure is the main focus in this work and the proposed algorithm achieves improved dependency on $m_2$ in addition to $L_g$.

\section{Finite-Sum Constrained Finite-Sum Optimization}
\label{sec:finite-sum-constrained}

In this section, we introduce an application for which the proposed SAVREP and SAVREP-m can be applied to. Consider the following problem:
\begin{eqnarray}
\begin{array}{lll}
(P) & \min & \sum_{i=1}^{m_2} g_i(x) \\
& \mbox{s.t.} & \sum_{j=1}^{m_1} h_j(x) \le 0 \\
& & x \in \XX.
\end{array}\label{finite-sum-cons-prob}
\end{eqnarray}
While it is not uncommon to formulate the objective function as finite-sum in machine learning research, the specific finite-sum structure of inequality constraints given in \eqref{finite-sum-cons-prob}  is also found in applications such as empirical risk minimization and Neyman-Pearson classification \cite{tong2016survey}. Previous research \cite{aravkin2019level, lin2018levelconvex, lin2018level} has developed level-set methods for solving \eqref{finite-sum-cons-prob}. In particular, \cite{lin2018level} proposed to reformulate the level-set subproblem into saddle-point problem and solve it with variance-reduced method \cite{palaniappan2016stochastic}. 


In this section, we propose to solve \eqref{finite-sum-cons-prob} through its Lagrangian dual formulation, which is equivalently a saddle point problem with a special structure that is suitable for applying the accelerated variance reduced method SAVREP-m. In our discussion, we assume $g_i(x)$ is {convex} for all $i=1,...,m_2$, $h_j(x)=(h_{j,1}(x),\cdots, h_{j,\ell}(x))^\top \in \RR^{\ell}$ and $h_{j,s}(x)$ is convex in $x$ for all $j=1,...,m_1$ and $s=1,...,\ell$, and $\XX\subseteq \RR^n$ is a closed convex set. The corresponding saddle point reformulation of \eqref{finite-sum-cons-prob} solves the following:
\begin{eqnarray}
\min\limits_{x\in\XX}\max\limits_{y\in\mathbb{R}^\ell_+}\,\,L(x;y) :=\sum_{i=1}^{m_2} g_i(x) + \sum_{j=1}^{m_1} y^\top h_j(x),\label{finite-sum-SPP-reform}
\end{eqnarray}
where $L(x;y)$ defines the Lagrangian function of $(P)$. The partial gradients of the Lagrangian function are given by:
\[
\left\{
\begin{array}{lcl}
\nabla_x L(x;y) &=& \sum_{i=1}^{m_2} \nabla g_i(x) + \sum_{j=1}^{m_1} \left( J h_j(x) \right)^\top  y \\
\nabla_y L(x;y) &=& \sum_{j=1}^{m_1} h_j(x) .
\end{array}
\right.
\]
Denote $\YY:=\RR^\ell_+$, then the stationarity condition for \eqref{finite-sum-SPP-reform} is the following VI problem:
\begin{quote}
Find $(x^*,y^*) \in \XX \times \YY$ such that
\begin{eqnarray}
\label{finite-sum-const-VI-prob}
\left( \begin{array}{c} \nabla_x L(x^*;y^*) \\ - \nabla_y L(x^*;y^*) \end{array}
\right)^\top \left( \begin{array}{c} x-x^* \\ y - y^* \end{array}
\right) \ge 0,\,\, \mbox{ for all } (x,y) \in \XX\times \YY,
\end{eqnarray}
\end{quote}
which can be written in the succinct notation: Find $z^*\in\ZZ$ such that $\langle F(z^*),z-z^*\rangle\ge0$ for all $z\in\ZZ$,
where we let $z:=(x;y)$, $\ZZ:=\XX\times\YY$, and
\begin{eqnarray}
F(z)&:=&  \left( \begin{array}{c} \nabla_x L(x;y) \\ - \nabla_y L(x;y) \end{array}
\right)= \sum\limits_{j=1}^{m_1}\begin{pmatrix}(Jh_j(x))^\top y\\-h_j(x)\end{pmatrix}+\sum_{i=1}^{m_2} \left( \begin{array}{c} \nabla g_i(x) \\ 0 \end{array} \right) =\sum\limits_{j=1}^{m_1}H_j(z)+\sum\limits_{i=1}^{m_2}\nabla g_i(z).\nonumber\\
\label{finite-sum-const-VI-op}
\end{eqnarray}
{We may assume a primal-dual solution $(x^*,y^*)$ exists for problem~\eqref{finite-sum-cons-prob} and its dual, which is also a saddle point solution to~\eqref{finite-sum-SPP-reform}. Therefore, we can transform} the original finite-sum constrained finite-sum optimization problem \eqref{finite-sum-cons-prob} into solving a VI problem with the mapping defined in \eqref{finite-sum-const-VI-op}, and such $F(z)$ takes the form of \eqref{HVI-prob-2}, which consists of a finite-sum general vectore mappings $\sum\limits_{j=1}^{m_1}\begin{pmatrix}(Jh_j(x))^\top y\\-h_j(x)\end{pmatrix}$ and a finite-sum gradient mapping $\sum\limits_{i=1}^{m_2} \left( \begin{array}{c} \nabla g_i(x) \\ 0 \end{array} \right)$. Since the mapping $F(z)$ is continuous, additionally we need to show that it is monotone for the VI problem~\eqref{HVI-prob-2} to be equivalent to the HVI problem~\eqref{finite-sum-HVI-summarize} considered in this paper.
Note the following Jacobian matrices:
\[
J\left[\left( \begin{array}{c} \nabla g_i(x) \\ 0 \end{array} \right)\right]=\left[  \begin{array}{cc} \nabla^2 g_i(x) & 0 \\ 0 & 0 \end{array} \right],\quad J\left[\begin{pmatrix}(Jh_j(x))^\top y\\-h_j(x)\end{pmatrix}\right]=\begin{pmatrix}\sum\limits_{s=1}^\ell y_s  \nabla^2 h_{j,s} (x)& (Jh_j(x))^\top\\-Jh_j(x)&0_{\ell \times \ell} \end{pmatrix},
\]
which are both positive semidefinite since $g_i(x)$, $h_{j,s}(x)$ are convex and $y_s\ge0$. Therefore, we can conclude that the mapping $F(z)$ in the VI reformulation is indeed monotone. As a result, to solve the problem (P) in~\eqref{finite-sum-cons-prob}, we can equivalently solve the following finite-sum HVI problem:
\begin{eqnarray}
\left\{
\begin{array}{l}
     \mbox{find $z^*\in\ZZ:=\XX\times \RR_+^\ell$\,\, s.t.}\,\,\langle H(z^*),z-z^*\rangle+g(x)-g(x^*)\ge0,\,\,\forall z:=\begin{pmatrix}
         x\\y
     \end{pmatrix}\in\ZZ, \\
     H(z):=\sum\limits_{j=1}^{m_1}H_j(z)=\sum\limits_{j=1}^{m_1}\begin{pmatrix}(Jh_j(x))^\top y\\-h_j(x)\end{pmatrix},\quad g(x):=\sum\limits_{i=1}^{m_2}g_i(x).
\end{array}
\right.
\label{HVI-form-finite-sum-const-opt}
\end{eqnarray}

While the efficiency of the variance reduced algorithms for optimization is now commonly recognized when the total number $m_2$ of functions $g_i(x)$ in the summation is large, it is also reasonable to apply similar variance reduced techniques for estimating the constraint functions $h_j(x)$ when the total number $m_1$ in the summation is large, as it can be costly to evaluate all these constraint functions (or their Jacobians) in each iteration.  Problem $(P)$ in \eqref{finite-sum-cons-prob} describes exactly such a situation, and by reformulating the original problem into a finite-sum VI with the special structure \eqref{finite-sum-const-VI-op}, the proposed SAVREP-m in Section \ref{sec:monotone} can be applied. It incorporates variance reduction into the update process for both finite-sum gradient/VI mappings, where the latter is attributed to the (Jacobians) of the constraints $h_i(x)$ and the corresponding dual variable $y$. {To apply the established theoretical results and for implementation purpose, {we also need each component mapping $H_j(z)$ and $\nabla g_i(z)$ to be Lipschitz continuous in addition to monotone, in view of Assumption~\ref{ass:mon-and-Lip} and~\ref{ass:strong-mon-and-Lip}. While in general it is not true for $H_j(z)$ due to the unbounded dual varaible $y$ involved in the mapping,} we may consider the algorithm can be run within a large enough convex compact set that contains one finite primal-dual solution $(x^*,y^*)$. That is, replace $\ZZ$ with a convex compact constraint set $\ZZ'\subset\ZZ$ such that $z^*\in\ZZ'$. {It is then easy to see that $H_j(z)$ is Lipschitz continuous within such compact set. We note that this is also a common approach in analysis if one tries to relax the boundedness assumption for monotone problems. One would assume such compact convex subset containing a solution exists, and the regular merit function will be replaced by a restricted merit function defined on it. See, e.g.\,~\cite{nesterov2007dual} for related discussion.
}

Alternatively, one can also apply SAVREP proposed in Section \ref{sec:strong-monotone}, which instead solves the HVI with strongly monotone vector mapping $H(\cdot)$. While such mapping in our HVI reformulation~\eqref{HVI-form-finite-sum-const-opt} is merely monotone, it can be easily transformed to a strongly monotone mapping by considering the following {\it approximated} HVI problem with the perturbed mapping:
\begin{eqnarray}
H_{\mu}(z):=H(z)+\mu z,\label{finite-sum-const-perturb-map}
\end{eqnarray}
which is strongly monotone with $\mu>0$ with $H(z)$ defined in \eqref{HVI-form-finite-sum-const-opt}. Note that SAVREP only requires $H(z)=\sum\limits_{j=1}^{m_1}H_j(z)$ to be strongly monotone, so the perturbation term $\mu z$ can be associated to $H_j(z)$ for arbitrary $j=1,2,..,m_1$. In particular, we can construct the variance reduced gradient estimators in \eqref{VR-grad-H} as $\hat H(z^{k+0.5}):=H(w^k)+H_{\xi_k}(z^{k+0.5})-H_{\xi_k}(w^k)+\mu z^{k+0.5}$,
where $\xi_k$ randomly samples from $j=1,2,...,m_1$ and $H_j(\cdot)$ is defined in \eqref{HVI-form-finite-sum-const-opt}. The counterpart for $\tilde\nabla g(z^k)$ remains unchanged from \eqref{VR-grad-g}.

{
In general, if we wish to approximate a solution $z^*$ to the VI problem with monotone mapping $F(z)$ and constraint $z\in\ZZ$, we may instead solve for an approximate solution to the ``perturbed'' strongly monotone mapping $F_\mu(z):=F(z)+\mu z$ in $\ZZ$. Under assumption such as boundedness of $\ZZ$, it can be shown that an $\epsilon$-solution to the regularized problem $F_\mu(x)$ (in terms of distance to the solution) is also an $\epsilon$-solution to the original problem with $F(x)$ (in terms of the dual gap function), provided that $\mu=\mathcal{O}(\epsilon)$. Since the HVI problem~\eqref{HVI-form-finite-sum-const-opt} is equivalent to the VI problem with mapping~\eqref{finite-sum-const-VI-op}, replacing $H(z)$ with $H_\mu(z)$ is same as replacing $F(z)$ with $F_\mu(z)$. Therefore, we may apply SAVREP to solve for an approximated solution for small perturbation $\mu$. In practice, the single-loop structure of SAVREP makes it easier to implement compared to its monotone variant SAVREP-m.
}

\section{Numerical Experiments}
\label{sec:numerical}
In this section, we evaluate the numerical performance of SAVREP and SAVREP-m 
by using the same example as in \cite{lin2018level}, which is a Neyman-Pearson
classification problem \cite{tong2016survey} formulated as
\begin{eqnarray}
\min _{\|\mathbf{x}\|_2 \leq \lambda} \frac{1}{n_0} \sum_{j=1}^{n_0} \phi\left(\mathbf{x}^{\top} \xi_{0 j}\right), \text { s.t. } \frac{1}{n_1} \sum_{j=1}^{n_1} \phi\left(-\mathbf{x}^{\top} \xi_{1 j}\right) \leq r_1,\label{NP-classification}
\end{eqnarray}
where $\phi$ is the loss function, defined as smoothed hinge loss function in the experiment for SAVREP and logistic loss function in the experiment for SAVREP-m.
The dataset is the rcv1 training data set from LIBSVM library with $20,242$ data points with $n_0 = 10, 491$ and $n_1 = 9, 751$ and a dimension of $47,236$. {In particular, in the form of reformulated finite-sum HVI problem~\eqref{HVI-form-finite-sum-const-opt}, the number of data points $n_0$ corresponds to $m_2$ and $n_1$ corresponds to $m_1$.} 

\subsection{SAVREP}
\label{sec:numerical-savrep}

\begin{figure}[h!]
  \begin{minipage}[h!]{0.33\linewidth}
    \centering
    \includegraphics[scale=0.3]{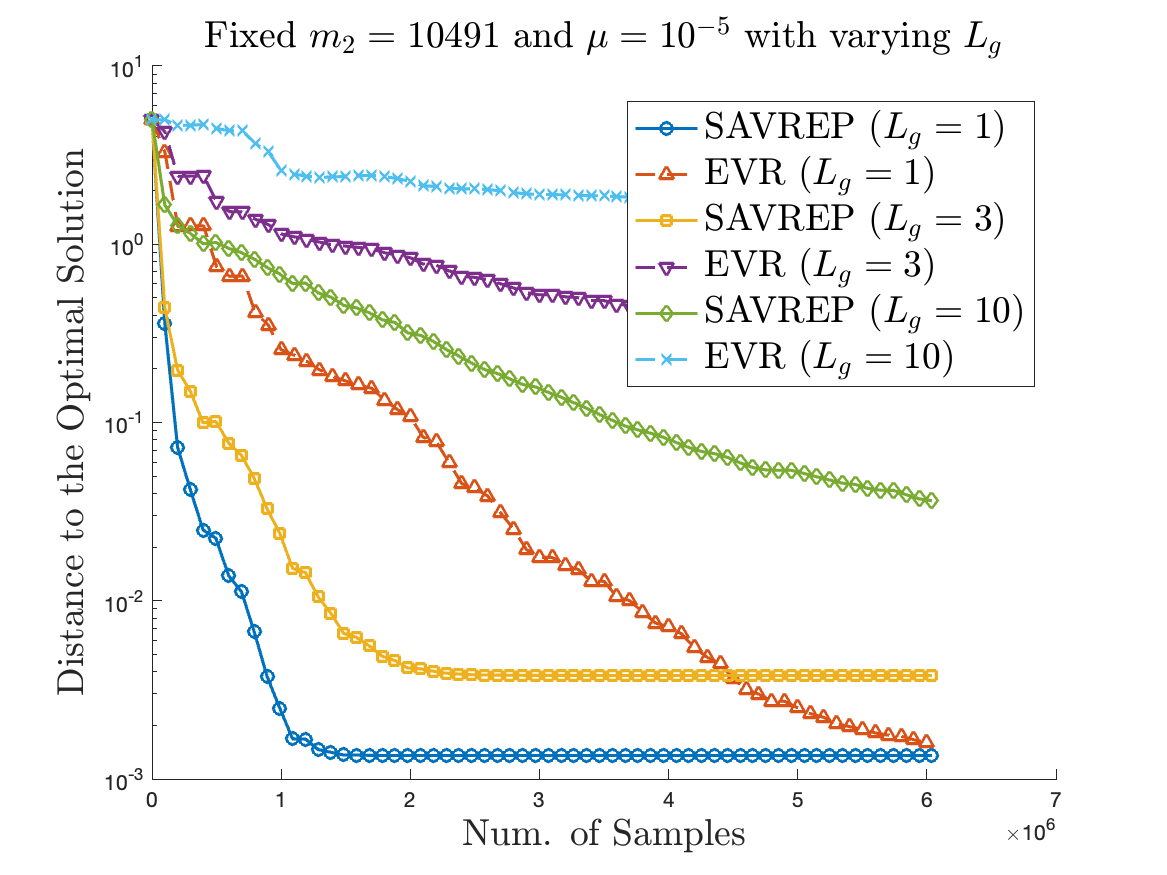}
  \end{minipage}
  \begin{minipage}[h!]{0.33\linewidth}
    \centering
    \includegraphics[scale=0.3]{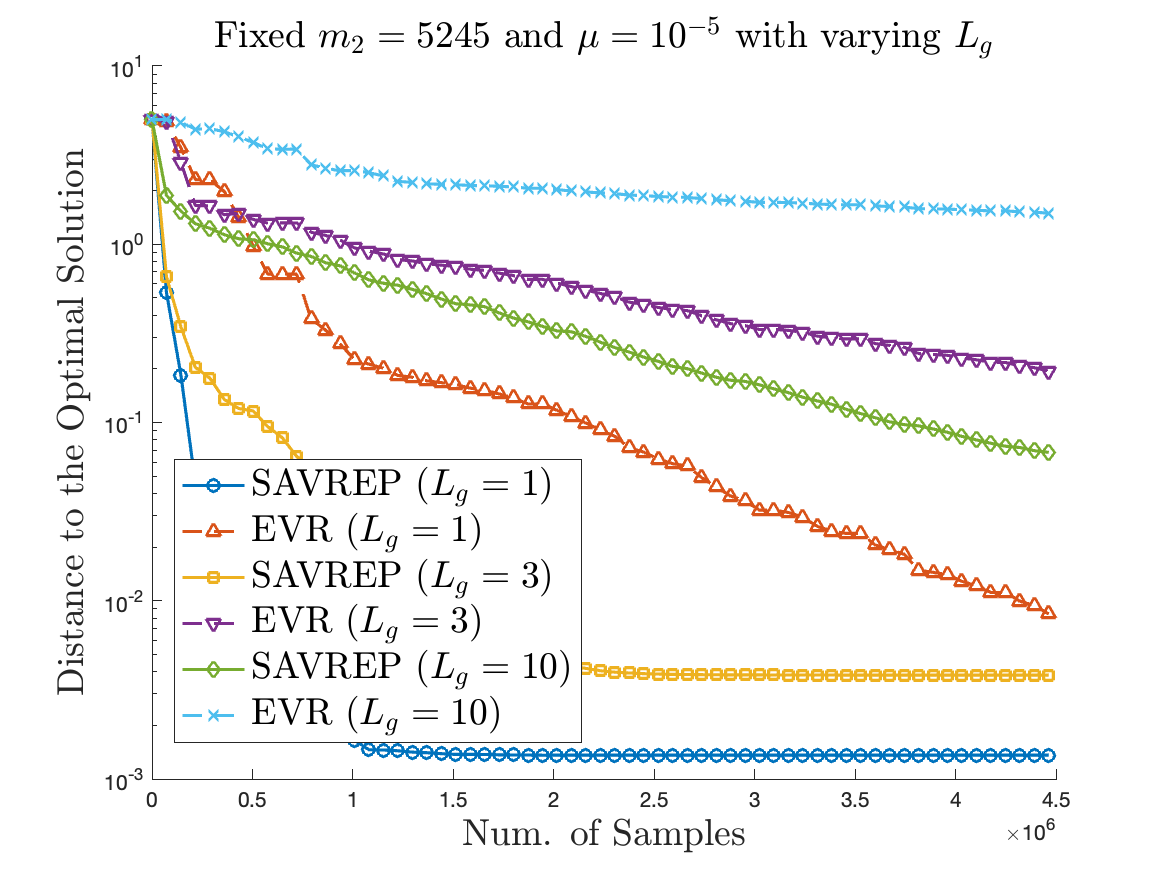}
  \end{minipage}
  \begin{minipage}[h!]{0.33\linewidth}
    \centering
    \includegraphics[scale=0.3]{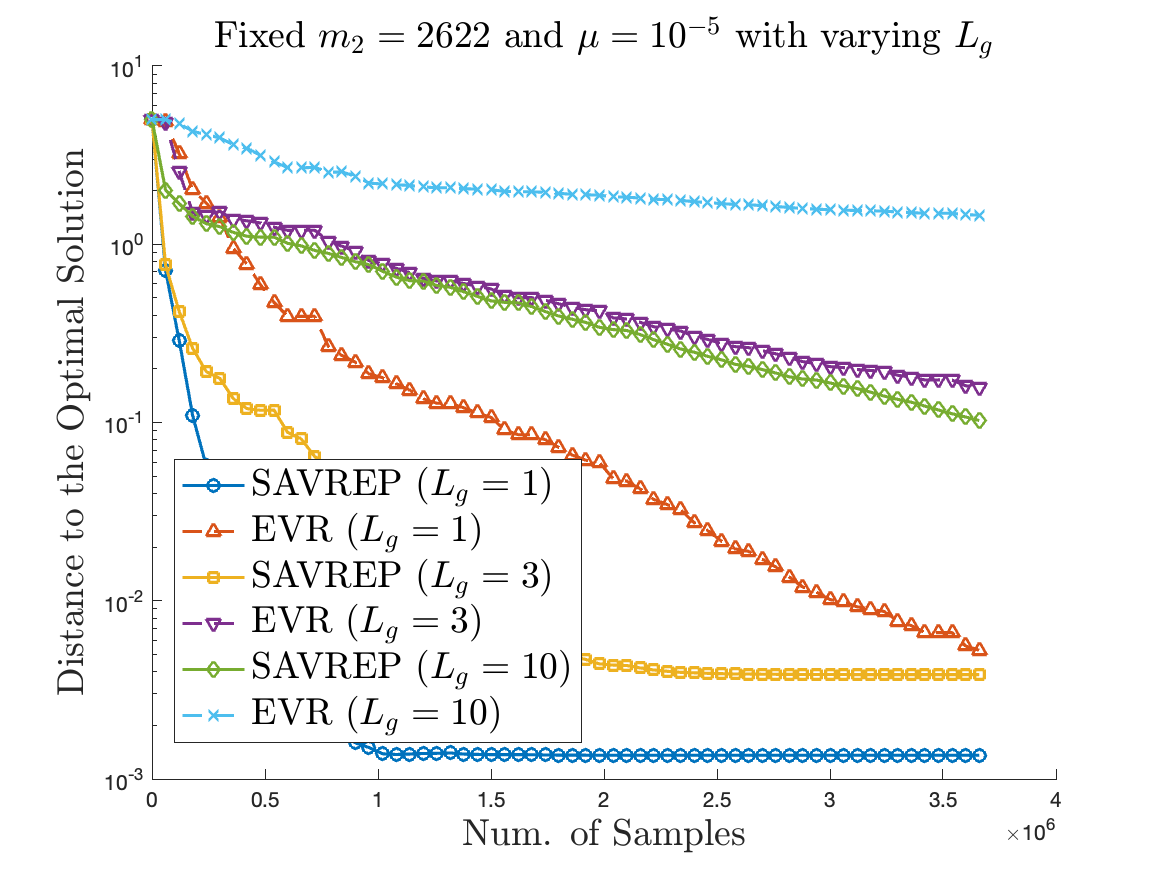}
  \end{minipage}
  \begin{minipage}[h!]{0.33\linewidth}
    \centering
    \includegraphics[scale=0.3]{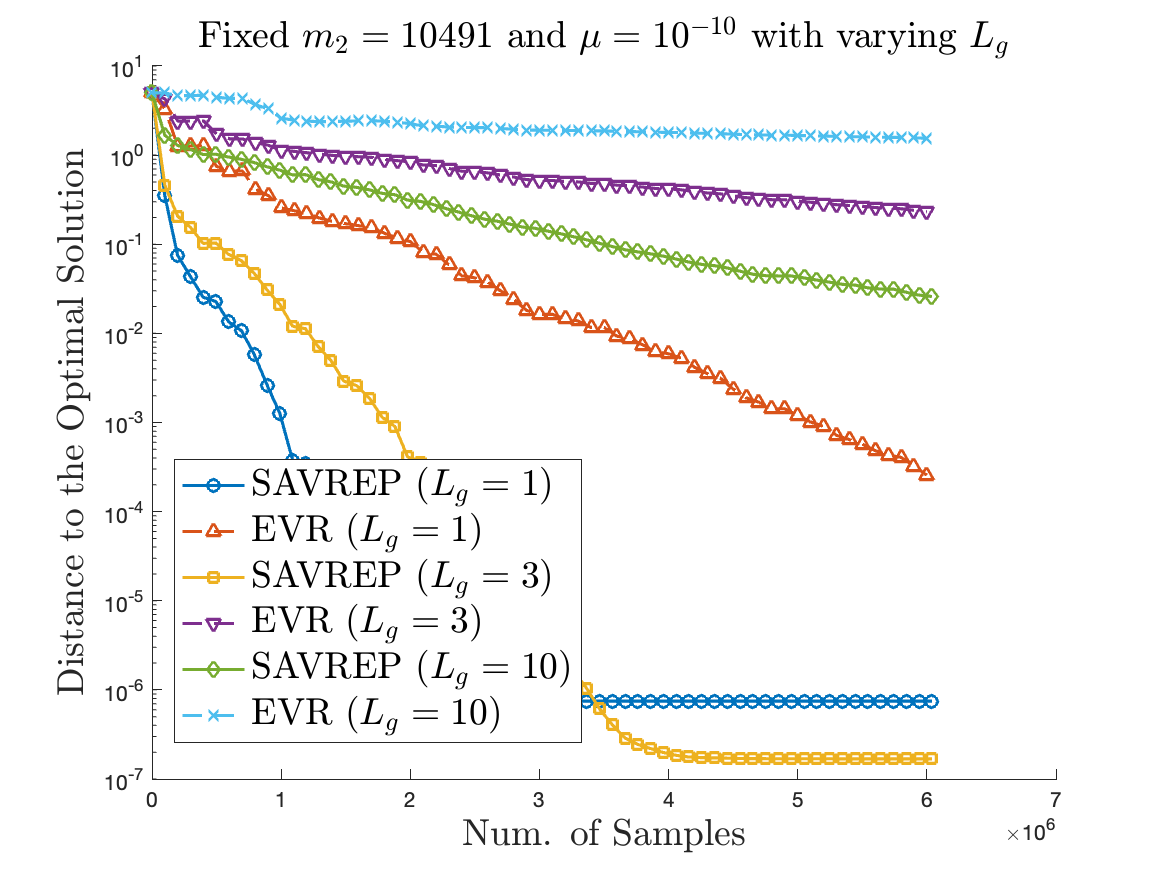}
  \end{minipage}%
  \begin{minipage}[h!]{0.33\linewidth}
    \centering
    \includegraphics[scale=0.3]{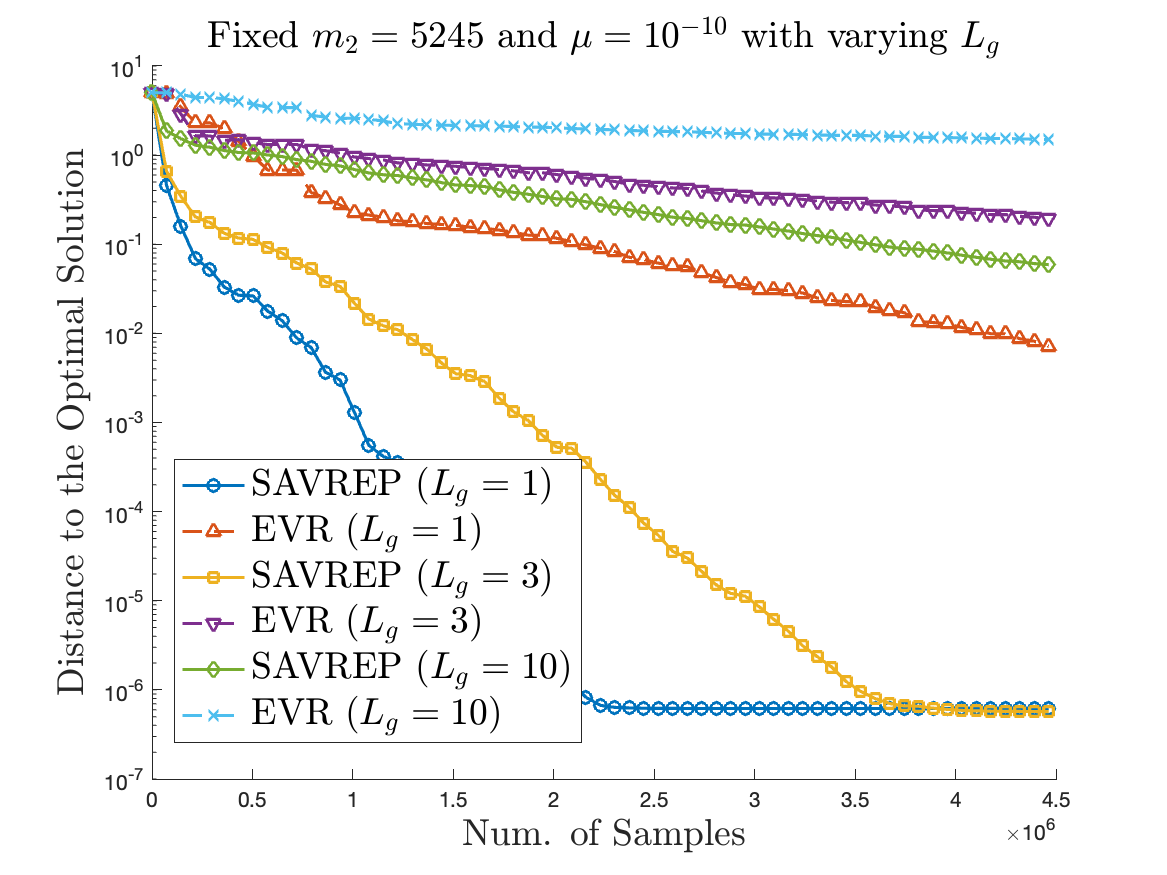}
  \end{minipage}
  \begin{minipage}[h!]{0.33\linewidth}
    \centering
    \includegraphics[scale=0.3]{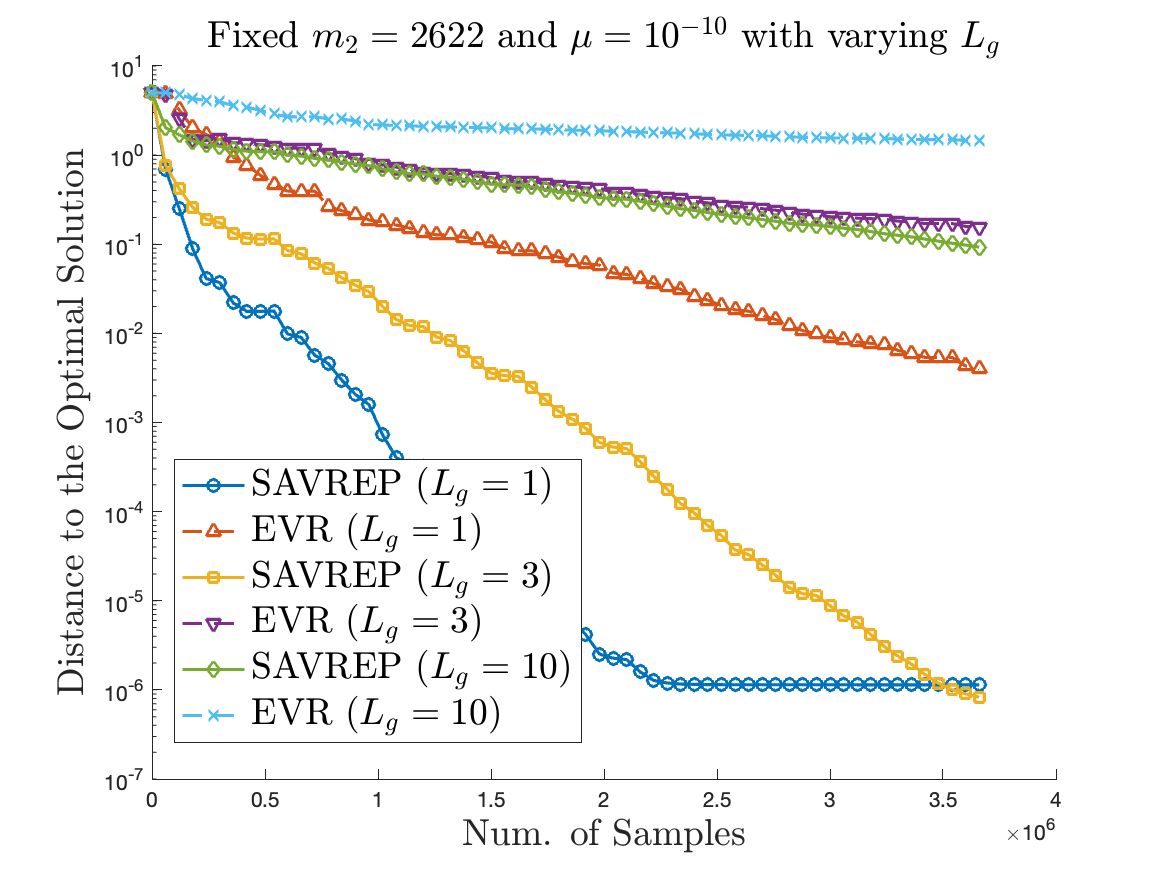}
  \end{minipage}
  \caption{Convergence of SAVREP (distance to the optimal solution): $\mu=10^{-5}$ (first row), $\mu=10^{-10}$ (second row); $m=10491$ (first column), $m=5245$ (second column), $m=2622$ (third column).}
  \label{fig:savrep-conv-fix-m-1-vary-L}
\end{figure}

In this experiment, the loss function is defined as
\begin{eqnarray}
\phi(t)=
\left\{
\begin{array}{ll}
    \frac{1}{2}-t,&t\le 0,\\
    \frac 1 2 (1-t)^2,& 0<t\le 1,\\
    0,& t>1.
\end{array}\right.\nonumber
\end{eqnarray}
To demonstrate the performance of SAVREP, the problem that is actually solved in this experiment is the {\it perturbed problem}~\eqref{finite-sum-const-perturb-map} of the HVI reformulation of the original problem~\eqref{NP-classification}.
The problem parameters are set as $\lambda = 5$ and $r_1 = 0.1$, and the perturbation is set as $\mu=10^{-5}, 10^{-10}$ respectively. We compare the performance of SAVREP with extragradient with variance reduction (EVR)~\cite{alacaoglu2022stochastic}. Both of the methods use the mini-batch with a batch size of $100$ to get the stochastic gradient estimators. {In these experiments, parameter-tuning is performed for $\tau$ in EVR and for $\alpha$ and $\gamma$ in SAVREP. The final parameters used in each of these experiments are determined by multiplying the theoretical values for each method with a learning late. To find the best learning rate, grid-search is performed and the values corresponding to the best convergence performance is used. Appendix~\ref{app:parameters-numerical} summarizes the learning rates used in each experiment, where the corresponding theoretical values can be referred to~\cite{alacaoglu2022stochastic} (Theorem 2.5) for EVR and~\eqref{SAVREP-parameter-choice} for SAVREP.
The results are shown in Figure~\ref{fig:savrep-conv-fix-m-1-vary-L} and Figure~\ref{fig:savrep-conv-fix-L-vary-m}, where we use distance to the optimal solution to the original problem~\eqref{NP-classification} (solved by CVX mosek) as the performance measure. In particular, Figure~\ref{fig:savrep-conv-fix-m-1-vary-L} demonstrates how varying condition number $\frac{L_g}{\mu}$ could have affected the convergence behavior of these two methods. The number of finite-sum convex function components $m_2$ is fixed at different values ($\{10491,5245,2622\}$ from left to right), and the perturbation is fixed at $\mu=10^{-5}$ (first row) or $\mu=10^{-10}$ (second row). Experiments with varying $L_g$ chosen from the set $\{1,3,10\}$ are performed under each of the combination for $m_2$ and $\mu$ above. The results show that our proposed method SAVREP not only converges faster than EVR under the same condition, but it can also possibly outperform EVR even when the condition number $\frac{L_g}{\mu}$ is roughly three times larger for the former. This can be observed from all the graphs in Figure~\ref{fig:savrep-conv-fix-m-1-vary-L} when we compare, for example, SAVREP ($L_g=3$) with EVR ($L_g=1$). Such an advantage is still present but becomes less conspicuous if we compare SAVREP ($L_g=10$) with EVR ($L_g=3$) and can vanish if this difference becomes as large as 10 times, see SAVREP ($L_g=10$) and EVR ($L_g=1$). We may conclude that, the advantages of applying SAVREP over EVR in these experiments manifest from the perspective that the former can better handle a larger condition number $\frac{L_g}{\mu}$ up until a ratio lying approximately between three times and ten times. 

On the other hand, Figure~\ref{fig:savrep-conv-fix-L-vary-m} shows the results for the experiments when $L_g$ is fixed at different values ($\{1,3,10\}$ from left to right) but $m_2$ is chosen from $\{2622, 5245, 10491\}$ (the data points from class $0$ were removed accordingly to reflect the changes in $m_2$). We can see that in general, under the same condition number $\frac{L_g}{\mu}$, SAVREP always performs better than EVR regardless of the different values of $m_2$ set in these experiments. Note that the changes in $m_2$ only bring limited effect to the convergence behavior for both methods, but the effect is slightly more obvious for EVR. Referring to the theoretical bounds in Table~\ref{tab:literature}, for strongly monotone problem both EVR~\cite{alacaoglu2022stochastic} and SAVREP have the same order of dependency on $m_2$ but multiplied by a different factor ($\frac{L_g}{\mu}$ or $\sqrt{\frac{L_g}{\mu}}$). This may explain both methods show similar sensitivity on the changes of $m_2$, but SAVREP always have a better performance. If we compare the performance of same method across different columns (when $L_g$ changes), then indeed EVR shows stronger dependency on the increase of $L_g$ and the performance decays more significantly compared to SAVREP. Overall, from Figure~\ref{fig:savrep-conv-fix-m-1-vary-L} and Figure~\ref{fig:savrep-conv-fix-L-vary-m} we observe that the condition number $\frac{L_g}{\mu}$ has larger influence to the convergence behavior than $m_2$.

}

\begin{figure}[htbp!]
  \begin{minipage}[h!]{0.33\linewidth}
    \centering
    \includegraphics[scale=0.3]{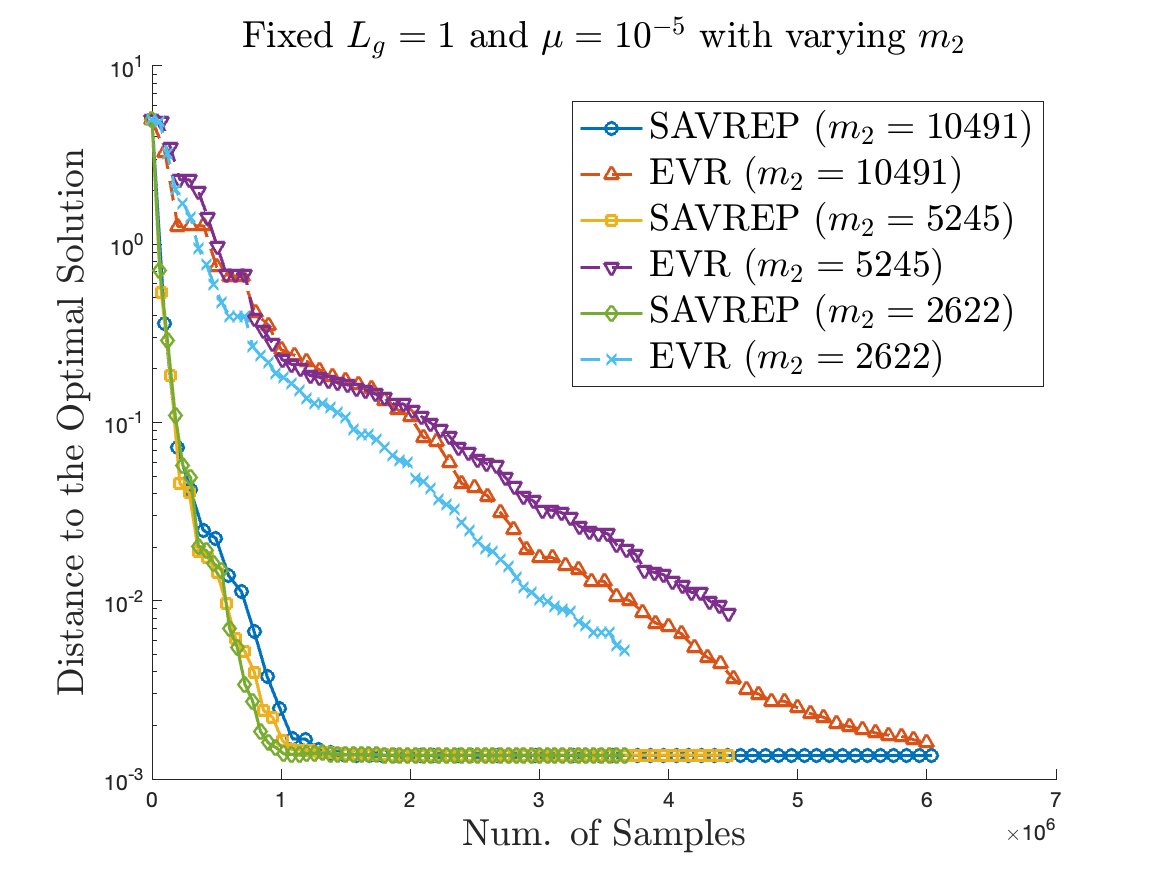}
  \end{minipage}
  \begin{minipage}[h!]{0.33\linewidth}
    \centering
    \includegraphics[scale=0.3]{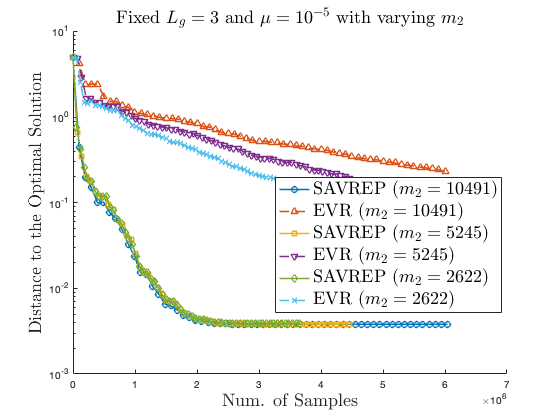}
  \end{minipage}
  \begin{minipage}[h!]{0.33\linewidth}
    \centering
    \includegraphics[scale=0.3]{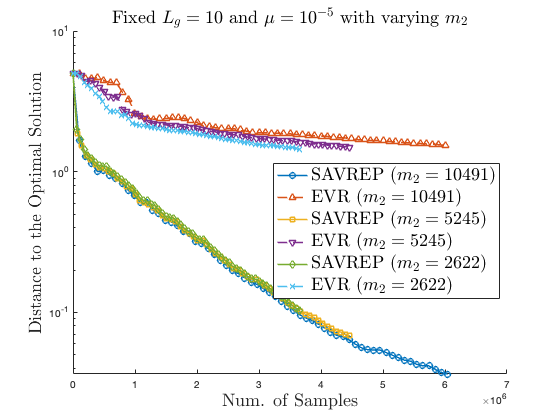}
  \end{minipage}
  \begin{minipage}[h!]{0.33\linewidth}
    \centering
    \includegraphics[scale=0.3]{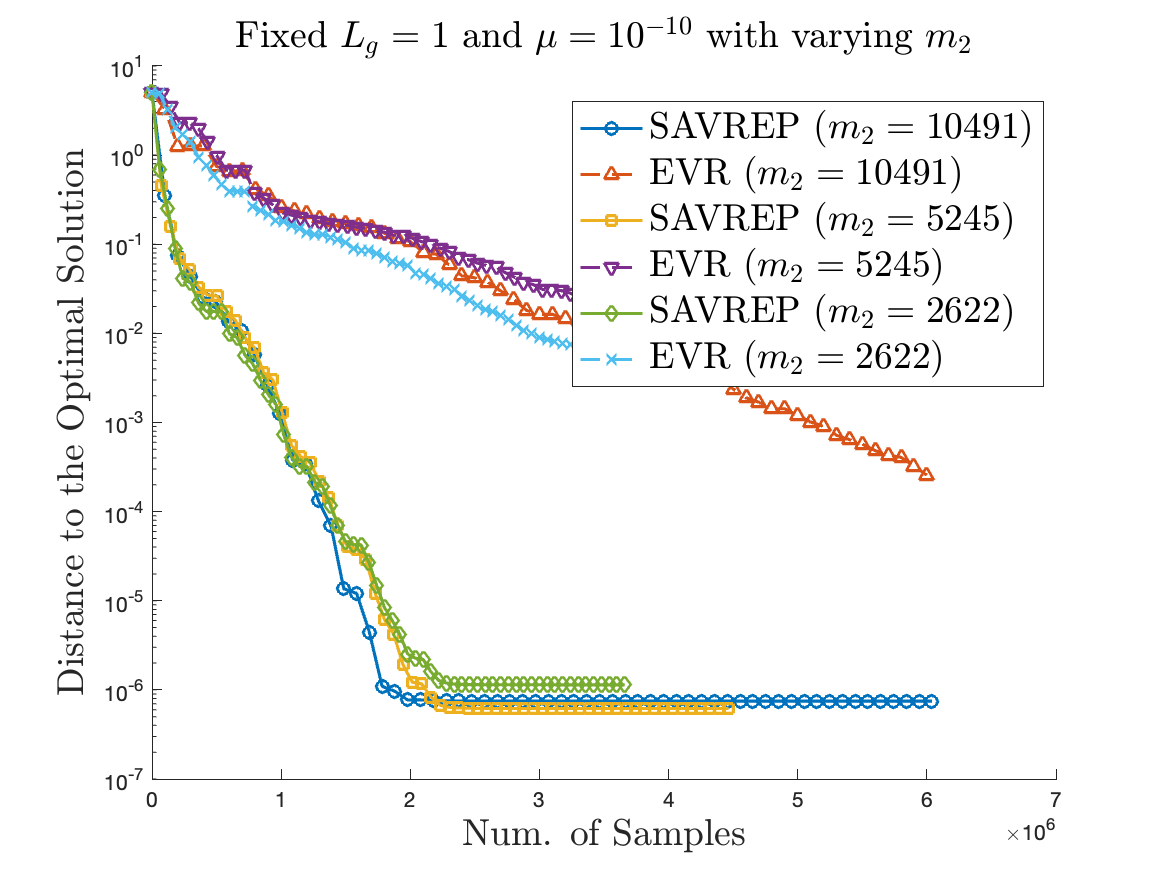}
  \end{minipage}%
  \begin{minipage}[h!]{0.33\linewidth}
    \centering
    \includegraphics[scale=0.3]{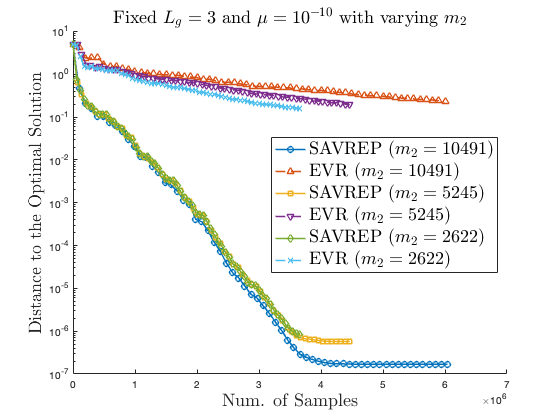}
  \end{minipage}
  \begin{minipage}[h!]{0.33\linewidth}
    \centering
    \includegraphics[scale=0.3]{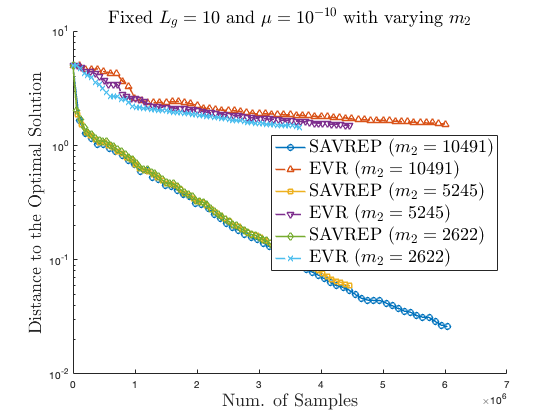}
  \end{minipage}
  \caption{Convergence of SAVREP (distance to the optimal solution): $\mu=10^{-5}$ (first row), $\mu=10^{-10}$ (second row); $L_g=1$ (first column), $L_g=3$ (second column), $L_g=10$ (third column).}
  \label{fig:savrep-conv-fix-L-vary-m}
\end{figure}

\subsection{SAVREP-m}

\begin{figure}[h!]
  \begin{minipage}[h!]{0.33\linewidth}
    \centering
    \includegraphics[scale=0.3]{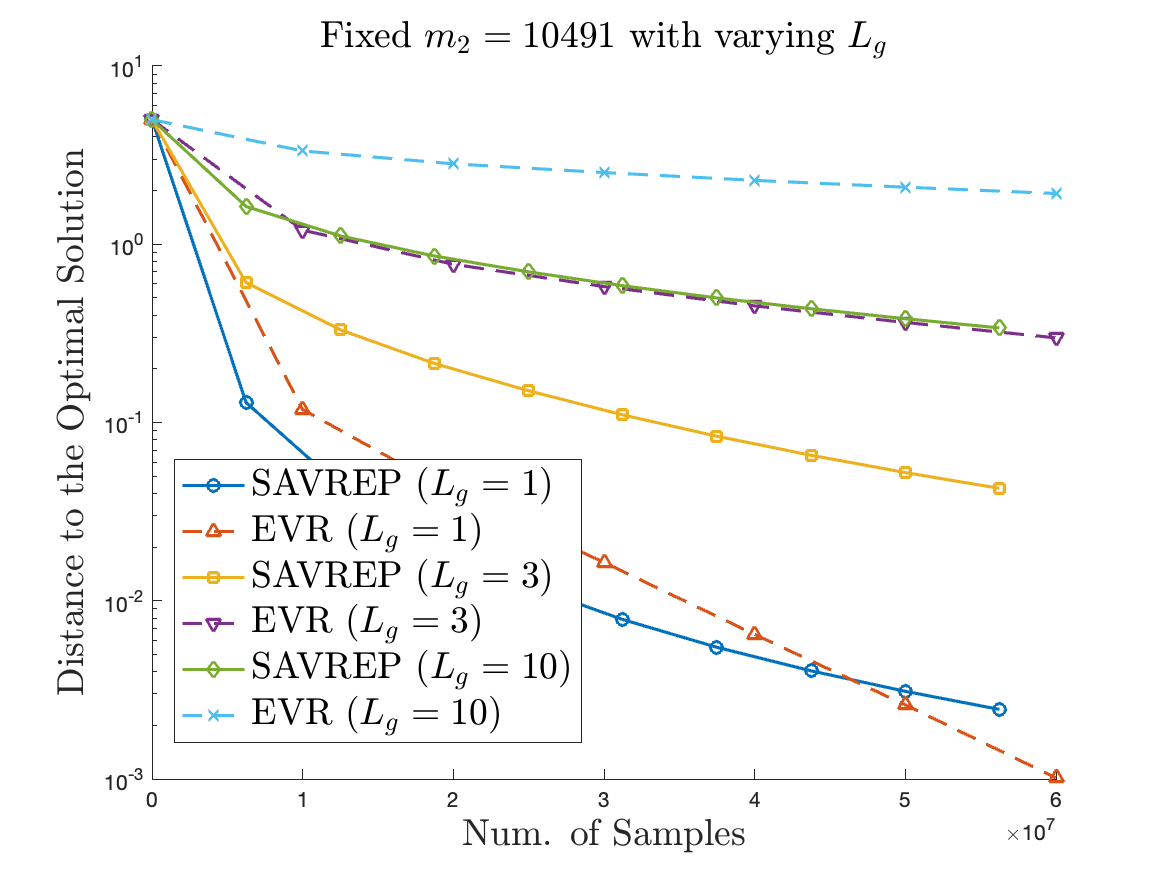}
  \end{minipage}
  \begin{minipage}[h!]{0.33\linewidth}
    \centering
    \includegraphics[scale=0.3]{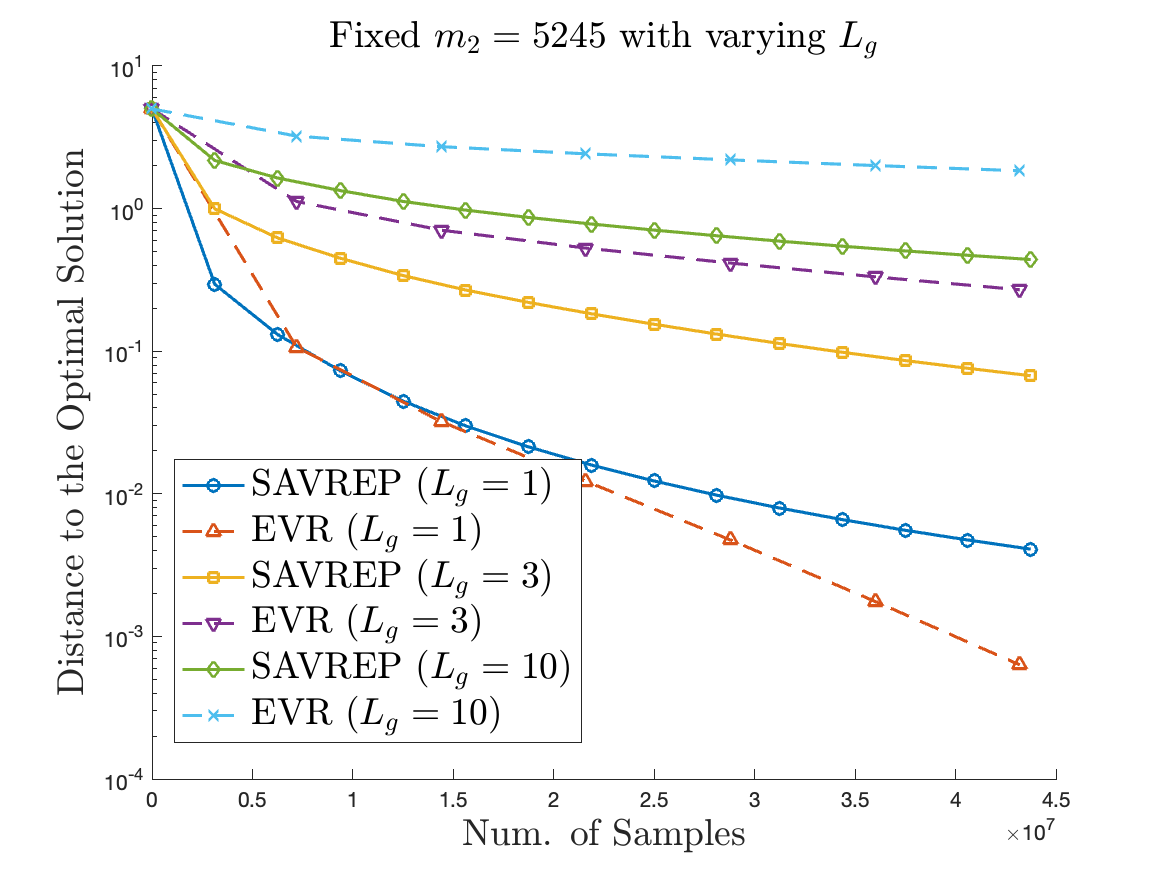}
  \end{minipage}
  \begin{minipage}[h!]{0.33\linewidth}
    \centering
    \includegraphics[scale=0.3]{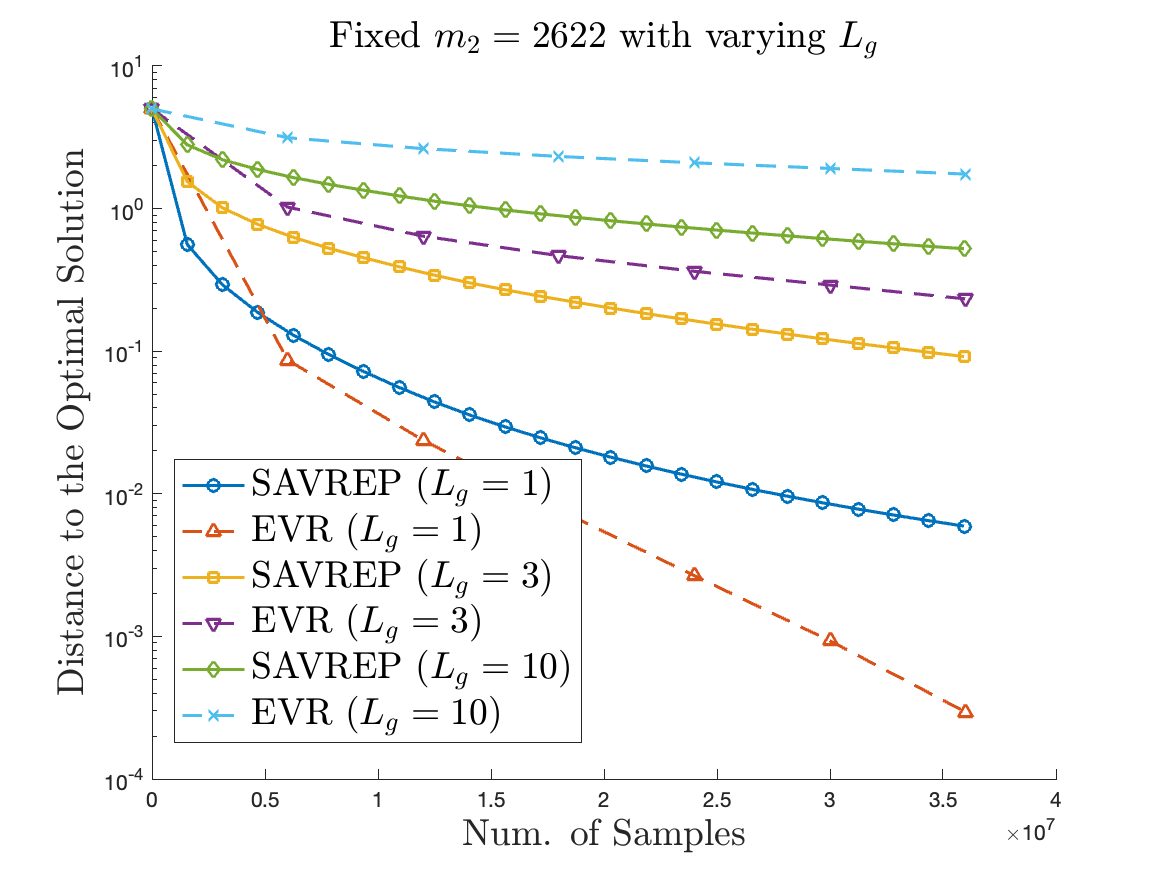}
  \end{minipage}
  \begin{minipage}[h!]{0.33\linewidth}
    \centering
    \includegraphics[scale=0.3]{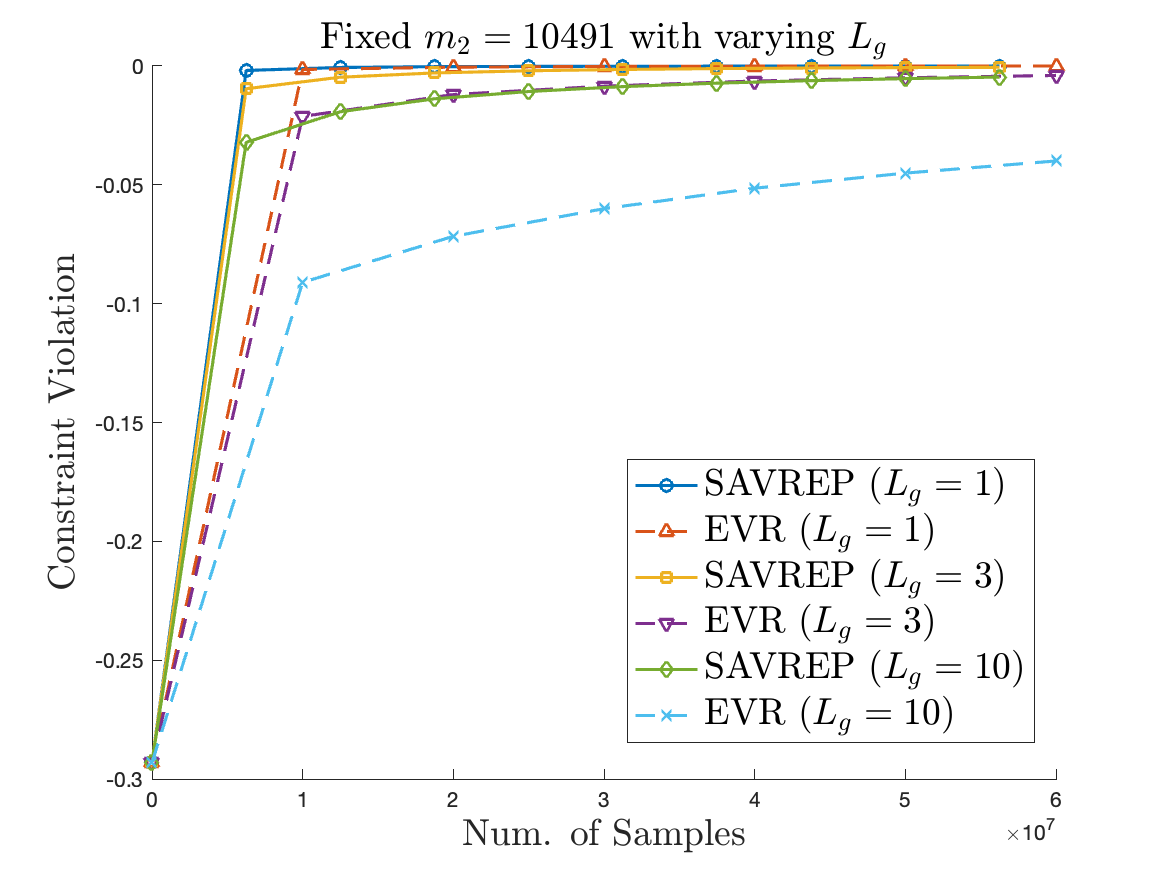}
  \end{minipage}%
  \begin{minipage}[h!]{0.33\linewidth}
    \centering
    \includegraphics[scale=0.3]{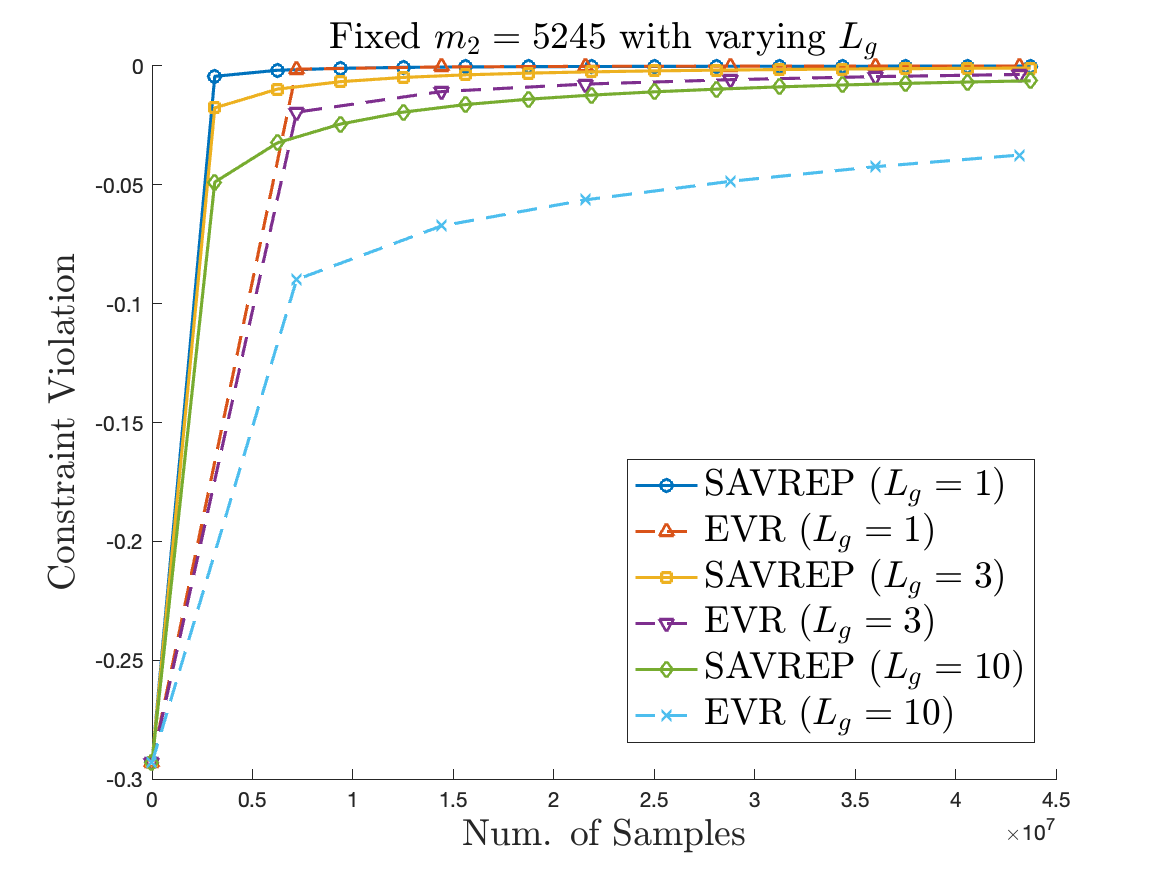}
  \end{minipage}
  \begin{minipage}[h!]{0.33\linewidth}
    \centering
    \includegraphics[scale=0.3]{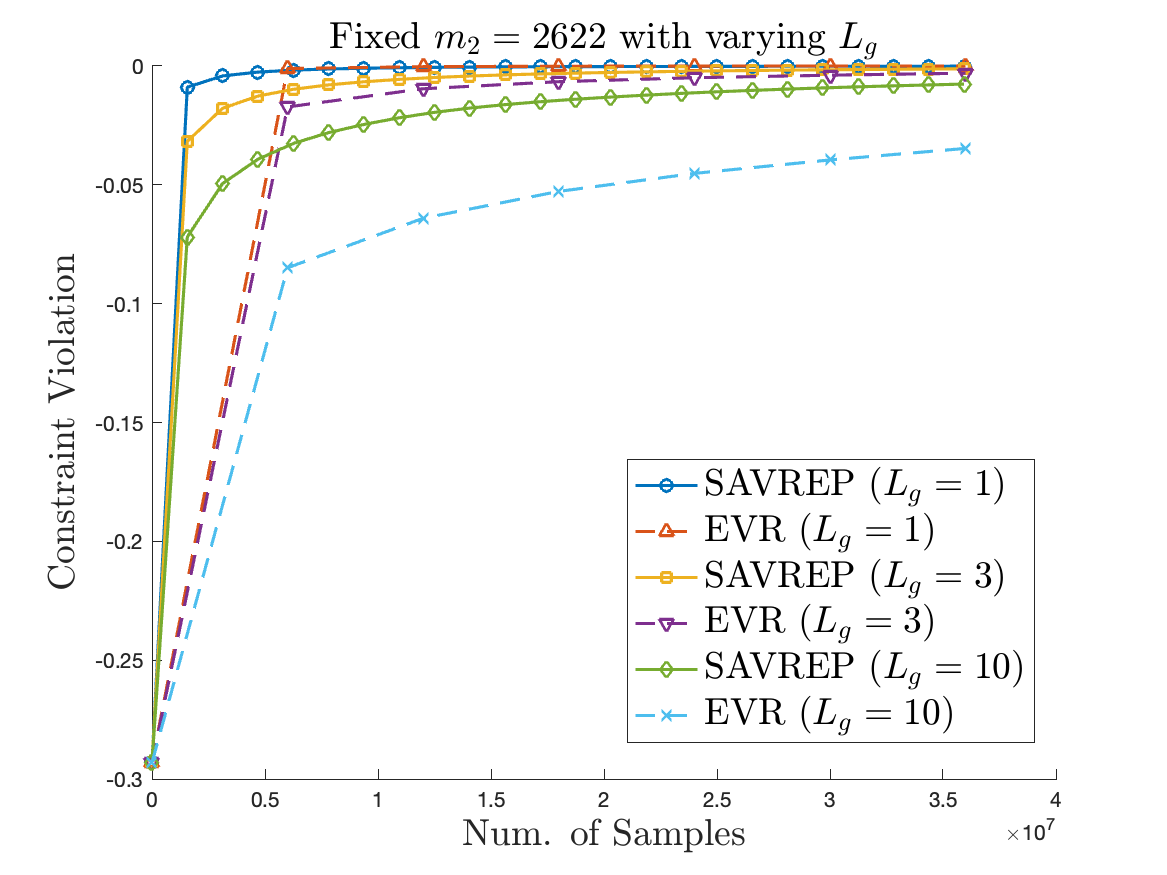}
  \end{minipage}
  \begin{minipage}[h!]{0.33\linewidth}
    \centering
    \includegraphics[scale=0.3]{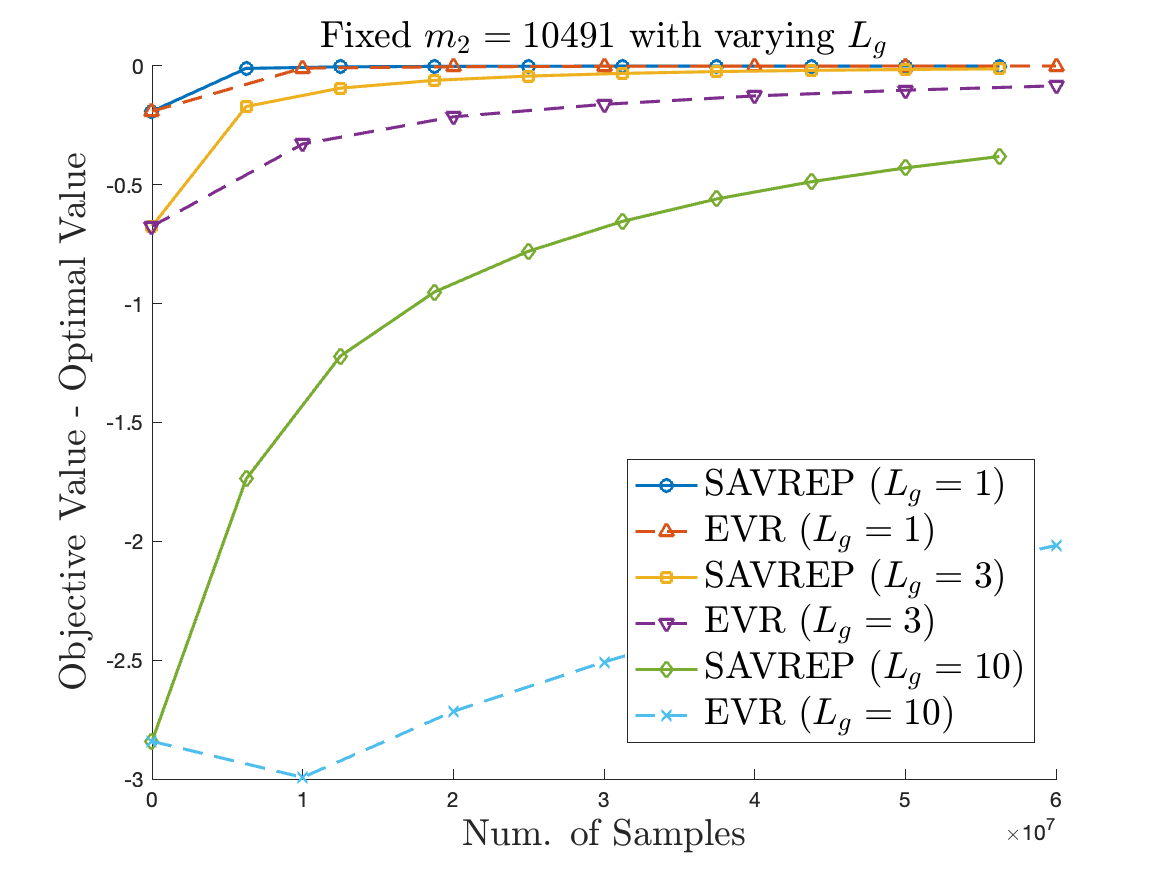}
  \end{minipage}
  \begin{minipage}[h!]{0.33\linewidth}
    \centering
    \includegraphics[scale=0.3]{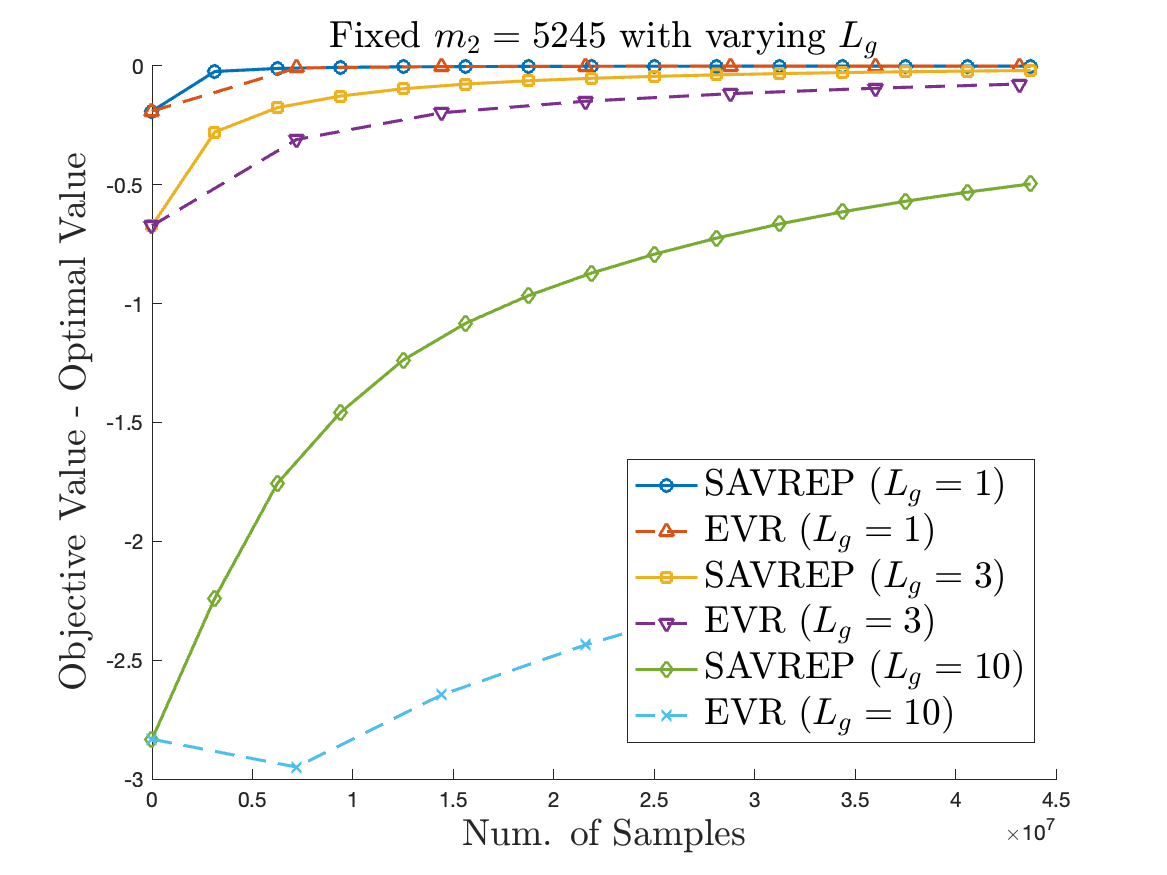}
  \end{minipage}
  \begin{minipage}[h!]{0.33\linewidth}
    \centering
    \includegraphics[scale=0.3]{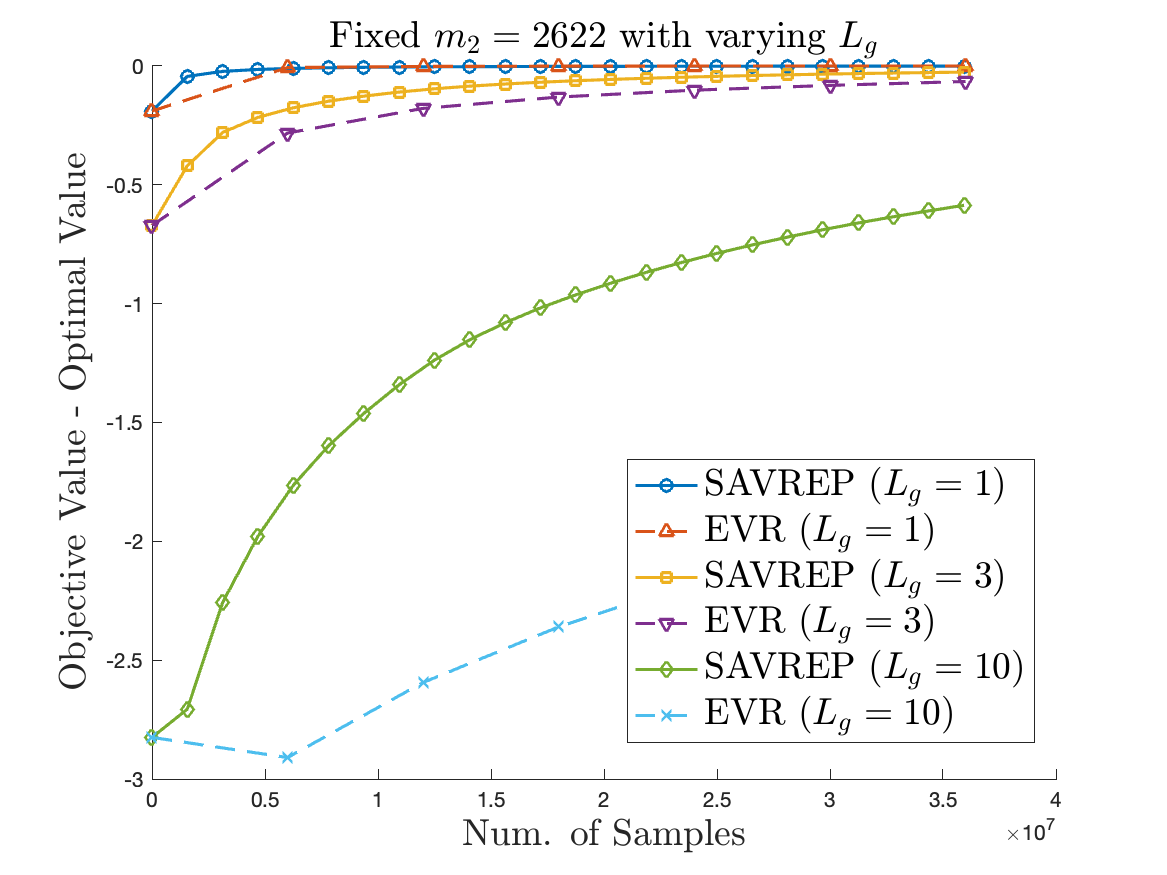}
  \end{minipage}
  \caption{Convergence of SAVREP-m: distance to the optimal solution (first row), constraint violation (second row), and objective function gap (third row); $m=10491$ (first column), $m=5245$ (second column), $m=2622$ (third column).}
  \label{fig:savrep-m-conv-1-fix-m-1-vary-L}
\end{figure}

In {this set of experiments}, we test SAVREP-m on the same problem \eqref{NP-classification}, using the HVI reformulation without perturbation \eqref{HVI-form-finite-sum-const-opt}. {The parameter-tuning follows the same process as described in Section~\ref{sec:numerical-savrep}, where the values of the learning rate are summarized in Appendix~\ref{app:parameters-numerical}.} The loss function is defined as the logistic loss function, i.e.\,$\phi(t)=\log(1+\exp(-t))$, with $\lambda=5$ and $r_1=0.4$. 
{
The convergence results are given in Figure~\ref{fig:savrep-m-conv-1-fix-m-1-vary-L} and Figure~\ref{fig:savrep-m-conv-1-fix-L-vary-m}, showing the distance to optimal solution in the first row, constraint violation in the second row, and objective function gap in the third row, respectively. 

\begin{figure}[h!]
  \begin{minipage}[h!]{0.33\linewidth}
    \centering
    \includegraphics[scale=0.3]{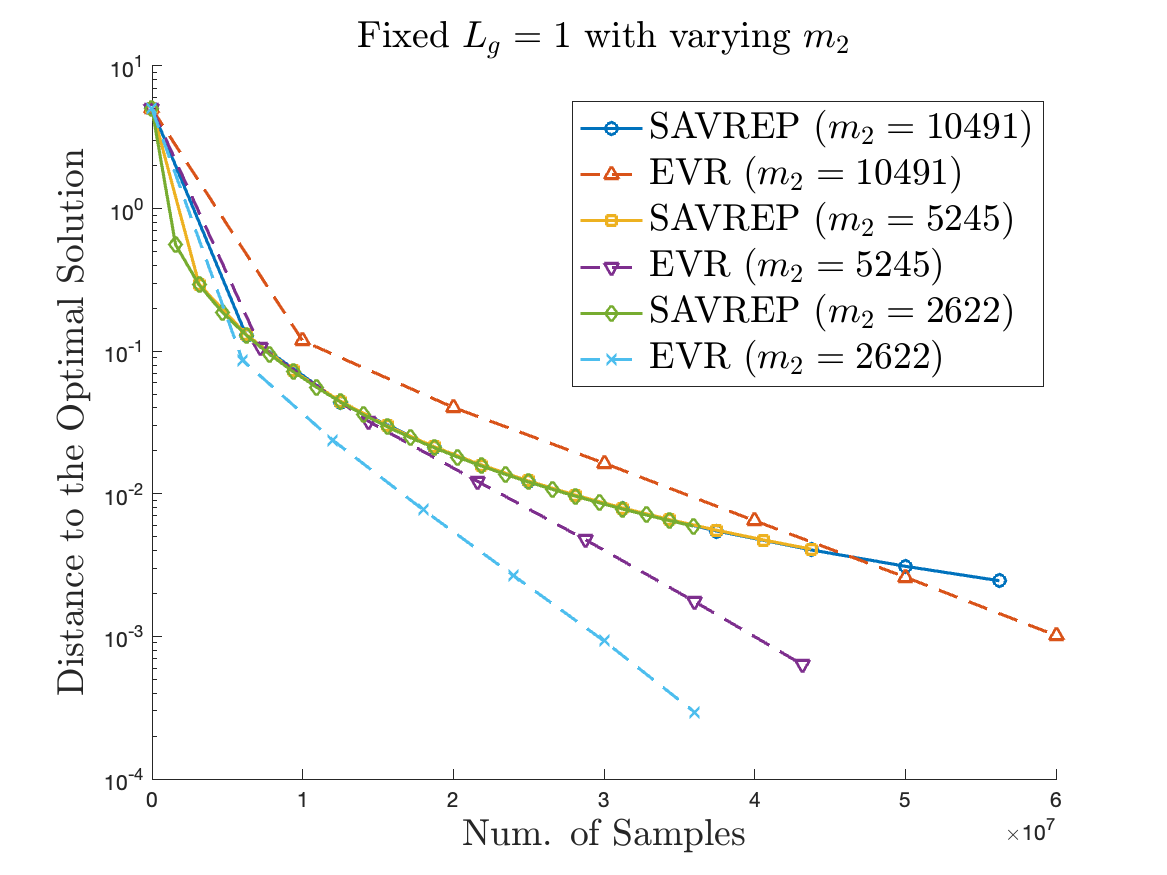}
  \end{minipage}
  \begin{minipage}[h!]{0.33\linewidth}
    \centering
    \includegraphics[scale=0.3]{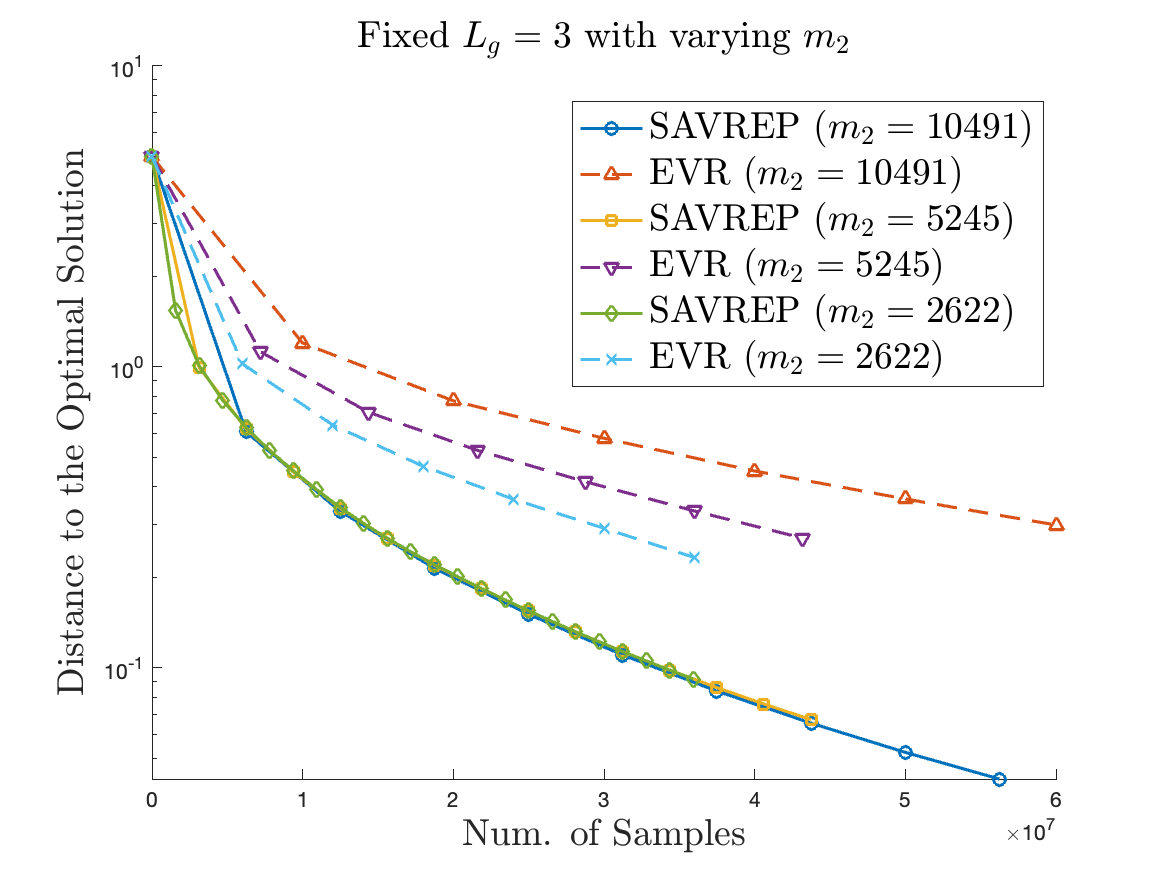}
  \end{minipage}
  \begin{minipage}[h!]{0.33\linewidth}
    \centering
    \includegraphics[scale=0.3]{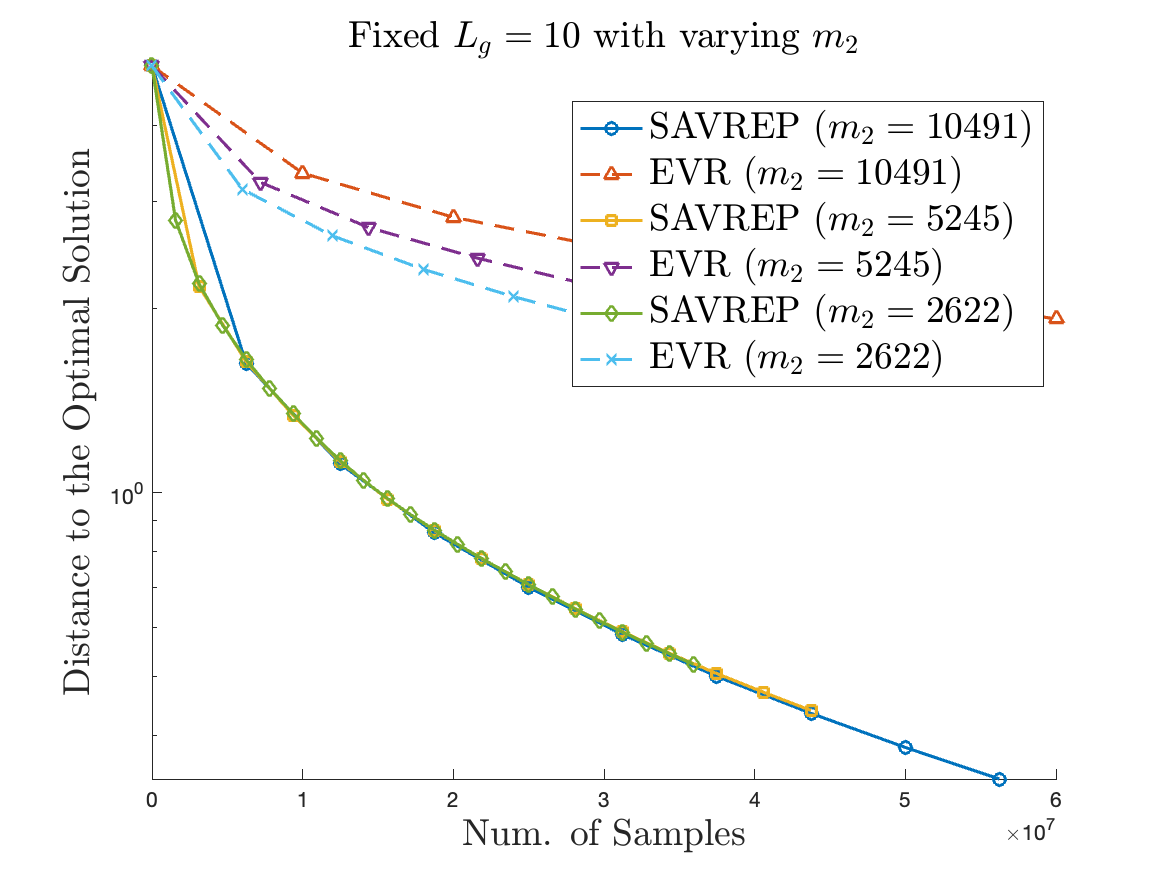}
  \end{minipage}
  \begin{minipage}[h!]{0.33\linewidth}
    \centering
    \includegraphics[scale=0.3]{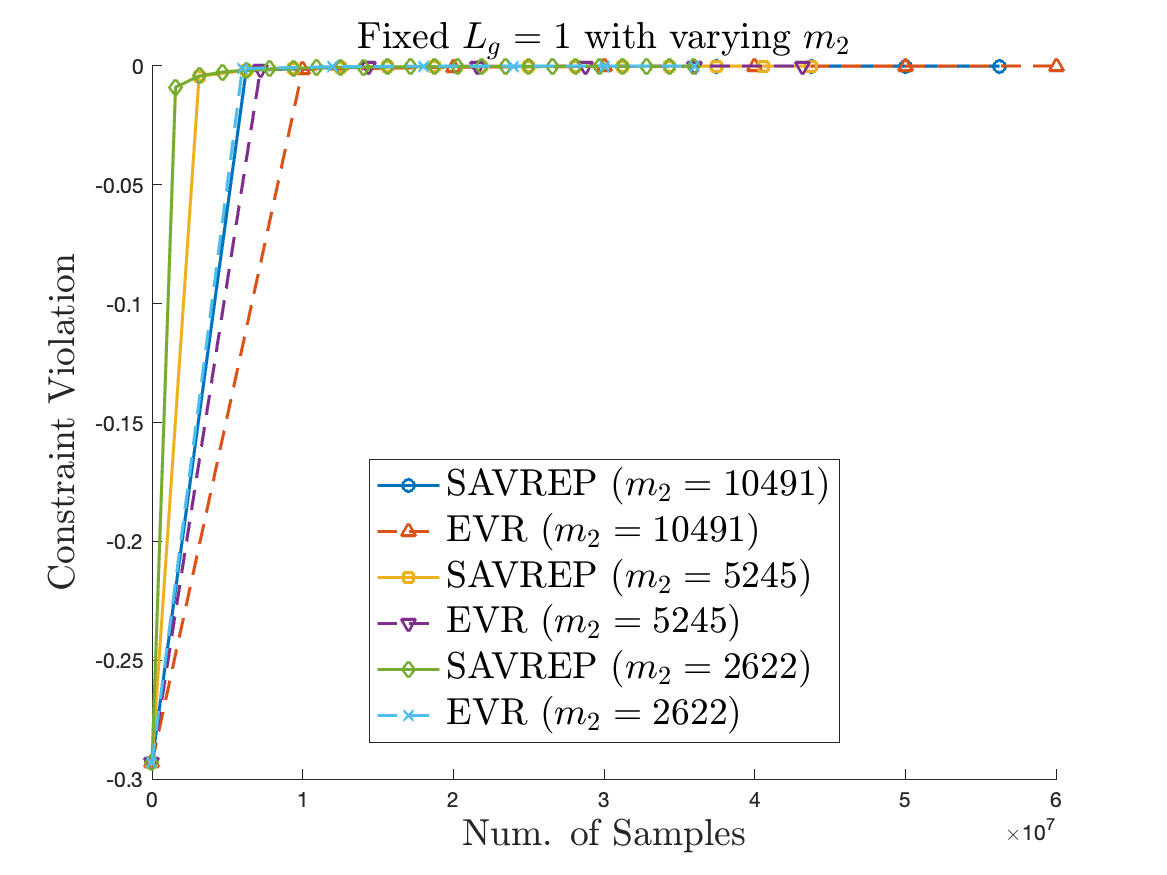}
  \end{minipage}%
  \begin{minipage}[h!]{0.33\linewidth}
    \centering
    \includegraphics[scale=0.3]{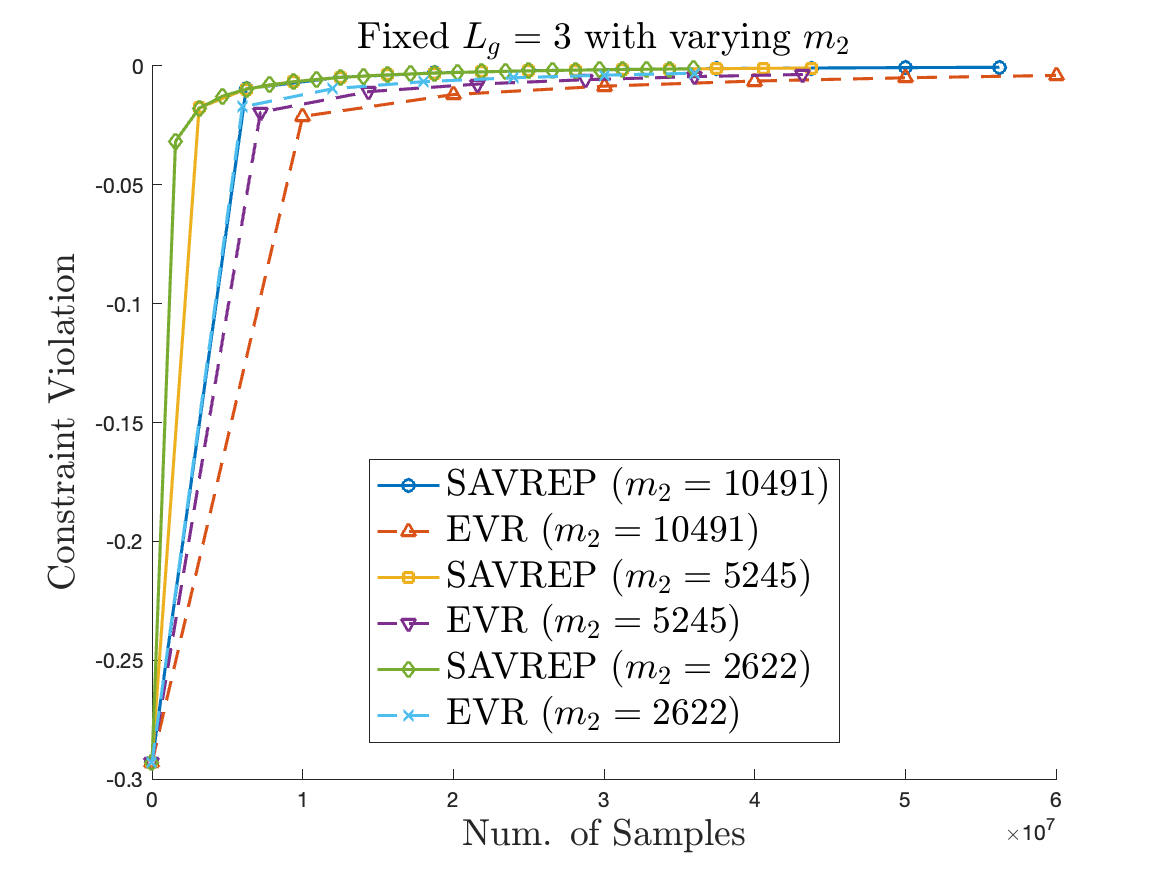}
  \end{minipage}
  \begin{minipage}[h!]{0.33\linewidth}
    \centering
    \includegraphics[scale=0.3]{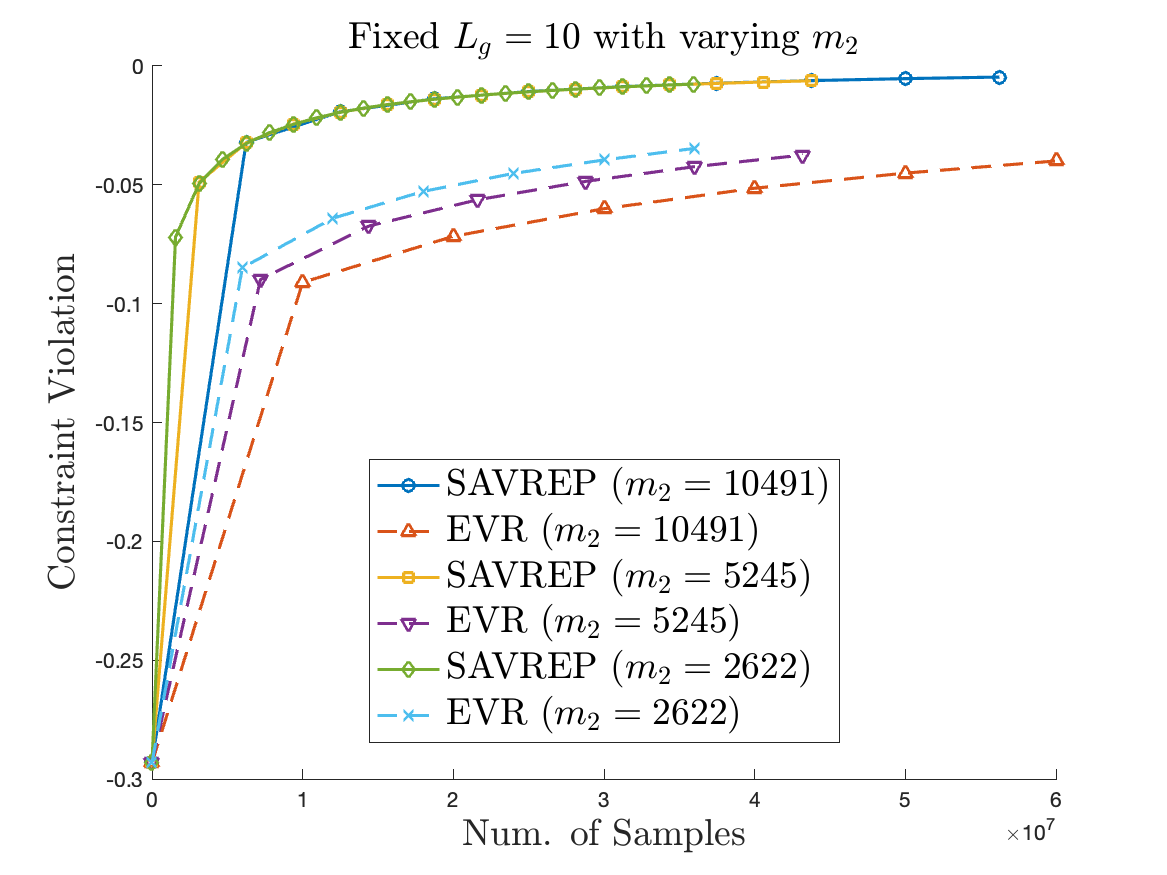}
  \end{minipage}
  \begin{minipage}[h!]{0.33\linewidth}
    \centering
    \includegraphics[scale=0.3]{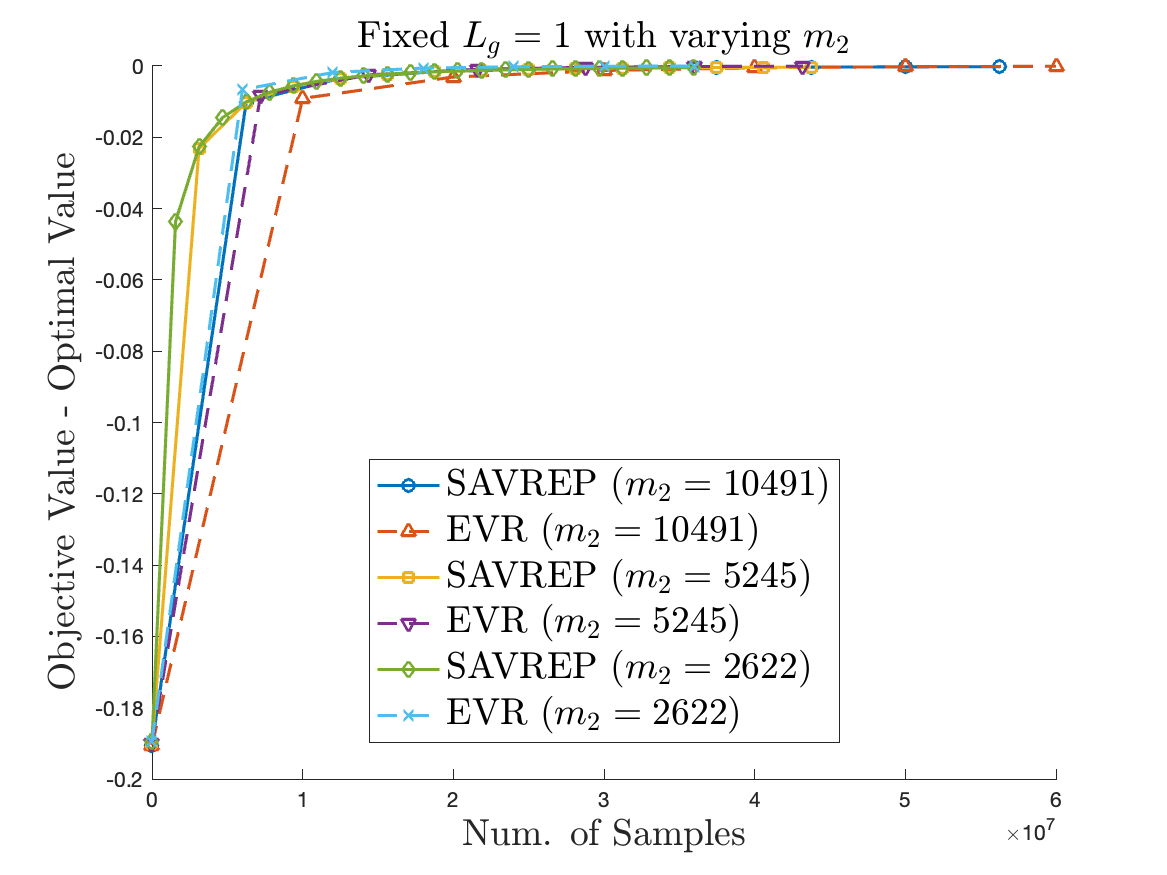}
  \end{minipage}
  \begin{minipage}[h!]{0.33\linewidth}
    \centering
    \includegraphics[scale=0.3]{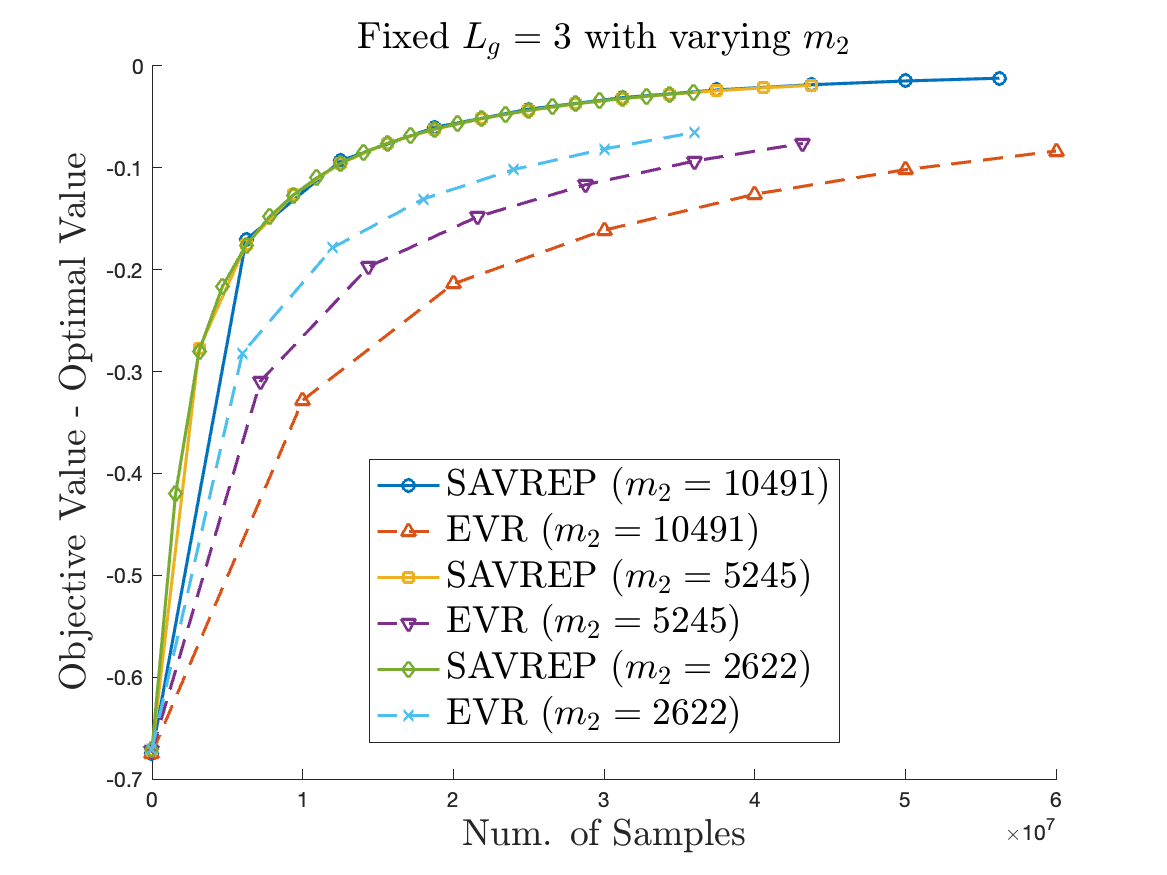}
  \end{minipage}
  \begin{minipage}[h!]{0.33\linewidth}
    \centering
    \includegraphics[scale=0.3]{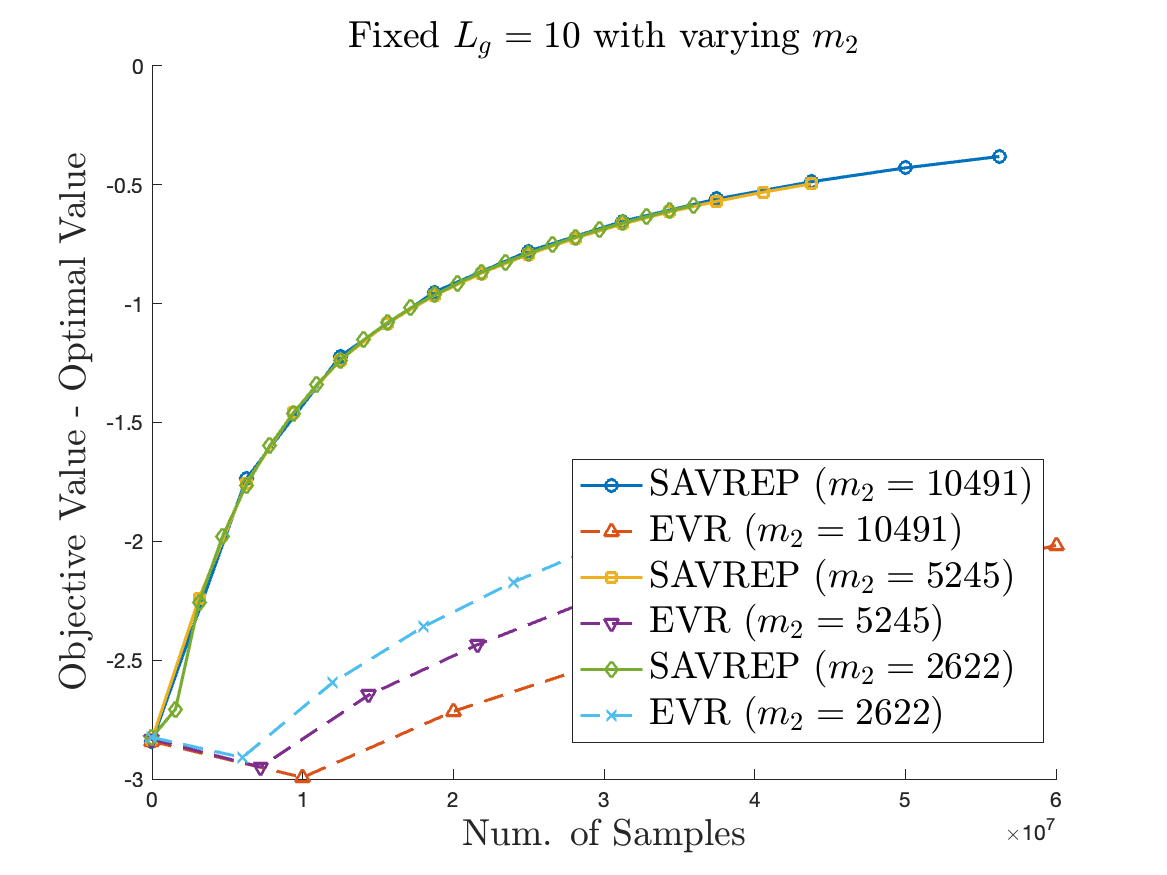}
  \end{minipage}
  \caption{Convergence of SAVREP-m: distance to the optimal solution (first row), constraint violation (second row), and objective function gap (third row); $L_g=1$ (first column), $L_g=3$ (second column), $L_g=10$ (third column).}
  \label{fig:savrep-m-conv-1-fix-L-vary-m}
\end{figure}

In Figure~\ref{fig:savrep-m-conv-1-fix-m-1-vary-L}, each column represents the results from experiments where $m_2$ is fixed at different values ($\{10491,5245,2622\}$ from left to right) with $L_g$ varying among $\{1,3,10\}$. Similar to the results for SAVREP in Figure~\ref{fig:savrep-conv-fix-m-1-vary-L}, we see that the the changes in $L_g$ bring conspicuous impact to the convergence speed. The major difference from the previous results lies in the fact that now EVR can perform better than SAVREP-m when $L_g$ is smaller (e.g.\,$L_g=1$), and that is specifically obvious when $m_2$ is also smaller (top right in Figure~\ref{fig:savrep-m-conv-1-fix-m-1-vary-L}). On the other hand, such advantage in the convergence rate observed from EVR vanishes when $L_g$ increases. The proposed SAVREP-m starts to outperform EVR when $L_g$ is increased to 3 and 10, and such performance gap grows even larger as $m_2$ increases from 2622 (right) to 10491 (left). Similar trends can be observed for constraint violation (second row) and objective value gap (third row), where the advantages for SAVREP-m stand out the most in experiments with $L_g=10$ and especially for $m_2=10491$. These results show that indeed EVR is more sensitive to the change of $L_g$ with a dependency of $\frac{L_g}{\epsilon}$ compared to $\sqrt{\frac{L_g}{\epsilon}}$ for SAVREP-m. The term $m_2\sqrt{\frac{1}{\epsilon}}$ in the complexity bound for SAVREP-m can be the major reason of it generating slower convergence than EVR when $L_g$ is smaller. However, increasing $L_g$ or $m_2$ seems to make this term less dominant and allows SAVREP-m to outperform EVR again. Note that for the metric of distance to optimal solution, EVR demonstrates a convergence behavior that is closer to linear convergence rather than sublinear convergence, particularly for $L_g=1$. This is likely due to hidden/local strong monotonicity of the reformulated problem even without any perturbation. While EVR is more capable of adapting to such hidden problem structure, SAVREP-m is specifically designed for problem that is merely monotone by explicitly adopting diminishing step sizes. Such practice results in a better theoretical dependence on the parameter $L_g$ but can be less adaptive when applied to strongly monotone problems. On the other hand, the linear convergence behavior for EVR also becomes less obvious as $L_g$ increases, indicating a possibly weakening effect from strong monotonicity.
}

{
In Figure~\ref{fig:savrep-m-conv-1-fix-L-vary-m}, each column represents the results from experiments where $L_g$ is fixed at different values ($\{1,3,10\}$ from left to right) with $m_2$ varying among $\{2622,5245,10491\}$. Similar to Figure~\ref{fig:savrep-m-conv-1-fix-m-1-vary-L}, EVR shows linear convergence when $L_g=1$, due to possibly stronger effect of (hidden) strong monotonicity. As $L_g$ increases, such linear convergence deteriorates into sublinear convergence, and the overall performance is outperformed by SAVREP-m by an increasingly larger gap. Similar trends can be observed in the convergence in terms of constraint violation as well as objective value gap. In general, the impact on the convergence behavior from changing $m_2$ is rather limited for both methods, while the performance of EVR indeed can be more subject to the increase in the parameter $L_g$. On the other hand, if $L_g$ is fixed, then EVR is also more sensitive to the change in $m_2$, and this is likely due to the combined effect of $m_2$ and $L_g$ in the same term, where $L_g$ can enlarge the effect caused by $m_2$. These results indeed align closely with the theoretical bounds summarized in Table~\ref{tab:literature}.
}

\section{Conclusions}
\label{sec:conclusion}
In this paper, we propose two stochastic variance reduced algorithms, SAVREP-m and SAVREP, for solving the finite-sum HVI problem with gradient Lipschitz function. In particular, both the vector mapping and the function involved in the HVI problem are of finite-sum structure. By exploiting this specific problem structure together with the variance reduction techniques developed in the literature, the proposed algorithms are able to {achieve gradient complexities that match the optimal bounds given by~\cite{alacaoglu2022stochastic} for finite-sum VI, as well as the accelerated bounds given by $\mbox{Katyusha}^{\scriptsize \mbox{ns}}$~\cite{allen2017katyusha} for finite-sum optimization.}
We show that an application of finite-sum optimization with finite-sum inequality constraints can be reformulated into the specific finite-sum HVI problem discussed in this paper, where the proposed schemes can be readily applied to. Preliminary numerical results are also provided to verify the convergence of our schemes, while demonstrating the merits of including variance reduction for the finite-sum function when solving this kind of finite-sum HVI problem.

{Finally, we note that while the proposed methods match the bounds for $\mbox{Katyusha}^{\scriptsize \mbox{ns}}$ when the HVI specializes to a pure (finite-sum) optimization problem, it is only optimal for strongly convex problems but suboptimal for convex problems. On the other hand, Varag proposed in~\cite{lan2019unified} has achieved an $m_2\log m_2$ complexity instead of $m_2\sqrt{\frac{1}{\epsilon}}$ given by $\mbox{Katyusha}^{\scriptsize \mbox{ns}}$ for solving convex finite-sum optimization. It remains to be open questions how these improvements can be made to achieve near-optimal complexity given by Varag or even optimal complexity suggested by the lower bound~\cite{lan2018optimal, woodworth2016tight} in the context of finite-sum HVI problem under the convex/monotone setting. We leave these interesting research questions to future work.}

\smallskip
 \noindent
{ {\bf Acknowledgment: } The authors would like to thank the two anonymous referees for their insightful 
comments that greatly helped improve this paper.}

\printbibliography

\begin{appendices}
\section{Proofs of some Technical Results}

{
\subsection{Proof of Lemma~\ref{lem:savrep-m-VI-bd}}
\label{proof:savrep-m-VI-bd}
The optimality conditions at $x^{k+0.5}$ and $x^{k+1}$ yield
\begin{eqnarray}
&&\langle\gamma_k\left({H}(w^k)+\tilde\nabla {g}(y^k)\right)+x^{k+0.5}-\Bar{x}^k,x-x^{k+0.5}\rangle\ge0,\quad\forall x\in\mathcal{Z}, \label{VI-opt-1-m}\\
&&\langle\gamma_k\left(\hat {H}(x^{k+0.5})+\tilde\nabla {g}(y^k)\right)+x^{k+1}-\Bar{x}^k,x-x^{k+1}\rangle\ge0,\quad\forall x\in\mathcal{Z}.\label{VI-opt-2-m}
\end{eqnarray}
{
From \eqref{VI-opt-2-m} and the definition of $\Bar{x}^k$ in~\eqref{savrep-m-Update}:
\begin{eqnarray}
&&\frac{1}{2}\left(\|x^{k+1}-x\|^2+(1-p_1)\|x^{k+1}-x^k\|^2-(1-p_1)\|x^k-x\|^2+p_1\|x^{k+1}-w^k\|^2-p_1\|w^k-x\|^2\right)\nonumber\\
&=&(1-p_1)\langle x^{k+1}-x^k,x^{k+1}-x\rangle+p_1\langle x^{k+1}-w^k,x^{k+1}-x\rangle\nonumber\\
&=&\langle x^{k+1}-\bar x^k,x^{k+1}-x\rangle\nonumber\\
&\le&\gamma_k\langle\hat{{H}}(x^{k+0.5})+\tilde\nabla {g}(y^k),x-x^{k+1}\rangle.\label{VI-bd-0-m}
\end{eqnarray}
To further upper bound~\eqref{VI-bd-0-m}, we first use the following identity:
}
\begin{eqnarray}
&&\gamma_k\langle\hat{{H}}(x^{k+0.5})+\tilde\nabla {g}(y^k),x-x^{k+1}\rangle\nonumber\\
&=& \gamma_k\langle H(x^{k+0.5})+\tilde\nabla {g}(y^k),x-x^{k+0.5}\rangle+\gamma_k\langle\hat{{H}}(x^{k+0.5})-H(x^{k+0.5}),x-x^{k+0.5}\rangle\nonumber\\
&&+\gamma_k\langle {H}(w^k)+\tilde\nabla {g}(y^k),x^{k+0.5}-x^{k+1}\rangle+\gamma_k\langle\hat{{H}}(x^{k+0.5})-{H}(w^k),x^{k+0.5}-x^{k+1}\rangle.\label{VI-bd-1-m}
\end{eqnarray}
The third term in~\eqref{VI-bd-1-m} can be bounded by using \eqref{VI-opt-1-m} with $x=x^{k+1}$:
\begin{eqnarray}
&&\gamma_k\langle {H}(w^k)+\tilde\nabla {g}(y^k),x^{k+0.5}-x^{k+1}\rangle\le \langle x^{k+0.5}-\Bar{x}^k,x^{k+1}-x^{k+0.5}\rangle\nonumber\\
&=& (1-p_1)\langle x^{k+0.5}-x^k,x^{k+1}-x^{k+0.5}\rangle+p_1\langle x^{k+0.5}-w^k,x^{k+1}-x^{k+0.5}\rangle\nonumber\\
&=& \frac{1}{2}\left(-\|x^{k+1}-x^{k+0.5}\|^2+(1-p_1)\|x^{k+1}-x^k\|^2-(1-p_1)\|x^{k+0.5}-x^k\|^2\right.\nonumber\\
&&\left.+p_1\|x^{k+1}-w^k\|^2-p_1\|x^{k+0.5}-w^k\|^2\right),\nonumber
\end{eqnarray}
while the fourth term in~\eqref{VI-bd-1-m} can be bounded by (definition of $\hat H(x^{k+0.5})$ given~\eqref{VR-grad-H}):
\begin{eqnarray}
&&\gamma_k\langle\hat{{H}}(x^{k+0.5})-{H}(w^k),x^{k+0.5}-x^{k+1}\rangle=\gamma_k\langle {H}_{\xi_k}(x^{k+0.5})-{H}_{\xi_k}(w^k),x^{k+0.5}-x^{k+1}\rangle\nonumber\\
&\le& \gamma_k^2\|{H}_{\xi_k}(x^{k+0.5})-{H}_{\xi_k}(w^k)\|^2+\frac{1}{4}\|x^{k+0.5}-x^{k+1}\|^2.\nonumber
\end{eqnarray}
Combining the above two inequalities with \eqref{VI-bd-1-m}, we get: 
\begin{eqnarray}
&&\gamma_k\langle\hat{{H}}(x^{k+0.5})+\tilde\nabla {g}(y^k),x-x^{k+1}\rangle\nonumber\\
&\le& \gamma_k\langle H(x^{k+0.5})+\tilde\nabla {g}(y^k),x-x^{k+0.5}\rangle+\gamma_k\langle\hat {H}(x^{k+0.5})-H(x^{k+0.5}),x-x^{k+0.5}\rangle\nonumber\\
&&+\frac{1}{2}\left(2\gamma_k^2\|{H}_{\xi_k}(x^{k+0.5})-{H}_{\xi_k}(w^k)\|^2+(1-p_1)\|x^{k+1}-x^k\|^2-\frac{1}{2}\|x^{k+1}-x^{k+0.5}\|^2\right.\nonumber\\
&&\left.-(1-p_1)\|x^{k+0.5}-x^k\|^2+p_1\|x^{k+1}-w^k\|^2-p_1\|x^{k+0.5}-w^k\|^2\right)\nonumber
\end{eqnarray}
Further combining the above inequality with \eqref{VI-bd-0-m}, we get:
\begin{eqnarray}
&&\underbrace{\frac{1}{2}\left(\|x^{k+1}-x\|^2-(1-p_1)\|x^k-x\|^2-p_1\|w^k-x\|^2+\left(1-p_1\right)\|x^{k+0.5}-x^k\|^2\right)}_{:=d_k(x)}\nonumber\\
&&+\gamma_k\langle H(x^{k+0.5})+\tilde\nabla {g}(y^k),x^{k+0.5}-x\rangle\nonumber\\
&\le& \underbrace{\frac{1}{2}\left(2\gamma_k^2\|{H}_{\xi_k}(x^{k+0.5})-{H}_{\xi_k}(w^k)\|^2-p_1\|x^{k+0.5}-w^k\|^2-\frac{1}{2}\|x^{k+1}-x^{k+0.5}\|^2\right)}_{:=e_{k1}(x)}\nonumber\\
&&+\underbrace{\gamma_k\langle\hat {H}(x^{k+0.5})-H(x^{k+0.5}),x-x^{k+0.5}\rangle}_{:=e_{k2}(x)}\label{d_k-def}
\end{eqnarray}
Rearranging the terms with the monotonicity of $H(\cdot)$, we obtain:
\begin{eqnarray}
&&\gamma_k\langle H(x)+\tilde\nabla {g}(y^k),x^{k+0.5}-x\rangle\nonumber\\
&\le& \gamma_k\langle H(x^{k+0.5})+\tilde\nabla {g}(y^k),x^{k+0.5}-x\rangle\le-d_k(x)+e_{k1}(x)+e_{k2}(x),\label{VI-bd-last-m}
\end{eqnarray}
which completes the proof.

\subsection{Proof of Lemma~\ref{lem:savrep-m-g-bd}}
\label{proof:savrep-m-g-bd}
{
Using the {Lipschitz} continuity of $g(\cdot)$ and the definition of $v^{k+1}$ and $y^{k}$ in~\eqref{savrep-m-Update}:
\begin{eqnarray}
g(v^{k+1})&\le& g(y^k)+\langle \nabla g(y^k),v^{k+1}-y^k\rangle+\frac{L_g}{2}\|v^{k+1}-y^k\|^2\nonumber\\
&=& g(y^k)+\langle \nabla g(y^k),(1-\alpha_k-\beta_k)v^k+\alpha_k x^{k+0.5}+\beta_k \bar w^k-y^k\rangle+\frac{L_g\alpha_k^2}{2}\|x^{k+0.5}-x^k\|^2\nonumber\\
&=& (1-\alpha_k-\beta_k)\left(g(y^k)+\langle \nabla g(y^k),v^k-y^k\rangle\right)+\alpha_k \left(g(y^k)+\langle \nabla g(y^k),x^{k+0.5}-y^k\rangle\right)\nonumber\\
&&+\beta_k\left(g(y^k)+\langle \nabla g(y^k),\bar w^k-y^k\rangle\right)+\frac{L_g\alpha_k^2}{2}\|x^{k+0.5}-x^k\|^2.\nonumber
\end{eqnarray}
By the convexity of $g$ we have $g(y^k)+\langle\nabla g(y^k),v^k-y^k\rangle\le g(v^k)$ and $g(y^k)\le g(x)-\langle \nabla g(y^k),x-y^k\rangle$, then the above inequality can be further bounded by:
\begin{eqnarray}
g(v^{k+1})
&\le& (1-\alpha_k-\beta_k)g(v^k)+\alpha_k \left(g(x)+\langle \nabla g(y^k),x^{k+0.5}-x\rangle\right)\nonumber\\
&&+\beta_k\left(g(y^k)+\langle \nabla g(y^k),\bar w^k-y^k\rangle\right)+\frac{L_g\alpha_k^2}{2}\|x^{k+0.5}-x^k\|^2\nonumber\\
&=& (1-\alpha_k-\beta_k)g(v^k)+\beta_k\left(g(y^k)+\langle \nabla g(y^k),\bar w^k-y^k\rangle\right)+\frac{L_g\alpha_k^2}{2}\|x^{k+0.5}-x^k\|^2\nonumber\\
&&+\alpha_k \left(g(x)+\langle \tilde\nabla {g}(y^k),x^{k+0.5}-x\rangle+\langle \nabla g(y^k)-\tilde\nabla g(y^k),x^{k+0.5}-x\rangle\right)\label{g-ineq-1-m},
\end{eqnarray}
Now let us define:
\begin{eqnarray}
    e_{k3}(x)&:=& \langle \nabla g(y^k)-\tilde\nabla g(y^k),x^{k+0.5}-x\rangle - \frac{\beta_k}{\alpha_k}\left(g(\bar w^k)-g(y^k)-\langle\nabla g(y^k),\bar w^k-y^k\rangle\right)\nonumber\\
    &&-\frac{\alpha_kL_g}{\beta_k}\|x^{k+0.5}-x^k\|^2.\label{error-3-m}
\end{eqnarray}
Then by adding and subtracting $\alpha_ke_{k3}(x)$ on RHS,~\eqref{g-ineq-1-m} can be written as:
\begin{eqnarray}
    g(v^{k+1})&\le& (1-\alpha_k-\beta_k)g(v^k)+\beta_k\left(g(y^k)+\langle \nabla g(y^k),\bar w^k-y^k\rangle\right)+\frac{L_g\alpha_k^2}{2}\|x^{k+0.5}-x^k\|^2\nonumber\\
    &&+\alpha_k \left(g(x)+\langle \tilde\nabla {g}(y^k),x^{k+0.5}-x\rangle\right)\nonumber\\
    &&+\beta_k\left(g(\bar w^k)-g(y^k)-\langle\nabla g(y^k),\bar w^k-y^k\rangle\right)+\frac{\alpha_k^2L_g}{\beta_k}\|x^{k+0.5}-x^k\|^2+\alpha_k e_{k3}(x)\nonumber\\
    &=& (1-\alpha_k-\beta_k)g(v^k)+\alpha_k g(x)+\beta_k g(\bar w^k)+\alpha_k\langle\tilde \nabla g(y^k),x^{k+0.5}-x\rangle\nonumber\\
    &&+\left(\frac{\alpha_k^2L_g}{2}+\frac{\alpha_k^2L_g}{\beta_k}\right)\|x^{k+0.5}-x^k\|^2+\alpha_k e_3(x)\nonumber
\end{eqnarray}
Subtracting $g(x)$ from both sides we obtain:
\begin{eqnarray}
    g(v^{k+1})-g(x)&\le& (1-\alpha_k-\beta_k)\left(g(v^k)-g(x)\right)+\beta_k\left(g(\bar w^k)-g(x)\right)+\alpha_k\langle\tilde \nabla g(y^k),x^{k+0.5}-x\rangle\nonumber\\
    &&+\left(\frac{\alpha_k^2L_g}{2}+\frac{\alpha_k^2L_g}{\beta_k}\right)\|x^{k+0.5}-x^k\|^2+\alpha_k e_{k3}(x),\label{g-bd-last-m}
\end{eqnarray}
which completes the proof.
}

\subsection{Proof of Lemma~\ref{lem:savrep-m-inner-relation}}
\label{proof:savrep-m-inner-relation}

{
By the definition of the function $Q(x';x)$ in~\eqref{Q-function} and the iterate $v^{k+1}$ in~\eqref{savrep-m-Update}, we first obtain the following identity:
\begin{eqnarray*}
    && Q(v^{k+1};x) =\langle H(x),v^{k+1}-x\rangle+g(v^{k+1})-g(x)\nonumber\\
&=& (1-\alpha_k-\beta_k)\langle H(x),v^k-x\rangle+\alpha_k\langle H(x),x^{k+0.5}-x\rangle+\beta_k\langle H(x),\bar w^k-x\rangle
+g(v^{k+1})-g(x).
\end{eqnarray*}
Applying the bound for $g(v^{k+1})-g(x)$ in Lemma~\ref{lem:savrep-m-g-bd}, the above identity is bounded by
\begin{eqnarray*}
    &&(1-\alpha_k-\beta_k)\langle H(x),v^k-x\rangle+\alpha_k\langle H(x),x^{k+0.5}-x\rangle+\beta_k\langle H(x),\bar w^k-x\rangle
+g(v^{k+1})-g(x)\\
&\le& (1-\alpha_k-\beta_k)\langle H(x),v^k-x\rangle+\alpha_k\langle H(x),x^{k+0.5}-x\rangle+\beta_k\langle H(x),\bar w^k-x\rangle\nonumber\\
&&+ (1-\alpha_k-\beta_k)\left(g(v^k)-g(x)\right)+\beta_k \left(g(\bar w^k)-g(x)\right)\nonumber\\
&&+\alpha_k\langle\tilde \nabla g(y^k),x^{k+0.5}-x\rangle+{\left(\frac{\alpha_k^2L_g}{2}+\frac{\alpha_k^2L_g}{\beta_k}\right)}\|x^{k+0.5}-x^k\|^2+\alpha_k e_{k3}(x)\\
&=& (1-\alpha_k-\beta_k)\left(\langle H(x),v^k-x\rangle+g(v^k)-g(x)\right)+\beta_k\left(\langle H(x),\bar w^k-x\rangle+g(\bar w^k)-g(x)\right)\nonumber\\
&&+\alpha_k\left(\langle H(x)+\tilde\nabla g(y^k),x^{k+0.5}-x\rangle+\left(\frac{\alpha_k L_g}{2}+\frac{\alpha_k L_g}{\beta_k}\right)\|x^{k+0.5}-x^k\|^2+e_{k3}(x)\right),
\end{eqnarray*}
where in the equality the terms with same coefficients $(1-\alpha_k-\beta_k)$, $\beta_k$, $\alpha_k$, are combined. In view of the definition of $Q(x';x)$ in~\eqref{Q-function}, the terms associated with the coefficient $(1-\alpha_k-\beta_k)$ can be replaced with $Q(v^k;x)$, and the terms associated with the coefficient $\beta_k$ can be replaced with $Q(\bar w^k;x)$. In addition, the term $\langle H(x)+\tilde\nabla g(y^k),x^{k+0.5}-x\rangle$ can be bounded using the result from Lemma~\ref{lem:savrep-m-VI-bd}, which results in the following upper bounds:
\begin{eqnarray*}
    &&(1-\alpha_k-\beta_k)\left(\langle H(x),v^k-x\rangle+g(v^k)-g(x)\right)+\beta_k\left(\langle H(x),\bar w^k-x\rangle+g(\bar w^k)-g(x)\right)\nonumber\\
&&+\alpha_k\left(\langle H(x)+\tilde\nabla g(y^k),x^{k+0.5}-x\rangle+\left(\frac{\alpha_k L_g}{2}+\frac{\alpha_k L_g}{\beta_k}\right)\|x^{k+0.5}-x^k\|^2+e_{k3}(x)\right)\\
&\le& (1-\alpha_k-\beta_k)Q(v^k;x)+\beta_k Q(\bar w^k;x)\nonumber\\
&&+\alpha_k\left(\frac{1}{\gamma_k}(-d_k(x)+e_{k1}(x)+e_{k2}(x))+\left(\frac{\alpha_k L_g}{2}+\frac{\alpha_k L_g}{\beta_k}\right)\|x^{k+0.5}-x^k\|^2+e_{k3}(x)\right)\nonumber\nonumber\\
&=& (1-\alpha_k-\beta_k)Q(v^k;x)+\beta_k Q(\bar w^k;x)+\alpha_k\left(\frac{1}{\gamma_k}e_{k1}(x)+\frac{1}{\gamma_k}e_{k2}(x)+e_{k3}(x)\right)\nonumber\\
&&-\frac{\alpha_k}{2\gamma_k}\left(1-p_1-\alpha_k\gamma_k L_g-\frac{\alpha_k\gamma_k 2L_g}{\beta}\right)\|x^{k+0.5}-x^k\|^2\nonumber\\
&&+\frac{\alpha_k}{2\gamma_k}\left((1-p_1)\|x^k-x\|^2+p_1\|w^k-x\|^2-\|x^{k+1}-x\|^2\right),
\end{eqnarray*}
where in the equality we use the definition for $d_k(x)$ given in~\eqref{d_k-def}, combined with the rest of the terms. With the condition $1-p_1-\alpha_k\gamma_k L_g-\frac{2\alpha_k \gamma_k L_g}{\beta_k}\ge0$ specified in~\eqref{const-3}, the upper bound for $Q(v^{k+1};x)$ can be summarized as follows:
\begin{eqnarray*}
    Q(v^{k+1};x)&\le& (1-\alpha_k-\beta_k)Q(v^k;x)+\beta_k Q(\bar w^k;x)+\alpha_k\left(\frac{1}{\gamma_k}e_{k1}(x)+\frac{1}{\gamma_k}e_{k2}(x)+e_{k3}(x)\right)\nonumber\\
&&+\frac{\alpha_k}{2\gamma_k}\left((1-p_1)\|x^k-x\|^2+p_1\|w^k-x\|^2-\|x^{k+1}-x\|^2\right).
\end{eqnarray*}
Finally, in order to obtain the bound in the form of~\eqref{mu-equal-zero-bd}, we define
\begin{eqnarray}
    e_{k4}(x) := \|w^{k+1}-x\|^2-p_1\|x^{k+1}-x\|^2-(1-p_1)\|w^k-x\|^2\nonumber
\end{eqnarray}
and add $\frac{\alpha_k}{2\gamma_k}e_{k4}(x)$ on both sides of the above inequality. In particular, the term $\frac{\alpha_k}{2\gamma_k}e_{k4}(x)$ on the RHS is combined with the terms $e_{k1},e_{k2},e_{k3}$, while moving the term $-\frac{\alpha_k}{2\gamma_k}\|x^{k+1}-x\|^2$ from right to left and moving the term $-\frac{\alpha_k}{2\gamma_k}(1-p_1)\|w^{k}-x\|^2$ (due to $\frac{\alpha_k}{2\gamma_k}e_{k4}(x)$) from left to right. The final bound thus becomes
\begin{eqnarray}
    &&Q(v^{k+1};x)+\frac{\alpha_k}{2\gamma_k}\left(\|w^{k+1}-x\|^2+(1-p_1)\|x^{k+1}-x\|^2\right)\nonumber\\
    &\le& (1-\alpha_k-\beta_k)Q(v^k;x)+\beta_k Q(\bar w^k;x)+\frac{\alpha_k}{2\gamma_k}\left(\|w^k-x\|^2+(1-p_1)\|x^k-x\|^2\right)\nonumber\\
    &&+\frac{\alpha_k}{\gamma_k}\left(e_{k1}(x)+e_{k2}(x)+\gamma_ke_{k3}(x)+\frac{1}{2}e_{k4}(x)\right),\nonumber
\end{eqnarray}
which completes the proof.
}


}

{
\subsection{An Upper Bound for Expectation of Maximum}
In the proofs that follow, we will use the following lemma that provides an upper bound for the expectation of maximum:
\begin{lemma}\label{lem:exp-max}
    Let $\mathcal{F}=(\mathcal{F}_k)_{k\ge0}$ be a filtration and $(u^k)$ be stochastic process adapted to $\mathcal{F}$ with $\EE[u^{k+1}\,|\,\mathcal{F}_k]=0$. Then for any $K\in\mathbb{N}$ and $x^0\in\ZZ$ where $\ZZ$ is a compact set,
    \begin{eqnarray}
        \EE\left[\max\limits_{x\in\ZZ}\,\sum\limits_{k=0}^{K-1}\langle u^{k+1},x\rangle\right]\le\frac{c}{2}\max\limits_{x\in\ZZ}\,\|x-x^0\|^2+\frac{1}{2c}\sum\limits_{k=0}^{K-1}\EE\left[\|u^{k+1}\|^2\right]\nonumber
    \end{eqnarray}
    where $c>0$ is some arbitrary constant. 
\end{lemma}
Lemma~\ref{lem:exp-max} is also used in the analysis in work such as of~\cite{nemirovski2009robust, alacaoglu2022stochastic}, and the proof can be found in the appendix of~\cite{alacaoglu2022stochastic}, which we shall omit here.
}

{
\subsection{Technical Lemma for Establishing the Proof of Lemma~\ref{lem:error-upp-bd}}

In Lemma~\ref{lem:error-upp-bd}, we aim to provide a constant upper bound (depending on problem parameters) on the ``stochastic error'' term $\EE\left[\max\limits_{x\in\ZZ}\left\{\sum\limits_{s=0}^{S-1}\sum\limits_{k=sm_2}^{(s+1)m_2-1}\frac{\alpha_{sm_2}}{\Gamma_{s+1}\gamma_{sm_2}}\bar e_k(x)\right\}\right]$. Since the term $\bar e_k(x)$ consists of several terms (see Lemma~\ref{lem:savrep-m-inner-relation} and the references therein for detailed definitions), the proof of such a bound can be decomposed into establishing bounds for each of the components, where the proof for each one of them can be lengthy on its own. Therefore, we divide the proof for Lemma~\ref{lem:error-upp-bd} into several parts, where the bound for each of the consisting component is summarized in the next lemma. Furthermore, in order to relieve the notation burden, we define the following succinct notation and use them whenever they are applicable in later proofs. In particular, we define:
\begin{eqnarray}
    \sum\limits_{s,k}:=\sum\limits_{s=0}^{S-1}\sum\limits_{k=sm_2}^{(s+1)m_2-1},\quad\phi_s:=\frac{\alpha_{sm_2}}{\Gamma_{s+1}\gamma_{sm_2}},\quad\varphi_s:=\frac{\alpha_{sm_2}}{\Gamma_{s+1}}.\label{simple-notation-1}
\end{eqnarray}
The bound to be proven in Lemma~\ref{lem:error-upp-bd} can then be written as:
\begin{eqnarray}
    \EE\left[\max\limits_{x\in\ZZ}\left\{\sum\limits_{s,k}\phi_s\bar e_k(x)\right\}\right]\le \frac{1}{2}(S_2+S_3+S_4)\Omega_\mathcal{Z}^2.\label{error-upp-bd-succ}
\end{eqnarray}
The lemma below will be used to establish the bound~\eqref{error-upp-bd-succ}.

    

\begin{lemma}
    \label{lem:bd-ek}
    Following the same assumptions as in Lemma~\ref{lem:error-upp-bd}, the following inequalities holds:
    \begin{enumerate}
        \item 
        \begin{eqnarray}
        \EE\left[\max\limits_{x\in\ZZ}\sum\limits_{s,k}\phi_se_{k1}(x)\right]\le \EE\left[\sum\limits_{s,k}\frac{\phi_s}{2}\left(\left(2\gamma_{sm_2}^2L_h^2-p_1\right)\|x^{k+0.5}-w^k\|^2-\frac{1}{2}\|x^{k+1}-x^{k+0.5}\|^2\right)\right].\label{e_1-exp-err}
        \end{eqnarray}
        \item 
        \begin{eqnarray}
        \EE\left[\max\limits_{x\in\ZZ}\sum\limits_{s,k}\phi_se_{k2}(x)\right]\le \frac{S_2}{2}\Omega_Z^2+\sum\limits_{s,k}\frac{\phi_s}{2}\EE\left[\frac{\gamma_{sm_2}^2L_h^2}{4}\|x^{k+0.5}-w^k\|^2\right].\label{e2-bd-2}
        \end{eqnarray}

        \item 
    \begin{eqnarray}
        \EE\left[\max\limits_{x\in\ZZ}\sum\limits_{s,k}\varphi_se_{k3}(x)\right]\le \frac{S_3}{2}\Omega_\ZZ^2.\label{error-3-2}
    \end{eqnarray}
    \item 
    \begin{eqnarray}
        \EE\left[\max\limits_{x\in\ZZ}\sum\limits_{s,k}\frac{\phi_s}{2}e_{k4}(x)\right]\le \frac{S_4}{2}\Omega_\ZZ^2+\sum\limits_{s,k}\frac{\phi_s}{2}p_1(1-p_1)\EE\left[\frac{1}{2}\|x^{k+1}-x^{k+0.5}\|^2+\frac{1}{2}\|x^{k+0.5}-w^k\|^2\right].\label{error-4-3}
    \end{eqnarray}
    \end{enumerate}
    
\end{lemma}



In the follows we shall prove the four bounds in Lemma~\ref{lem:bd-ek} one by one. 
\begin{enumerate}
    \item To prove the first bound~\eqref{e_1-exp-err}, note that by definition of $e_{k1}(x)$ (see Lemma~\ref{lem:savrep-m-VI-bd}),
\[
e_{k1}(x)=\frac{1}{2}\left(2\gamma_k^2\|{H}_{\xi_k}(x^{k+0.5})-{H}_{\xi_k}(w^k)\|^2-p_1\|x^{k+0.5}-w^k\|^2-\frac{1}{2}\|x^{k+1}-x^{k+0.5}\|^2\right)
\]
is in fact a constant in terms of $x$. Furthermore, by applying the tower property, we have
\begin{eqnarray}
\EE\left[\|{H}_{\xi_k}(x^{k+0.5})-{H}_{\xi_k}(w^k)\|^2\right]=\EE\left[\EE_{k_1}\left[\|{H}_{\xi_k}(x^{k+0.5})-{H}_{\xi_k}(w^k)\|^2\right]\right]\le \EE\left[L_h^2\|x^{k+0.5}-w^k\|^2\right],\label{tower-lipschitz}
\end{eqnarray}
with $\EE_{k_1}[\cdot]:=\EE_{\xi_k}[\cdot|x^k,w^k]$ defined in~\eqref{expectations-1} and the stochastic oracle $H_{\xi_k}(\cdot)$ defined in~\eqref{sample-prob}. Therefore, we can conclude that
\begin{eqnarray}
    &&\EE\left[\max\limits_{x\in\ZZ}\sum\limits_{s,k}\phi_se_{k1}(x)\right]\nonumber\\
    &=&\EE\left[\sum\limits_{s,k}\frac{\phi_s}{2}\left(2\gamma_{sm_2}^2\|{H}_{\xi_k}(x^{k+0.5})-{H}_{\xi_k}(w^k)\|^2-p_1\|x^{k+0.5}-w^k\|^2-\frac{1}{2}\|x^{k+1}-x^{k+0.5}\|^2\right)\right]\nonumber\\
    &\le& \EE\left[\sum\limits_{s,k}\frac{\phi_s}{2}\left(\left(2\gamma_{sm_2}^2L_h^2-p_1\right)\|x^{k+0.5}-w^k\|^2-\frac{1}{2}\|x^{k+1}-x^{k+0.5}\|^2\right)\right],\nonumber
\end{eqnarray}
completing the proof for the first bound~\eqref{e_1-exp-err}.
    \item Now, let us consider the second bound~\eqref{e2-bd-2}. By directly plugging in the definition for $e_{k2}(x)$ (Lemma~\ref{lem:savrep-m-VI-bd}), we first obtain
\begin{eqnarray}
    \EE\left[\max\limits_{x\in\ZZ}\sum\limits_{s,k}\phi_se_{k2}(x)\right]= \EE\left[\max\limits_{x\in\ZZ}\left\{\sum\limits_{s,k}\phi_s\gamma_k\langle\hat {H}(x^{k+0.5})-H(x^{k+0.5}),x-x^{k+0.5}\rangle\right\}\right].\nonumber
\end{eqnarray}
We note that the inner product $\langle\hat {H}(x^{k+0.5})-H(x^{k+0.5}),x^{k+0.5}\rangle$ is independent of $x$, and by applying the tower property we obtain
\[
\EE\left[\langle\hat {H}(x^{k+0.5})-H(x^{k+0.5}),x^{k+0.5}\rangle\right]=\EE\left[\EE_{k_1}\left[\langle\hat {H}(x^{k+0.5})-H(x^{k+0.5}),x^{k+0.5}\rangle\right]\right]=0,
\]
since $x^{k+0.5}$ is deterministic with respect to $\EE_{k_1}[\cdot]$ and $\EE_{k_1}\left[\hat H(x^{k+0.5})\right]=H(x^{k+0.5})$ (see~\eqref{VR-grad-H} for definition). In addition, since for any fixed $s\ge0$, $\gamma_k$ remains constant throughout the iterations $k=sm_2,...,(s+1)m_2-1$, the coefficient can be combined as
\[
\phi_s\gamma_k=\frac{\alpha_{sm_2}}{\Gamma_{s+1}\gamma_{sm_2}}\gamma_k=\frac{\alpha_{sm_2}}{\Gamma_{s+1}}=\varphi_s.
\]
In view of the above observations, we obtain
\begin{eqnarray}
    &&\EE\left[\max\limits_{x\in\ZZ}\sum\limits_{s,k}\phi_se_{k2}(x)\right]=\EE\left[\max\limits_{x\in\ZZ}\left\{\sum\limits_{s,k}\varphi_s\langle\hat {H}(x^{k+0.5})-H(x^{k+0.5}),x\rangle\right\}\right]\nonumber\\
    &\le& \frac{S_2}{2}\max\limits_{x\in\ZZ}\|x-x^0\|^2+\frac{1}{2S_2}\sum\limits_{s,k}\EE\left[\left\|\varphi_s\left(\hat {H}(x^{k+0.5})-H(x^{k+0.5})\right)\right\|^2\right],\label{bd-e2-1.1}
\end{eqnarray}
where we apply Lemma~\ref{lem:exp-max} in the inequality and define $S_2 = \frac{4\alpha_{(S-1)m_2}}{\Gamma_{S}\gamma_{(S-1)m_2}}$. To complete the bound in the final form in~\eqref{e2-bd-2}, we upper bound the second term above by first applying the definition of $\hat H
(x^{k+0.5})$:
\begin{eqnarray}
    &&\EE\left[\left\|\hat {H}(x^{k+0.5})-H(x^{k+0.5})\right\|^2\right]=\EE\left[\left\|{H}(w^k)+{H}_{\xi_k}(x^{k+0.5})-{H}_{\xi_k}(w^k)-H(x^{k+0.5})\right\|^2\right]\nonumber\\
    &\le& \EE\left[\|{H}_{\xi_k}(x^{k+0.5})-{H}_{\xi_k}(w^k)\|^2\right]\le \EE\left[L_h^2\|x^{k+0.5}-w^k\|^2\right]\label{tower-lipschitz-2},
\end{eqnarray}
where in the last inequality we apply the tower property same as in~\eqref{tower-lipschitz}. We then simplify the coefficients in the second term of~\eqref{bd-e2-1.1} as follows:
\begin{eqnarray*}
    \frac{\varphi_s^2L_h^2}{2S_2}\le \frac{\varphi_s^2L_h^2}{2}\cdot \frac{1}{4\phi_s}=\frac{\phi_s^2}{2}\cdot \frac{1}{4\phi_s}\cdot \gamma_{sm_2}^2L_h^2=\frac{\phi_s}{2}\cdot\frac{\gamma_{sm_2}^2L_h^2}{4}.
\end{eqnarray*}
Note that the inequality is due to the condition~\eqref{paramter-conditions}, which requires
\[
\phi_{s-1}\le\phi_s\le \phi_{S-1}=\frac{S_2}{4},\quad s=1,...,S-1.
\]
Combining the above bounds, together with Assumption~\ref{ass:monotone-const-bd} that bounds the compact set $\mathcal{Z}$ with $\Omega_\mathcal{Z}$, the next inequality follows from~\eqref{bd-e2-1.1}:
\begin{eqnarray}
    \EE\left[\max\limits_{x\in\ZZ}\sum\limits_{s,k}\phi_se_{k2}(x)\right]\le \frac{S_2}{2}\Omega_\mathcal{Z}^2+\sum\limits_{s,k}\frac{\phi_s}{2}\cdot\frac{\gamma_{sm_2}^2L_h^2}{4}\EE\left[\left\|x^{k+0.5}-w^k\right\|^2\right],\nonumber
\end{eqnarray}
completing the proof for~\eqref{e2-bd-2}.
    \item Now let us consider the third bound~\eqref{error-3-2}. Applying the definition (Lemma~\ref{lem:savrep-m-g-bd}), we have
\begin{eqnarray}
    &&\EE\left[\max\limits_{x\in\ZZ}\sum\limits_{s,k}\varphi_se_{k3}(x)\right]=\EE\left[\max\limits_{x\in\ZZ}\left\{\sum\limits_{s,k}\varphi_s\langle \nabla g(y^k)-\tilde\nabla g(y^k),x^{k+0.5}-x\rangle\right\}\right]\nonumber\\
    &&-\EE\left[\sum\limits_{s,k}\varphi_s \frac{\beta_{sm_2}}{\alpha_{sm_2}}\left(g(\bar w^k)-g(y^k)-\langle\nabla g(y^k),\bar w^k-y^k\rangle\right)\right]\nonumber\\
    &&-\EE\left[\sum\limits_{s,k}\varphi_s \frac{\alpha_{sm_2}L_g}{\beta_{sm_2}}\|x^{k+0.5}-x^k\|^2\right].\label{bd-e3-1.0}
\end{eqnarray}
We shall focus on establishing an upper bound for the first term, which eventually cancels out the rest of the two terms. We first split the first term in the following fashion:
\begin{eqnarray}
    &&\langle \nabla g(y^k)-\tilde\nabla g(y^k),x^{k+0.5}-x\rangle=\langle \nabla g(y^k)-\tilde\nabla g(y^k),x^{k+0.5}-x^k\rangle\nonumber\\
    &&+\langle \nabla g(y^k)-\tilde\nabla g(y^k),x^{k}\rangle+\langle \nabla g(y^k)-\tilde\nabla g(y^k),-x\rangle.\label{bd-e3-1.2}
\end{eqnarray}
In particular, the term $\langle \nabla g(y^k)-\tilde\nabla g(y^k),x^{k}\rangle$ is independent of $x$ and has expectation 0 due to the identity:
\begin{eqnarray}
    \EE\left[\langle \nabla g(y^k)-\tilde\nabla g(y^k),x^k\rangle\right]=\EE\left[\EE_{k_2}\left[\langle \nabla g(y^k)-\tilde\nabla g(y^k),x^k\rangle\right]\right]=0,\label{bd-e3-1.3}
\end{eqnarray}
since $x^k$ and $y^k$ are deterministic with respect to $\EE_{k_2}[\cdot]$, $\EE_{k2}\left[\tilde\nabla g(y^k)\right]=\nabla g(y^k)$ with $\EE_{k_2}[\cdot]:=\EE_{\zeta_k}[\cdot|x^k,\bar w^k,v^k]$ as defined in~\eqref{expectations-1}, and $\tilde\nabla g(y^k)$ is defined in~\eqref{VR-grad-g}. The term $\langle \nabla g(y^k)-\tilde\nabla g(y^k),x^{k+0.5}-x^k\rangle$ is bounded via Young's inequality:
\begin{eqnarray}
    \langle \nabla g(y^k)-\tilde\nabla g(y^k),x^{k+0.5}-x^k\rangle\le \frac{\beta_{sm_2}}{4\alpha_{sm_2}L_g}\|\nabla g(y^k)-\tilde\nabla g(y^k)\|^2+\frac{\alpha_{sm_2} L_g}{\beta_{sm_2}}\|x^{k+0.5}-x^k\|^2.\label{bd-e3-1.4}
\end{eqnarray}
Finally, we apply Lemma~\ref{lem:exp-max} to obtain
\begin{eqnarray}
    &&\EE\left[\max\limits_{x\in\ZZ}\left\{\sum\limits_{s,k}\varphi_s\langle \nabla g(y^k)-\tilde\nabla g(y^k),-x\rangle\right\}\right]\nonumber\\
    &\le & \frac{S_3}{2}\max\limits_{x\in\ZZ}\|x-x^0\|^2+\frac{1}{2S_3}\sum\limits_{s,k}\varphi_s^2\EE\left[\left\|\nabla g(y^k)-\tilde\nabla g(y^k)\right\|^2\right]\label{bd-e3-1.1},
\end{eqnarray}
where $S_3$ is defined as follows with condition~\eqref{para-condition-2} satisfied:
\begin{eqnarray}
    S_3=\frac{2\alpha_{(S-1)m_2}L_g}{\beta_{(S-1)m_2}}\cdot S
    ,\quad S_3\ge \frac{2\alpha^2_{sm_2}L_g}{\Gamma_{s+1}\beta_{sm_2}},\quad s=0,...,S-1.\label{condition-S3}
\end{eqnarray}
The coefficient in~\eqref{bd-e3-1.1} then can be simplified as
\begin{eqnarray}
    \frac{\varphi_s^2}{2S_3}=\frac{1}{2S_3}\left(\frac{\alpha_{sm_2}}{\Gamma_{s+1}}\right)^2= \left(\frac{\alpha_{sm_2}}{\Gamma_{s+1}}\right)\left(\frac{\alpha_{sm_2}}{2\Gamma_{s+1}S_3}\right)\le  \left(\frac{\alpha_{sm_2}}{\Gamma_{s+1}}\right)\left(\frac{\beta_{sm_2}}{4\alpha_{sm_2}L_g}\right)=\varphi_s\left(\frac{\beta_{sm_2}}{4\alpha_{sm_2}L_g}\right).\label{bd-e3-1.5}
\end{eqnarray}
Combining the bounds from~\eqref{bd-e3-1.2}-~\eqref{bd-e3-1.5}, we establish the next bound:
\begin{eqnarray*}
    &&\EE\left[\max\limits_{x\in\ZZ}\left\{\sum\limits_{s,k}\varphi_s\langle \nabla g(y^k)-\tilde\nabla g(y^k),x^{k+0.5}-x\rangle\right\}\right]\le\frac{S_3}{2}\max\limits_{x\in\ZZ}\|x-x^0\|^2\nonumber\\
    &&+\sum\limits_{s,k}\varphi_s\left(\frac{\beta_{sm_2}}{2\alpha_{sm_2}L_g}\right)\EE\left[\left\|\nabla g(y^k)-\tilde\nabla g(y^k)\right\|^2\right]+\sum\limits_{s,k}\varphi_s\frac{\alpha_{sm_2} L_g}{\beta_{sm_2}}\EE\left[\|x^{k+0.5}-x^k\|^2\right].\nonumber
\end{eqnarray*}
Note that the term $\EE\left[\left\|\nabla g(y^k)-\tilde\nabla g(y^k)\right\|^2\right]$ can be further bounded using the following inequality:
\begin{eqnarray}
&&\EE_{k_2}\left[\|\nabla g(y^k)-\tilde\nabla g(y^k)\|^2\right]= \EE_{k_2}\left[\|\nabla g_{\zeta_k}(\bar w^k)-\nabla g_{\zeta_k}(y^k)-\left(\nabla g(\bar w^k)-\nabla g(y^k)\right)\|^2\right]\nonumber\\
&\le& \EE_{k_2}\left[\|\nabla g_{\zeta_k}(\bar w^k)-\nabla g_{\zeta_k}(y^k)\|^2\right]= \sum\limits_{i=1}^{m_2}\frac{1}{\pi_i}\|\nabla g_{i}(\bar w^k)-\nabla g_{i}(y^k)\|^2\nonumber\\
&\leq & \sum\limits_{i=1}^{m_2}\frac{2L_{g(i)}}{\pi_i}\left(g_{i}(\bar w^k)-g_{i}(y^k)-\langle\nabla g_i(y^k),\bar w^k-y^k\rangle\right)\nonumber\\
&=&  2L_g\sum\limits_{i=1}^{m_2}\left(g_{i}(\bar w^k)-g_{i}(y^k)-\langle\nabla g_i(y^k),\bar w^k-y^k\rangle\right)\nonumber\\
&=& 2L_g\left(g(\bar w^k)-g(y^k)-\langle\nabla g(y^k),\bar w^k-y^k\rangle\right),\label{g-gradient-est-bd}
\end{eqnarray}
where the second inequality is from Theorem 2.1.5 in \cite{nesterov2003introductory}; $\pi_i$ is the sample probability given in~\eqref{sample-prob}, and $L_{g(i)}$ is the Lipscthiz-smooth constant for $g_i$ and $L_g:=\sum\limits_{i=1}^{m_2}L_{g(i)}$. Therefore, we obtain the next inequality:
\begin{eqnarray}
    &&\EE\left[\max\limits_{x\in\ZZ}\left\{\sum\limits_{s,k}\varphi_s\langle \nabla g(y^k)-\tilde\nabla g(y^k),x^{k+0.5}-x\rangle\right\}\right]\le\frac{S_3}{2}\max\limits_{x\in\ZZ}\|x-x^0\|^2\nonumber\\
    &&+\sum\limits_{s,k}\varphi_s\left(\frac{\beta_{sm_2}}{\alpha_{sm_2}}\right)\EE\left[g(\bar w^k)-g(y^k)-\langle\nabla g(y^k),\bar w^k-y^k\rangle\right]+\sum\limits_{s,k}\varphi_s\frac{\alpha_{sm_2} L_g}{\beta_{sm_2}}\EE\left[\|x^{k+0.5}-x^k\|^2\right].\nonumber
\end{eqnarray}
The above inequality combined with~\eqref{bd-e3-1.0} results in the desirable bound in~\eqref{error-3-2}:
\begin{eqnarray*}
    \EE\left[\max\limits_{x\in\ZZ}\sum\limits_{s,k}\varphi_se_{k3}(x)\right]\le \frac{S_3}{2}\max\limits_{x\in\ZZ}\|x-x^0\|^2.\nonumber
\end{eqnarray*}
    \item To prove the last bound in~\eqref{error-4-3}, we first note that by definition (Lemma~\ref{lem:savrep-m-inner-relation}),
\begin{eqnarray}
    e_{k4}(x)&=&\|w^{k+1}-x\|^2-p_1\|x^{k+1}-x\|^2-(1-p_1)\|w^k-x\|^2\nonumber\\
    &=& \|w^{k+1}\|^2-p_1\|x^{k+1}\|^2-(1-p_1)\|w^k\|^2-2\langle w^{k+1}-p_1x^{k+1}-(1-p_1)w^k,x\rangle.\nonumber
\end{eqnarray}
Since $\EE_{k_1+}[\|w^{k+1}\|^2]=p_1\|x^{k+1}\|^2+(1-p_1)\|w^k\|^2$ with $\EE_{k_1+}[\cdot]:=\EE[\cdot|w^k,x^{k+1}]$ as defined in~\eqref{expectations-2}, we have
\begin{eqnarray}
    \EE\left[\|w^{k+1}\|^2-p_1\|x^{k+1}\|^2-(1-p_1)\|w^k\|^2\right]=0.\nonumber
\end{eqnarray}
Therefore,
\begin{eqnarray}
    &&\EE\left[\max\limits_{x\in\ZZ}\left\{\sum\limits_{s,k}\frac{\phi_s}{2}e_{k4}(x)\right\}\right]= \EE\left[\max\limits_{x\in\ZZ}\left\{\sum\limits_{s,k}\phi_s\langle -(w^{k+1}-p_1x^{k+1}-(1-p_1)w^k),x\rangle\right\}\right]\nonumber\\
    &\le& \frac{S_4}{2}\max\limits_{x\in\ZZ}\|x-x^0\|^2+\frac{1}{2S_4}\sum\limits_{s,k}\phi_s^2\EE\left[\|w^{k+1}-p_1x^{k+1}-(1-p_1)w^k\|^2\right],\label{bd-e4-1.1}
\end{eqnarray}
where we applied Lemma~\ref{lem:exp-max} to derive the upper bound above, and we shall specify the constant $S_4$ later. To establish a further upper bound, we first note the following identity:
\begin{eqnarray}
    &&\EE\left[\|w^{k+1}-p_1x^{k+1}-(1-p_1)w^k\|^2\right]=\EE\left[\EE_{k_1+}\left[\|w^{k+1}-p_1x^{k+1}-(1-p_1)w^k\|^2\right]\right]\nonumber\\
    &=& \EE\left[\EE_{k_1+}\left[\|w^{k+1}-\EE_{k_1+}[w^{k+1}]\|^2\right]\right]=\EE\left[\EE_{k_1+}\left[\|w^{k+1}\|^2\right]-\left\|\EE_{k_1+}[w^{k+1}]\right\|^2\right]\nonumber\\
    &=& \EE\left[p_1\|x^{k+1}\|^2+(1-p_1)\|w^k\|^2-\|p_1x^{k+1}+(1-p_1)w^k\|^2\right]=p_1(1-p_1)\EE\left[\|x^{k+1}-w^k\|^2\right],\nonumber
\end{eqnarray}
followed by applying Young's inequality to obtain:
\begin{eqnarray}
    &&\EE\left[\|w^{k+1}-p_1x^{k+1}-(1-p_1)w^k\|^2\right]=p_1(1-p_1)\EE\left[\|x^{k+1}-w^k\|^2\right]\nonumber\\
    &\le& 2p_1(1-p_1)\EE\left[\|x^{k+1}-x^{k+0.5}\|^2+\|x^{k+0.5}-w^k\|^2\right].\nonumber
\end{eqnarray}
Now let us simplify the coefficients by substituting the choice for $S_4=\frac{4\alpha_{(S-1)m_2}}{\Gamma_{S}\gamma_{(S-1)m_2}}$, and we note that by condition~\eqref{paramter-conditions} we have $S_4\ge 4\phi_s$ for $s=0,...,S-1$, which results in the bound:
\begin{eqnarray}
    \frac{\phi_s^2}{2S_4}\le \frac{\phi_s}{8}.\nonumber
\end{eqnarray}
Combining the above results, the bound in~\eqref{bd-e4-1.1} becomes
\begin{eqnarray}
    \EE\left[\max\limits_{x\in\ZZ}\sum\limits_{s,k}\frac{\phi_s}{2}e_{k4}(x)\right]\le \frac{S_4}{2}\Omega_\ZZ^2+\sum\limits_{s,k}\frac{\phi_s}{2}p_1(1-p_1)\EE\left[\frac{1}{2}\|x^{k+1}-x^{k+0.5}\|^2+\frac{1}{2}\|x^{k+0.5}-w^k\|^2\right],\nonumber
\end{eqnarray}
which is exactly the desired bound.
\end{enumerate}

}

{
\subsection{Proof of Lemma~\ref{lem:error-upp-bd}}
\label{proof:error-upp-bd}




The bound to be proven in Lemma~\ref{lem:error-upp-bd} with simplified notation (see also the definition in~\eqref{simple-notation-1}) is:
\begin{eqnarray}
    \EE\left[\max\limits_{x\in\ZZ}\left\{\sum\limits_{s,k}\phi_s\bar e_k(x)\right\}\right]\le \frac{1}{2}(S_2+S_3+S_4)\Omega_\mathcal{Z}^2.\nonumber
\end{eqnarray}
Now since $\bar e_k(x):=e_{k1}(x)+e_{k2}(x)+\gamma_ke_{k3}(x)+\frac{1}{2}e_{k4}(x)$ with definition of each term given in Lemma~\ref{lem:savrep-m-VI-bd}, ~\ref{lem:savrep-m-g-bd}, and ~\ref{lem:savrep-m-inner-relation}, we can combine the conclusions from Lemma~\ref{lem:bd-ek} to derive an overall upper bound. First of all, the next inequality is immediate:
\begin{eqnarray}
     \EE\left[\max\limits_{x\in\ZZ}\left\{\sum\limits_{s,k}\phi_s\bar e_k(x)\right\}\right]&\le&  \EE\left[\max\limits_{x\in\ZZ}\left\{\sum\limits_{s,k}\phi_s e_{k1}(x)\right\}\right]+\EE\left[\max\limits_{x\in\ZZ}\left\{\sum\limits_{s,k}\phi_s e_{k2}(x)\right\}\right]\nonumber\\
     &&+\EE\left[\max\limits_{x\in\ZZ}\left\{\sum\limits_{s,k}\varphi_s e_{k3}(x)\right\}\right]+\EE\left[\max\limits_{x\in\ZZ}\left\{\sum\limits_{s,k}\frac{\phi_s}{2} e_{k4}(x)\right\}\right].\nonumber
\end{eqnarray}
Here we simply sum up all the bounds established in Lemma~\ref{lem:bd-ek}. In particular, we arrange the terms and the sum of these bounds results in
\begin{eqnarray}
    &&\frac{1}{2}\left(S_2+S_3+S_4\right)\Omega_\ZZ^2+\sum\limits_{s,k}\frac{\phi_s}{2}\EE\left[\frac{1}{2}(p_1(1-p_1)-1)\|x^{k+1}-x^{k+0.5}\|^2\right]\nonumber\\
     &&+\sum\limits_{s,k}\frac{\phi_s}{2}\EE\left[\left(\frac{9}{4}\gamma_{sm_2}^2L_h^2-\frac{p_1}{2}-\frac{p_1^2}{2}\right)\|x^{k+0.5}-w^k\|^2\right].\nonumber
\end{eqnarray}
In view of condition~\eqref{para-condition-2} and the fact that $p_1\in[0,1]$, the last two terms are non-positive, concluding the desired bound
\begin{eqnarray*}
    \EE\left[\max\limits_{x\in\ZZ}\left\{\sum\limits_{s,k}\phi_s\bar e_k(x)\right\}\right]\le\frac{1}{2}\left(S_2+S_3+S_4\right)\Omega_\ZZ^2.
\end{eqnarray*}

}

\subsection{Proof of Lemma \ref{lem:VI-relation-1}}
\label{proof:VI-relation-1}
The proof for Lemma~\ref{lem:VI-relation-1} is similar to {the proof given in Appendix~\ref{proof:savrep-m-VI-bd}} for Lemma~\ref{lem:savrep-m-VI-bd}, except that the parameter involved is constant, i.e.\,$\gamma_k=\gamma$ for $k=0,1,...$. We shall directly start from~\eqref{VI-bd-last-m} and apply strong monotonicity instead of mere monotone {(see definitions for each term in~\eqref{d_k-def})}:
\begin{eqnarray}
&&\gamma\langle H(x)+\tilde\nabla {g}(y^k),x^{k+0.5}-x\rangle+\gamma\mu\|x-x^{k+0.5}\|^2\nonumber\\
&\le& \gamma\langle H(x^{k+0.5})+\tilde\nabla {g}(y^k),x^{k+0.5}-x\rangle\le-d_k(x)+e_{k1}(x)+e_{k2}(x),\nonumber
\end{eqnarray}
Taking conditional expectation {$\EE_{k_1}[\cdot]:=\EE_{\xi_k}[\cdot|x^k,w^k]$} on both sides gives:
\begin{eqnarray}
    &&\EE_{k_1}\left[\gamma\langle H(x)+\tilde\nabla {g}(y^k),x^{k+0.5}-x\rangle\right]+\EE_{k_1}\left[\gamma\mu\|x-x^{k+0.5}\|^2\right]\nonumber\\
    &\le& \EE_{k_1}[-d_k(x)+e_{k1}(x)+e_{k2}(x)]\le \EE_{k_1}\left[-d_k(x)+\frac{1}{2}\left(2\gamma^2L_h^2-p_1\right)\|x^{k+0.5}-w^k\|^2\right],\label{app-lem-3.2-1}
\end{eqnarray}
where the last inequality is due to the identity
\begin{eqnarray}
\EE_{k_1}[e_{k2}(x)]=\EE_{k_1}\left[\gamma\langle\hat H(x^{k+0.5})-H(x^{k+0.5}),x-x^{k+0.5}\rangle\right]= \gamma\langle H(x^{k+0.5})-H(x^{k+0.5}),x-x^{k+0.5}\rangle=0,\nonumber
\end{eqnarray}
{since $x^{k+0.5}$ is deterministic with respect to $\EE_{k_1}[\cdot]$ and $\EE_{k_1}\left[\hat H(x^{k+0.5})\right]=H(x^{k+0.5})$}, and
\begin{eqnarray}
    &&\EE_{k_1}[e_{k1}(x)]{=\EE_{k_1}\left[\frac{1}{2}\left(2\gamma^2\|{H}_{\xi_k}(x^{k+0.5})-{H}_{\xi_k}(w^k)\|^2-p_1\|x^{k+0.5}-w^k\|^2-\frac{1}{2}\|x^{k+1}-x^{k+0.5}\|^2\right)\right]}\nonumber\\
    &\le& \EE_{k_1}\left[\frac{1}{2}\left(2\gamma^2\|{H}_{\xi_k}(x^{k+0.5})-{H}_{\xi_k}(w^k)\|^2-p_1\|x^{k+0.5}-w^k\|^2\right)\right]\nonumber\\
    &\le& \EE_{k_1}\left[\frac{1}{2}\left(2\gamma^2L_h^2-p_1\right)\|x^{k+0.5}-w^k\|^2\right],\nonumber
\end{eqnarray}
{since $\EE_{k_1}\left[\|H_{\xi_k}(x^{k+0.5})-H_{\xi_k}(w^k)\|^2\right]\le L_h^2\|x^{k+0.5}-w^k\|^2$ (see also~\eqref{tower-lipschitz}). 



We obtain the next inequality by
rearranging the terms in~\eqref{app-lem-3.2-1}:
\begin{eqnarray}
&&\EE_{k_1}\left[\gamma\langle H(x)+\tilde\nabla {g}(y^k),x^{k+0.5}-x\rangle\right]\nonumber\\
&\le& -\EE_{k_1}[d_k(x)]-\EE_{k_1}\left[\frac{1}{2}(p_1-2\gamma^2L_h^2)\|x^{k+0.5}-w^k\|^2+\gamma{\mu_h}\|x^{k+0.5}-x\|^2\right].\label{app-lem-3.2-2}
\end{eqnarray}
Finally, applying Young's inequality to the term $\|x^{k+0.5}-x\|^2$ and the definition of $d_k(x)$ gives:
\begin{eqnarray}
&&-\EE_{k_1}[d_k(x)]-\EE_{k_1}\left[\frac{1}{2}(p_1-2\gamma^2L_h^2)\|x^{k+0.5}-w^k\|^2+\gamma{\mu_h}\|x^{k+0.5}-x\|^2\right]\nonumber\\
&\le& -\EE_{k_1}[d_k(x)]-\EE_{k_1}\left[\frac{1}{2}(p_1-2\gamma^2L_h^2)\|x^{k+0.5}-w^k\|^2+\frac{1}{2}\gamma{\mu_h}\|x^{k}-x\|^2-{}\gamma\mu_h\|x^{k+0.5}-x^k\|^2\right]\nonumber\\
&=& \frac{1}{2}\EE_{k_1}\left[(1-p_1-{}\gamma\mu_h)\|x^k-x\|^2+p_1\|w^k-x\|^2-\|x^{k+1}-x\|^2\right]\nonumber\\
&&-\frac{1}{2}\left(p_1-2\gamma^2L_h^2\right)\EE_{k_1}\left[\|x^{k+0.5}-w^k\|^2\right]-\frac{1}{2}\left(1-p_1-2\gamma\mu_h\right)\EE_{k_1}\left[\|x^{k+0.5}-x^k\|^2\right].\nonumber\label{strong-telescope}
\end{eqnarray}
Combining with~\eqref{app-lem-3.2-2} completes the proof. }

\subsection{Proof of Lemma \ref{lem:function-relation-1}}
\label{proof:function-relation-1}
The proof for Lemma~\ref{lem:function-relation-1} is again similar to the proof in {Appendix~\ref{proof:savrep-m-g-bd}} for Lemma~\ref{lem:savrep-m-g-bd} with $\alpha_k=\alpha$ and $\beta_k=\beta$ for all $k$. We may directly start from~\eqref{g-bd-last-m} and take conditional expectation {$\EE_{k_2}[\cdot]:=\EE_{\zeta_k}[\cdot|x^k,\bar w^k,v^k]$} on both sides {(refer to~\eqref{VR-grad-g} for the definition of $\tilde\nabla g(y^k)$)}:
\begin{eqnarray}
&&\EE_{k_2}\left[g(v^{k+1})-g(x)\right]\nonumber\\
&\le& \EE_{k_2}\left[(1-\alpha-\beta)\left(g(v^k)-g(x)\right)+\beta \left(g(\bar w^k)-g(x)\right)\right]\nonumber\\
&&+\EE_{k_2}\left[\alpha\langle\tilde\nabla g(y^k),x^{k+0.5}-x\rangle\right]+{\left(\frac{\alpha^2L_g}{2}+\frac{\alpha^2L_g}{\beta}
\right)}\EE_{k_2}\left[\|x^{k+0.5}-x^k\|^2\right]+\EE_{k_2}[e_{k3}(x)]\nonumber\\
&\le& \EE_{k_2}\left[(1-\alpha-\beta)\left(g(v^k)-g(x)\right)+\beta \left(g(\bar w^k)-g(x)\right)\right]\nonumber\\
&&+\EE_{k_2}\left[\alpha\langle\tilde\nabla g(y^k),x^{k+0.5}-x\rangle\right]+{\left(\frac{\alpha^2L_g}{2}+\frac{\alpha^2L_g}{2\beta}
\right)}\EE_{k_2}\left[\|x^{k+0.5}-x^k\|^2\right],\nonumber
\end{eqnarray}
{where we use the result $\EE_{k_2}[e_{k3}(x)]\le0$ in the last inequality. This completes the proof for Lemma~\ref{lem:function-relation-1}. 

To see why $\EE_{k_2}[e_{k3}(x)]\le0$, we apply the definition of $e_{k3}(x)$ in~\eqref{error-3-m} and the fact that $\EE_{k_2}\left[\tilde \nabla g(y^k)\right]=\nabla g(y^k)$ to obtain
\begin{eqnarray}
    \EE_{k_2}[e_{k3}(x)]&=& \EE_{k_2}\left[\langle \nabla g(y^k)-\tilde\nabla g(y^k),x^{k+0.5}-x^k\rangle\right]+\EE_{k_2}\left[\langle \nabla g(y^k)-\tilde\nabla g(y^k),x^{k}-x\rangle\right]\nonumber\\
    &&+\EE_{k_2}\left[ -\frac{\beta}{\alpha}\left(g(\bar w^k)-g(y^k)-\langle\nabla g(y^k),\bar w^k-y^k\rangle\right)-\frac{\alpha L_g}{\beta}\|x^{k+0.5}-x^k\|^2\right]\nonumber\\
    &=& \EE_{k_2}\left[\langle \nabla g(y^k)-\tilde\nabla g(y^k),x^{k+0.5}-x^k\rangle\right]\nonumber\\
    &&+\EE_{k_2}\left[ -\frac{\beta}{\alpha}\left(g(\bar w^k)-g(y^k)-\langle\nabla g(y^k),\bar w^k-y^k\rangle\right)-\frac{\alpha L_g}{\beta}\|x^{k+0.5}-x^k\|^2\right]\nonumber\\
    &\le& \EE_{k_2}\left[\frac{\beta}{2\alpha L_g}\|\nabla g( y^k)-\tilde\nabla g(y^k)\|^2+\frac{\alpha L_g}{2\beta}\|x^{k+0.5}-x^k\|^2\right]\nonumber\\
    &&+\EE_{k_2}\left[ -\frac{\beta}{\alpha}\left(g(\bar w^k)-g(y^k)-\langle\nabla g(y^k),\bar w^k-y^k\rangle\right)-\frac{\alpha L_g}{\beta}\|x^{k+0.5}-x^k\|^2\right]\nonumber\\
    &=& \EE_{k_2}\left[\frac{\beta}{2\alpha L_g}\|\nabla g( y^k)-\tilde\nabla g(y^k)\|^2-\frac{\beta}{\alpha}\left(g(\bar w^k)-g(y^k)-\langle\nabla g(y^k),\bar w^k-y^k\rangle\right)\right]\nonumber\\
    &&+\EE_{k_2}\left[-\frac{\alpha L_g}{2\beta}\|x^{k+0.5}-x^k\|^2\right].\label{app:proof-lem-3.3-1}
\end{eqnarray}
where in the inequality we apply Young's inequality. Finally, by~\eqref{g-gradient-est-bd}, the first two terms in~\eqref{app:proof-lem-3.3-1} combine to be non-positive, which proves the argument that $\EE_{k_2}[e_{k3}(x)]\le0$.
}

\subsection{Proof of Theorem \ref{thm:per-iter-conv-1}}
\label{proof:per-iter-conv-1}

{
By the definition of the function $Q(x';x)$ in~\eqref{Q-function} and the iterate $v^{k+1}$ in~\eqref{VR-VI-Opt-Strong-Update}, we first obtain the following identity:
\begin{eqnarray*}
    && \EE\left[Q(v^{k+1};x)\right] =\EE\left[\langle H(x),v^{k+1}-x\rangle+g(v^{k+1})-g(x)\right]\nonumber\\
    &=& \EE\left[(1-\alpha-\beta)\langle H(x),v^k-x\rangle+\alpha\langle H(x),x^{k+0.5}-x\rangle+\beta\langle H(x),\bar w^k-x\rangle\right]+\EE\left[g(v^{k+1})-g(x)\right]\nonumber.
\end{eqnarray*}
Noting the tower property $\EE\left[g(v^{k+1})-g(x)\right]=\EE\left[\EE_{k2}\left[g(v^{k+1})-g(x)\right]\right]$, we apply Lemma~\ref{lem:function-relation-1} to bound the last term in the above equation and obtain:
\begin{eqnarray*}
    &&\EE\left[(1-\alpha-\beta)\langle H(x),v^k-x\rangle+\alpha\langle H(x),x^{k+0.5}-x\rangle+\beta\langle H(x),\bar w^k-x\rangle\right]+\EE\left[g(v^{k+1})-g(x)\nonumber\right]\\
    &\le& \EE\left[(1-\alpha-\beta)\langle H(x),v^k-x\rangle+\alpha\langle H(x),x^{k+0.5}-x\rangle+\beta\langle H(x),\bar w^k-x\rangle\right]\nonumber\\
    &&+ \EE\left[(1-\alpha-\beta)\left(g(v^k)-g(x)\right)+\beta \left(g(\bar w^k)-g(x)\right)\right]\nonumber\\
    &&+\EE\left[\alpha\langle\tilde \nabla g(y^k),x^{k+0.5}-x\rangle\right]+{\left(\frac{\alpha^2L_g}{2}+\frac{\alpha^2L_g}{2\beta}\right)}\EE_{k_2}\left[\|x^{k+0.5}-x^k\|^2\right]\\
    &=& (1-\alpha-\beta)\EE\left[\langle H(x),v^k-x\rangle+g(v^k)-g(x)\right]
    +\beta\EE\left[\langle H(x),\bar w^k-x\rangle+g(\bar w^k)-g(x)\right]\nonumber\\
    &&+\alpha\EE\left[\langle H(x)+\tilde\nabla g(y^k),x^{k+0.5}-x\rangle+{\left(\frac{\alpha L_g}{2}+\frac{\alpha L_g}{2\beta}\right)}\|x^{k+0.5}-x^k\|^2\right],
\end{eqnarray*}
where in the equality the terms with same coefficients $(1-\alpha-\beta),\,\beta,\,\alpha$, are combined. In view of the definition of $Q(x';x)$ in~\eqref{Q-function}, the terms associated with the coefficient $(1-\alpha-\beta)$ can be replaced with $Q(v^k;x)$, and the terms associated with the coefficient $\beta$ can be replaced with $Q(\bar w^k;x)$. In addition, the term $\EE\left[\langle H(x)+\tilde\nabla g(y^k),x^{k+0.5}-x\rangle\right]$ can be bounded using the result from Lemma~\ref{lem:VI-relation-1} together with tower property. 
The next upper bound then follows:
\begin{eqnarray*}
    &&(1-\alpha-\beta)\EE\left[\langle H(x),v^k-x\rangle+g(v^k)-g(x)\right]+\beta\EE\left[\langle H(x),\bar w^k-x\rangle+g(\bar w^k)-g(x)\right]\nonumber\\
&&+\alpha\EE\left[\langle H(x)+\tilde\nabla g(y^k),x^{k+0.5}-x\rangle+{\left(\frac{\alpha L_g}{2}+\frac{\alpha L_g}{2\beta}\right)}\|x^{k+0.5}-x^k\|^2\right]\nonumber\\
&\le& (1-\alpha-\beta)\EE\left[Q(v^k;x)\right]+\beta\EE\left[ Q(\bar w^k;x)\right]\nonumber\\
&&+\frac{\alpha}{2\gamma}\EE\left[\left(1-p_1-{}\gamma\mu_h\right)\|x^k-x\|^2+p_1\|w^k-x\|^2-\|x^{k+1}-x\|^2\right]\nonumber\\
&&
-\frac{\alpha}{2\gamma}\left(p_1-2\gamma^2L_h^2\right)\EE\left[\|x^{k+0.5}-w^k\|^2\right]-\frac{\alpha}{2\gamma}\left(1-p_1-2\gamma\mu_h-\alpha\gamma L_g-\frac{\alpha \gamma L_g}{\beta}\right)\EE\left[\|x^{k+0.5}-x^k\|^2\right].
\end{eqnarray*}
Moving the term $\|x^{k+1}-x\|^2$ to LHS, we obtain
\begin{eqnarray}
    &&\EE\left[Q(v^{k+1};x)\right]+\frac{\alpha}{2\gamma}\EE\left[\|x^{k+1}-x\|^2\right]\nonumber\\
    &\le&(1-\alpha-\beta)\EE\left[Q(v^k;x)\right]+\beta\EE\left[ Q(\bar w^k;x)\right]\nonumber\\
&&+\frac{\alpha}{2\gamma}\EE\left[\left(1-p_1-{}\gamma\mu_h\right)\|x^k-x\|^2+p_1\|w^k-x\|^2\right]-\frac{\alpha}{2\gamma}\left(p_1-2\gamma^2L_h^2\right)\EE\left[\|x^{k+0.5}-w^k\|^2\right]\nonumber\\
&&
-\frac{\alpha}{2\gamma}\left(1-p_1-2\gamma\mu_h-\alpha\gamma L_g-\frac{\alpha \gamma L_g}{\beta}\right)\EE\left[\|x^{k+0.5}-x^k\|^2\right].\label{strong-potential-ineq}
\end{eqnarray}
To proceed, we shall construct a potential function to measure the convergence, while keeping the coefficients of $\|x^{k+0.5}-w^k\|^2$ and $\|x^{k+0.5}-x^k\|^2$ non-positive. To this end, we shall introduce the following bound while noting the expectation $\EE_{k_1+}[\cdot]:=\EE[\cdot|w^k,x^{k+1}]$ and the definition for $w^{k+1}$ in~\eqref{VR-VI-Opt-Strong-Update}:
\begin{eqnarray}
&&\EE\left[\|w^{k+1}-x\|^2\right]=\EE\left[\EE_{k_1+}\left[\|w^{k+1}-x\|^2\right]\right]=p_1\EE\left[\|x^{k+1}-x\|^2\right]+(1-p_1)\EE\left[\|w^k-x\|^2\right]\nonumber\\
&=& \EE\left[p_1\|x^{k+1}-x\|^2+(1-p_1-c)\|w^k-x\|^2+c\|w^k-x\|^2\right]\nonumber\\
&\le& \EE\left[p_1\|x^{k+1}-x\|^2+(1-p_1-c)\|w^k-x\|^2\right]\nonumber\\
&&+\EE\left[2c\|x^k-x\|^2+4c\|x^k-x^{k+0.5}\|^2+4c\|x^{k+0.5}-w^k\|^2\right],\nonumber
\end{eqnarray}
where $c>0$ is a parameter that needs to satisfy certain constraints to be determined later. Moving the term $\|x^{k+1}-x\|^2$ to LHS and combining the resulting inequality with~\eqref{strong-potential-ineq}, we have:
\begin{eqnarray}
&&\EE\left[Q(v^{k+1};x)\right]+\frac{\alpha}{2\gamma}\EE\left[(1-p_1)\|x^{k+1}-x\|^2+\|w^{k+1}-x\|^2\right]\nonumber\\
&\le& (1-\alpha-\beta)\EE\left[Q(v^k;x)\right]+\beta\EE\left[ Q(\bar w^k;x)\right]\nonumber\\
&&+\frac{\alpha}{2\gamma}\EE\left[(1-p_1-{}\gamma\mu_h+2c)\|x^k-x\|^2+(1-c)\|w^k-x\|^2\right]\nonumber\\
&&-\frac{\alpha}{2\gamma}\left(p_1-2\gamma^2L_h^2-4c\right)\EE\left[\|x^{k+0.5}-w^k\|^2\right]\nonumber\\
&&-\frac{\alpha}{2\gamma}\left(1-p_1-{2}\gamma\mu_h-\alpha\gamma L_g-\frac{\alpha \gamma L_g}{\beta}-4c\right)\EE\left[\|x^{k+0.5}-x^k\|^2\right].\label{Q-bd-c}
\end{eqnarray}
Taking $c=\frac{\gamma\mu_h}{3}$ together with the constraints given in~\eqref{const-2}, then the RHS of \eqref{Q-bd-c} can be 
bounded by:
\begin{eqnarray}
&&(1-\alpha-\beta)\EE\left[Q(v^k;x)\right]+\beta\EE\left[ Q(\bar w^k;x)\right]\nonumber\\
&&+\frac{\alpha}{2\gamma}\EE\left[\left(1-p_1-\frac{\gamma\mu_h}{3}\right)\|x^k-x\|^2+\left(1-\frac{\gamma\mu_h}{3}\right)\|w^k-x\|^2\right]
\nonumber\\
&\le& (1-\alpha-\beta)\EE\left[Q(v^k;x)\right]+\beta\EE\left[ Q(\bar w^k;x)\right]
\nonumber+\left(1-\frac{\gamma\mu_h}{3}\right)\frac{\alpha}{2\gamma}\EE\left[\left(1-p_1\right)\|x^k-x\|^2+\|w^k-x\|^2\right]\nonumber,
\end{eqnarray}
where the inequality is due to a simple observation that
\begin{eqnarray*}
    \left(1-p_1-\frac{\gamma\mu_h}{3}\right)\le \left(1-\frac{\gamma\mu_h}{3}\right)(1-p_1).
\end{eqnarray*}
Therefore we obtain
\begin{eqnarray}
    &&\EE\left[Q(v^{k+1};x)\right]+\frac{\alpha}{2\gamma}\EE\left[(1-p_1)\|x^{k+1}-x\|^2+\|w^{k+1}-x\|^2\right]\nonumber\\
    &\le& (1-\alpha-\beta)\EE\left[Q(v^k;x)\right]+\beta\EE\left[ Q(\bar w^k;x)\right]\nonumber\\
    &&+\left(1-\frac{\gamma\mu_h}{3}\right)\frac{\alpha}{2\gamma}\EE\left[\left(1-p_1\right)\|x^k-x\|^2+\|w^k-x\|^2\right]
\label{Q-bd-c2}
\end{eqnarray}
Finally, we note that
\begin{eqnarray}
\phi\EE\left[Q(\bar w^{k+1};x)\right]&=&\phi\EE\left[\EE_{k_2+}\left[Q(\bar w^{k+1};x)\right]\right]= \phi p_2\EE\left[Q(v^{k+1};x)\right]+\phi(1-p_2)\EE\left[Q(\bar w^k;x)\right],\nonumber
\end{eqnarray}
for any $\phi>0$, {where $\EE_{k_2+}[\cdot]:=\EE_{\zeta_k}[\cdot|\bar w^k,v^{k+1}]$ as defined in~\eqref{expectations-2}}. Adding the above identity to~\eqref{Q-bd-c2}, we obtain:
\begin{eqnarray}
&&\EE\left[(1-\phi p_2)Q(v^{k+1};x)+\phi Q(\bar w^{k+1};x)\right]+\frac{\alpha}{2\gamma}\EE\left[(1-p_1)\|x^{k+1}-x\|^2+\|w^{k+1}-x\|^2\right]\nonumber\\
&\le& \EE\left[(1-\alpha-\beta)Q(v^k;x)+(\beta+\phi(1-p_2)) Q(\bar w^k;x)\right]\nonumber\\
&&+\left(1-\frac{\gamma\mu_h}{3}\right)\frac{\alpha}{2\gamma}\EE\left[\left(1-p_1\right)\|x^k-x\|^2+\|w^k-x\|^2\right],\nonumber
\end{eqnarray}
which complete the proof by taking $x=x^*$.

}

\subsection{Proof of Corollary \ref{prop:grad-complexity-1}}
\label{proof:grad-complexity-1}
We first show that the parameters specified in Corollary \ref{prop:grad-complexity-1} satisfy the constraint~\eqref{const-2}. Indeed, the constraints will be reduced to the following:
\[
p_1-\frac{p_1}{8}-\frac{p_1}{3}\ge0,\quad \frac{15}{16}-\frac{11p_1}{6}\ge0,
\]
where the second inequality holds because $p_1=\frac{1}{m_1}$ and we assume trivially that $m_1\ge2$.

Next, inequality \eqref{potential-reduce-sto-error-2} in Theorem \ref{thm:per-iter-conv-1} implies that the reduction rate, denote as $C_{\scriptsize\mbox{red}1}$, is given by:
\begin{eqnarray}
C_{\scriptsize\mbox{red}1}:=\max\left\{\frac{1-\alpha-\beta}{1-\phi p_2},\frac{\beta+\phi(1-p_2)}{\phi},1-\frac{\gamma\mu_h}{3}\right\}.\nonumber
\end{eqnarray}
With the choice of $\phi$ and $p_2$, the following bounds hold:
\begin{eqnarray}
&\frac{1-\alpha-\beta}{1-\phi p_2}=\frac{1-2\alpha}{1-\alpha}\le 1-\alpha,&\nonumber\\
&\frac{\beta+\phi(1-p_2)}{\phi}=1-\frac{1}{m_2}+\frac{1}{(1+\alpha)m_2}=1-\frac{\alpha}{(1+\alpha)m_2}\le 1-\frac{\alpha}{2m_2}.&\nonumber
\end{eqnarray}
Therefore, the reduction rate is can be further expressed as
{
\begin{eqnarray}
&&\max\left\{\frac{1-\alpha-\beta}{1-\phi p_2},\frac{\beta+\phi(1-p_2)}{\phi},1-\frac{\gamma\mu_h}{3}\right\}\le\max\left\{1-\alpha,1-
\frac{\alpha}{2m_2},1-\frac{\gamma\mu_h}{3}\right\}\nonumber\\
&=&\max\left\{1-
\frac{\alpha}{2m_2},1-\frac{\gamma\mu_h}{3}\right\}\nonumber\\
&\overset{\eqref{SAVREP-parameter-choice}}{=}&{\max\left\{\max\left(1-\frac{\sqrt{\mu_h}}{24\sqrt{L_gm_2}},1-\frac{1}{24m_2}\right),\max\left(1-\frac{\mu_h}{12L_h\sqrt{m_1}},1-\frac{\sqrt{\mu_h}}{12\sqrt{L_gm_2}},1-\frac{1}{12m_1}\right)\right\}} \nonumber\\
&{=}& \max\left\{1-\frac{\sqrt{\mu_h}}{24\sqrt{L_gm_2}},1-\frac{1}{24m_2},1-\frac{\mu_h}{12L_h\sqrt{m_1}},1-\frac{\sqrt{\mu_h}}{12\sqrt{L_gm_2}},1-\frac{1}{12m_1}\right\}{:=C_{\scriptsize\mbox{red}2}}.\nonumber
\end{eqnarray}
We are now ready to derive the convergence rate, by denoting the potential function as
}
\[
W_k:=\left(\EE\left[(1-\phi p_2)Q(v^k;x^*)+\phi Q(\bar w^k;x^*)\right]+\frac{\alpha}{2\gamma}\EE\left[\left(1-p_1\right)\|x^k-x^*\|^2+\|w^k-x^*\|^2\right]\right),
\]
then we have
\begin{eqnarray}
&&W_{k+1}\nonumber
\le C_{\scriptsize\mbox{red}1}\cdot W_k
\le C_{\scriptsize\mbox{red}2}\cdot W_{k}
\le C_{\scriptsize\mbox{red}2}^{k+1}\cdot W_0
\nonumber,
\end{eqnarray}
where the first inequality is due to comparing the RHS of~\eqref{potential-reduce-sto-error-2} with $W_k$ and the definition of $C_{\scriptsize\mbox{red}1}$.
Note $v^0:=\bar w^0:=w^0=x^0$. Therefore,
\begin{eqnarray}
&&\EE\left[(1-p_1)\|x^{k+1}-x^*\|^2+\|w^{k+1}-x^*\|^2\right]\le \frac{2\gamma}{\alpha}\cdot W_{k+1}\le \frac{2\gamma}{\alpha}\cdot C_{\scriptsize\mbox{red}2}^{k+1}\cdot W_0\nonumber\\
&\le& C_{\scriptsize\mbox{red}2}^{k+1}\cdot\left(\frac{4\gamma}{\alpha}Q(x^0;x^*)+2\|x^0-x^*\|^2\right)
\le C_{\scriptsize\mbox{red}2}^{k+1}\cdot\left(\frac{\gamma}{\alpha\mu_h}\left\|H(x^0)+\nabla g(x^0)\right\|^2+2\|x^0-x^*\|^2\right)
\nonumber
\end{eqnarray}
By using the expression of $C_{\scriptsize\mbox{red}2}$, the above rate guarantees the iteration complexity {for obtaining $\EE\left[\|w^k-x^*\|^2\right]\le \epsilon$ is}
\begin{eqnarray}
\mathcal{O}\left(\frac{1}{1-C_{\scriptsize\mbox{red}2}}\log\frac{d_0}{\epsilon}\right)=\mathcal{O}
\left(\left(m_1+m_2+\sqrt{\frac{L_gm_2}{\mu_h}}+\frac{L_h\sqrt{m_1}}{\mu_h}\right)\ln\frac{d_0}{\epsilon}\right)
\label{grad-comp-1}
\end{eqnarray}
where the expected per iteration gradient cost is
$
\mathcal{O}(p_1m_1+p_2m_2+4)=\mathcal{O}(1).
$

\section{Parameter Choices in Numerical Experiments}
\label{app:parameters-numerical}

All the values listed in the following tables (the learing rates) are multiplied with the theoretical values given in the analysis for the corresponding methods (Corollary~\ref{cor-parameter-complexity} for SAVREP-m, Corollary~\ref{prop:grad-complexity-1} for SAVREP, and Theorem 2.5 in~\cite{alacaoglu2022stochastic} for EVR).
\subsection{Strongly Monotone Problem (perturbation $\mu=10^{-5}$)}

\begin{tabular}{c|c|c|c}
$L_g=1$ & \multicolumn{2}{c|}{SAVREP} & EVR \\
$m_2$ & $\alpha$ & $\gamma$ & $\tau$ \\
\hline
10491 & 200 & 20 & 80 \\
5245 & 40 & 20 & 40 \\
2622 & 40 & 20 & 40 \\
\end{tabular}
\begin{tabular}{c|c|c|c}
$L_g=3$ & \multicolumn{2}{c|}{SAVREP} & EVR \\
$m_2$ & $\alpha$ & $\gamma$ & $\tau$ \\
\hline
10491 & 400 & 40 & 100 \\
5245 & 1000 & 40 & 100 \\
2622 & 1000 & 40 & 100 \\
\end{tabular}
\begin{tabular}{c|c|c|c}
$L_g=10$ & \multicolumn{2}{c|}{SAVREP} & EVR \\
$m_2$ & $\alpha$ & $\gamma$ & $\tau$ \\
\hline
10491 & 1000 & 20 & 100 \\
5245 & 2000 & 20 & 100 \\
2622 & 2000 & 20 & 100 \\
\end{tabular}

\subsection{Strongly Monotone Problem (perturbation $\mu=10^{-10}$)}

{
\begin{tabular}{c|c|c|c}
$L_g=1$ & \multicolumn{2}{c|}{SAVREP} & EVR \\
$m_2$ & $\alpha$ & $\gamma$ & $\tau$ \\
\hline
10491 & 8e+04 & 20 & 80 \\
5245 & 8e+04 & 20 & 40 \\
2622 & 2e+04 & 20 & 40 \\
\end{tabular}
\begin{tabular}{c|c|c|c}
$L_g=3$ & \multicolumn{2}{c|}{SAVREP} & EVR \\
$m_2$ & $\alpha$ & $\gamma$ & $\tau$ \\
\hline
10491 & 2e+05 & 40 & 100 \\
5245 & 4e+05 & 40 & 100 \\
2622 & 4e+05 & 40 & 100 \\
\end{tabular}
\begin{tabular}{c|c|c|c}
$L_g=10$ & \multicolumn{2}{c|}{SAVREP} & EVR \\
$m_2$ & $\alpha$ & $\gamma$ & $\tau$ \\
\hline
10491 & 4e+05 & 20 & 100 \\
5245 & 4e+05 & 20 & 100 \\
2622 & 1e+06 & 20 & 100 \\
\end{tabular}
}

\subsection{Monotone Problem}
\begin{tabular}{c|c|c|c}
$L_g=1$ & \multicolumn{2}{c|}{SAVREP-m} & EVR \\
$m_2$ & $\alpha$ & $\gamma$ & $\tau$ \\
\hline
10491 & 0.1 & 40 & 40 \\
5245 & 0.1 & 40 & 40 \\
2622 & 0.1 & 40 & 40 \\
\end{tabular}
\begin{tabular}{c|c|c|c}
$L_g=3$ & \multicolumn{2}{c|}{SAVREP-m} & EVR \\
$m_2$ & $\alpha$ & $\gamma$ & $\tau$ \\
\hline
10491 & 0.1 & 40 & 40 \\
5245 & 0.1 & 40 & 40 \\
2622 & 0.1 & 40 & 40 \\
\end{tabular}
\begin{tabular}{c|c|c|c}
$L_g=10$ & \multicolumn{2}{c|}{SAVREP-m} & EVR \\
$m_2$ & $\alpha$ & $\gamma$ & $\tau$ \\
\hline
10491 & 0.1 & 40 & 40 \\
5245 & 0.1 & 40 & 40 \\
2622 & 0.1 & 40 & 40 \\
\end{tabular}


\end{appendices}

\end{document}